\documentclass[11pt,twoside, leqno]{article}

\usepackage{amssymb,amsmath,amsfonts,amsthm,color,mathrsfs}
\usepackage[Symbol]{upgreek}
\usepackage{txfonts}
\usepackage[nottoc,notlot,notlof]{tocbibind}
\usepackage[active]{srcltx}
\usepackage{hyperref}

\allowdisplaybreaks
\def\rr{{\mathbb R}}
\def\rn{{{\rr}^n}}

\def\zz{{\mathbb Z}}

\def\cn{{\mathbb N}}

\def\D{{\mathscr D}}

\def\ccc{{\mathscr C}}

\def\fz{\infty}

\def\supp{{\mathop\mathrm{\,supp\,}}}
\def\dist{{\mathop\mathrm {\,dist\,}}}
\def\diam{{\mathop\mathrm {\,diam\,}}}
\def\loc{{\mathop\mathrm{\,loc\,}}}
\def\Lip{{\mathop\mathrm{\,Lip\,}}}

\def\lz{\lambda}

\def\ez{\varepsilon}

\def\kz{\kappa}

\def\gz{{\gamma}}

\def\E{\mathscr{E}}

\def\hs{\hspace{0.3cm}}

\def\dint{\displaystyle\int}

\def\r{\right}
\def\lf{\left}

\def\la{\langle}
\def\ra{\rangle}

\newtheorem{thm}{Theorem}[section]
\newtheorem{lem}[thm]{Lemma}
\newtheorem{prop}[thm]{Proposition}
\newtheorem{cor}[thm]{Corollary}

\newtheorem{defn}[thm]{Definition}
\newtheorem{rem}[thm]{Remark}

\numberwithin{equation}{section}

\begin{document}
\arraycolsep=1pt
\author{Thierry Coulhon, Renjin Jiang, Pekka Koskela and Adam Sikora} \arraycolsep=1pt
\title{{\bf Gradient estimates for heat kernels and
harmonic functions}
 \footnotetext{\hspace{-0.35cm} 2010 {\it Mathematics
Subject Classification}. Primary 53C23; Secondary 31C25; 58J05; 58J35; 31C05; 31E05; 35K08; 43A85.
\endgraf{
{\it Key words and phrases: harmonic functions, heat kernels, Li-Yau estimates, Poisson equation, Poincar\'e inequality, Riesz transform}
\endgraf}}}

\date{\today}
\maketitle
\begin{center}
\begin{minipage}{11.5cm}\small
{\noindent{\bf Abstract}.  Let  $(X,d,\mu)$ be a doubling metric measure space endowed with a Dirichlet form $\E$ deriving from a ``carr\'e du
champ".  Assume that $(X,d,\mu,\E)$ supports a  scale-invariant $L^2$-Poincar\'e inequality. In  this article, we study the following properties of harmonic functions, heat kernels and Riesz transforms for $p\in (2,\infty]$:

(i)  $(G_p)$:   $L^p$-estimate  for the gradient of the associated heat semigroup;

(ii) $(RH_p)$: $L^p$-reverse H\"older inequality for the gradients of
harmonic functions;

(iii) $(R_p)$: $L^p$-boundedness of the Riesz transform ($p<\infty$);

(iv) $(GBE)$: a generalised Bakry-\'Emery condition.

We show that,  for $p\in (2,\infty)$, (i), (ii) (iii) are equivalent, while for $p=\infty$, (i), (ii), (iv) are equivalent.
 Moreover, some of these equivalences still hold under weaker conditions than the $L^2$-Poincar\'e inequality.

Our result  gives  a characterisation of Li-Yau's gradient estimate of heat kernels for $p=\infty$,
%and Yau type gradient estimate of harmonic functions%
 while for $p\in (2,\infty)$
it is a substantial improvement  as well as a generalisation of earlier results by Auscher-Coulhon-Duong-Hofmann \cite{acdh} and Auscher-Coulhon \cite{ac05}.
Applications to isoperimetric inequalities and Sobolev inequalities are given.
Our results apply to Riemannian and sub-Riemannian manifolds as well as to non-smooth spaces, and to degenerate elliptic/parabolic
equations in these settings. }\end{minipage}
\end{center}

\vspace{0.2cm}
\tableofcontents

\section{Introduction}

\subsection{Background and main results}

\hskip\parindent
On complete Riemannian manifolds and on more general metric measure spaces endowed with a Dirichlet form, Gaussian heat kernel upper and lower estimates have been well understood  since the works of Saloff-Coste \cite{sal2},  Grigor'yan \cite{gri92}, Sturm \cite{st1,st2,st3}, see also
\cite{bcs15,bcf14} and references therein. Together these estimates imply the doubling volume property and the  H\"older regularity of the heat kernel (see \cite{DP} for a new and direct proof of the latter fact). A fundamental and non-trivial consequence of the known characterisation of these estimates, in terms of the doubling volume property and a scale-invariant $L^2$-Poincar\'e inequality (see \cite{sal2, sal, SA}), is that they are stable under quasi-isometries.

By contrast, the matching upper estimate
$$|\nabla_x h_t(x,y)|\le \frac{C}{\sqrt{t}V(y,\sqrt t)}\exp\left\{-c\frac{ d^2(x,y)}{t}\right\} \leqno(GLY_{\infty})$$
(see Theorem \ref{main-har-heat-infty} below) of the gradient of the heat kernel is only known to hold in  very specific cases: on manifolds with non-negative Ricci curvature \cite{ly86}, on Lie groups with polynomial volume growth \cite{sal3}, and on covering manifolds with polynomial volume growth \cite{dng04a,dng04b}.  There have also been  many efforts to derive upper bounds of the gradient of the heat kernel by using
probabilistic methods including coupling and  derivation of Bismut type formulae,
but only for small time (i.e. essentially  local results) unless one assumes non-negativity of the curvature; see \cite{Cr91,Pi02,qia95,ST98,TW98} and references
therein.

No handy global characterisation exists for $(GLY_\infty)$ (see however \cite[Theorem 4.2]{CS1} in  the polynomial volume growth case). Note that no equivalent property can exhibit invariance under quasi-isometry: the example of divergence form operators with bounded measurable coefficients  shows that the Lipschitz character of the heat kernel  is not generic and not stable under quasi-isometry. However,  non-negative curvature is too restrictive a  sufficient condition, since it is very unstable under perturbations of any kind. Moreover, it is desirable to find a common reason that would explain why the property holds in the above three families of examples.
Such a condition was introduced  in \cite[Theorem 3.2]{ji15} and \cite[Theorem 3.1]{jky14}, where it is proven that a certain quantitative Lipschitz regularity of Cheeger-harmonic functions implies an upper estimate of the gradient of the heat kernel. We shall see in Section \ref{Ex}  that it is relatively easy to obtain such regularity of harmonic functions in the aforementioned settings.

In the present paper, we first give a converse to this implication, and follow with an $L^p$-version of this equivalence which can  be seen as an $L^\infty$ one.
An important motivation for the study of pointwise estimates of the gradient of the heat kernel is that they open up the way to the boundedness of Riesz transforms on $L^p$ for all $p\in (1,+\infty)$ (see \cite{acdh}). Further, it was discovered in \cite{acdh} that the weaker $L^p$-version of these estimates governs the boundedness of Riesz transforms on $L^p$ in an interval $(2,p_0)$, for $2<p_0<+\infty$. Details will be given below.

To summarise, we give characterisations of these pointwise and integrated estimates for the gradient of the heat kernel in terms of estimates for the gradients of harmonic functions. In other words, we eliminate time. This is a first step towards a geometric understanding of these estimates, and we expect this will enable one to treat new examples.

Let us now fix our setting.
Let $X$ be a locally compact, separable, metrisable, and connected
space equipped with a  Borel measure $\mu$ that is finite on compact sets and strictly positive on non-empty open sets.
Consider a strongly local and regular Dirichlet form $\E$ on $L^2(X,\mu)$ with dense domain $\D\subset L^2(X,\mu)$ (see \cite{fot} or \cite{GSC} for precise definitions).  According to Beurling and Deny \cite{bd59},  such a form can be written as
$$\E (f, g)= \int_X \,d\Gamma(f,g)$$
for all $f, g\in \D$, where $\Gamma$ is a measure-valued non-negative
 and symmetric bilinear form defined by the formula
$$\int_X\varphi\,d\Gamma(f,g):=\frac12\left[\E (f,\varphi g) + \E (g,\varphi f)-\E (fg,\varphi)\right]$$
for all $f,g\in \D\cap L^\infty(X,\mu)$ and $\varphi\in \D\cap \ccc_0(X).$
Here and in what follows, $\ccc(X)$ denotes the space of
continuous functions  on $X$ and $\ccc_0(X)$ the space of functions
in $\ccc(X)$ with compact support. We shall assume in addition that $\E$  admits a ``{\it carr\'e du champ}",
meaning that $\Gamma(f,g)$ is absolutely continuous with respect to $\mu$,
for all $f, g\in \D$. In what follows, for
simplicity of notation, we  will denote by  $\langle \nabla f,\nabla
g\rangle$  the energy density
$\frac{\,d\Gamma(f,g)}{\,d\mu}$, and by $|\nabla f|$  the square root
of $\frac{\,d\Gamma(f,f)}{\,d\mu}$.

Since $\E$ is strongly local, $\Gamma$ is local and satisfies the
Leibniz rule and the chain rule; see \cite{fot}. Therefore we can
define $\E(f, g)$ and $\Gamma(f, g)$ locally. Denote by $\D_\loc$
the collection of all $f\in L^2_\loc(X)$ for which, for each
relatively compact set $K\subset X$, there exists a function
$h\in\D$ such that $f = h$ almost everywhere on $K$.  The intrinsic
(pseudo-)distance on $X$ associated to $\E$ is then defined by
$$d(x, y):=\sup\left\{f(x)-f(y):\, f\in\D_\loc\cap\ccc(X),\, |\nabla f|\le 1\mbox{ a.e.}\right\}.$$
%Here $\Gamma(u,u)\le \mu$ means that $\Gamma (u, u)$ is absolutely
%continuous with respect to $\mu$ and $\frac{\,d\Gamma}{\,d\mu}(u,
%u)\le 1$ almost everywhere.

In this paper, we always assume  that $d$ is indeed a distance (meaning that for $x\not= y$, $0<d(x,y)<+\infty$)
 and that the topology induced by $d$ is equivalent to the original
topology on $X$.
%\comment{assume also that $d$ is a true distance {\color{red} It seems that, the assumption of $X$ being locally compact, connected, and
%the topology induced by $d$ is equivalent to the original
%topology on $X$ are sufficient to imply that  $d$ is a true distance, c.f. Sturm \cite{st1}. am i missing something?
%} OK, I'll have a look at Sturm}
Moreover, we assume that $(X,d)$ is a complete metric space.
Under this assumption,   $(X, d)$  is a geodesic length space; see  for instance
\cite{st1,ags3,GSC}.

To summarise the above situation, we shall say that $(X,d,\mu,\E)$  is a  Dirichlet metric measure  space endowed with a ``{\it carr\'e du champ}", in short a Dirichlet metric measure space.

The domain $\D$ endowed with the norm
$\sqrt{\|f\|^2_{2}+\E(f,f)}$ is a Hilbert space which we denote by $W^{1,2}(X,\mu,\E)$, in short $W^{1,2}(X)$. For an open set $U\subset X$,
the local Sobolev space $W_{\loc}^{1,2}(U)$
is defined to be the collection of all  functions $f$ such that for any compact set $K\subset U$
there exists $F\in\D$ satisfying $f=F$ a.e. on $K$.
For each $p\ge 2$, the Sobolev space  $W^{1,p}(U)$
is then defined as the collection of all  functions $f\in W_{\loc}^{1,2}(U)$ satisfying $f,\,|\nabla f|\in L^p(U)$;
see Appendix \ref{appendix-domain} for the existence of $|\nabla f|$.
The space $W^{1,p}_0(U)$
is defined to be the closure in $W^{1,p}(X)$ of functions in $W^{1,p}(X)$ with compact support in $U$.
Then each Lipschitz function with compact support in $U$
belongs to $W^{1,p}_0(U)$ for any $p\in [2,\infty]$; see Appendix \ref{appendix-domain}.
%For $1\le p<2$, other definitions are in order, but this case will not be needed in  what follows.

Corresponding to such a Dirichlet form $\E$, there exists an operator  denoted by $\mathcal{L} $,
acting on a dense domain $\mathscr{D}(\mathcal{L})$ in
$L^2(X,\mu)$, $\mathscr{D}(\mathcal{L})\subset W^{1,2}(X)$, such that for all
$f\in \mathscr{D}(\mathcal{L})$ and each
$g\in W^{1,2}(X)$,
$$\int_X f(x)\mathcal{L} g(x)\,d\mu(x)=\mathscr{E}(f,g).$$
The opposite $-\mathcal{L}$ of $\mathcal{L} $ is the infinitesimal generator of the heat semigroup $H_t=e^{-t\mathcal{L}}$, $t>0$.

Let $B(x,r)$ denote the open ball with center $x$ and radius $r$ with
respect to the distance $d$, and set $CB(x,r):=B(x,Cr).$
%\comment{I would not say so: Poincar\'e is a strong assumption
%and it is worth studying spaces where it does not hold (e.g.  the connected sum of two copies of the Euclidean space)
%
%{\color{red} How about putting the assumptions into local, i.e.,  a (local) doubling
%measure and an (local) $L^2$-Poincar\'e inequality. Since in the abstract setting, even H\"older
%regularity is not automatic and needs  (local) $L^2$-Poincar\'e inequality.
%And in cases of gluing spaces, the local $L^2$-Poincar\'e inequality holds.
%}
%
%This is not what I mean. On the contrary, I am interested in global problems (on manifolds), so that I don't mind assuming global Poincar\'e; my point is that depending on the kind of properties one expects, it may be natural or not to assume Poincar\'e ; I shall propose another formulation
% }
For simplicity we write $V(x,r):=\mu(B(x,r))$ for $x\in X$ and $r>0$.
We say that the metric measure space $(X,d,\mu)$ satisfies the volume doubling property if there exists a constant $C_D>1$
such that for
every $x\in X$ and all $r>0$,
$$V(x,2r)\le C_D V(x,r).\leqno(D)$$
If $(X,d,\mu,\E)$  is a Dirichlet  metric measure space endowed with a ``{\it carr\'e du champ}" and $(X,d,\mu)$ satisfies $(D)$, we say that $(X,d,\mu,\E)$  is a  doubling Dirichlet  metric measure space endowed with a ``{\it carr\'e du champ}", in short a doubling Dirichlet metric measure space.
It easily follows from $(D)$ that there exist  $Q>0$ and $C_Q>0$ depending only on $C_D$ such that for
every $x\in X$ and all $0<r<R$,
$$V(x,R)\le C_Q\lf(\frac{R}{r}\r)^QV(x,r).\leqno(D_Q)$$
Notice that $(X,d,\mu)$ satisfies $(D)$ if and only if it satisfies $(D_Q)$  for some $Q>0$.
Moreover, since $(D_Q)$  implies $(D_{\tilde Q})$ for each $\tilde Q>Q$, we shall assume without loss of generality that $Q\ge 2$.

One says that the  local Sobolev inequality $(LS_q)$, $q>2$, holds on $(X,d,\mu,\E)$  if for every ball $B=B(x,r)$ and each $f\in W^{1,2}_0(B)$,
$$\left(\fint_{B}|f|^q\,d\mu\right)^{2/q}\le C_{LS}\left(\fint_{B}|f|^2\,d\mu+\frac{r^2}{V(x,r)}\E(f,f)\right).\leqno(LS_q)$$
Under the volume doubling property $(D)$, it is known that   $(LS_q)$, for some $q>2$, is equivalent to
the assumption that the heat semigroup $H_t=e^{-t\mathcal{L}}$ has a  kernel $h_t$, called the heat kernel, which satisfies an  upper Gaussian bound
$$h_t(x,y)\le
  \frac{C}{V(x,{\sqrt t})}\exp\lf\{-c\frac{d^2(x,y)}{t}\r\}, \forall\,t>0, \mbox{for a.e. }x,y\in X,\leqno(UE)
$$
see \cite{bcs15}.

%
%it is known (see \cite[Corollary 2.2]{CS1}) that the pointwise estimate of the gradient of the heat kernel, which is the main
%object of study in this paper, implies the upper and lower Gaussian estimates of the heat kernel, hence by the above-mentioned characterisation a scale-invariant $L^2$-Poincar\'e inequality.
%This is therefore a natural assumption in the following.
We say that $(X,d,\mu,\E)$ supports a local $L^p$-Poincar\'e
inequality, $p\in [2,\infty)$, if for all $r_0>0$ there exists
$C_P(r_0)>0$
 such that, for all $0<r<r_0$ and for every ball $B=B(x,r)$ and each $f\in W^{1,p}(B)$,
$$
\fint_{B(x,r)}|f-f_B|\,d\mu\le
C_P(r_0)r\lf(\fint_{B(x,r)}|\nabla f|^p\,d\mu\r)^{1/p}.\leqno(P_{p,\loc})$$
Similarly, $(P_{\infty,\loc})$ requires for each $f\in W^{1,\infty}(B)$ that
$$
\fint_{B(x,r)}|f-f_B|\,d\mu\le
C_P(r_0)r\||\nabla f|\|_{L^\infty(B)}.\leqno(P_{\infty,\loc})$$
Further, if there exists a constant $C_P>0$ such that the above inequalities hold for every ball $B(x_0,r)$ and each $f\in W^{1,p}(B)$
with $C_P(r_0)$ replaced by $C_P$, then we say that  $(X,d,\mu)$ supports a scale-invariant $L^p$-Poincar\'e inequality, $(P_p)$, $p\in [2,\infty]$.
%In the above definitions of Poincar\'e type inequalities, we assume that $p\ge 2$ to avoid the technicalities of
%defining $W^{1,p}(U)$  for $1\le p<2$. Note however that we can extend this definition to all $p$ by assuming that for $p<2$  test functions $f$ range in $W^{1,2}$. This approach allows us to define $(P_p)$ and
%$(P_{p,\loc})$ for all $p\in [1,\infty]$, which will be useful in Section \ref{riesz}.

 Obviously,  inequalities $(P_p)$ as well as $(P_{p,\loc})$ are weaker and weaker as $p$ increases.
 Since $(X,d)$ is geodesic, our Poincar\'e inequalities $(P_p)$ and  $(P_{p,\loc})$
 have self-improving properties for $2\le p<\infty$ by \cite{kz08}, see Appendix \ref{appendix-poincare} for the precise
  statement in our setting. This fails, in general, for $(P_\infty)$ and $(P_{\infty,\loc})$, see \cite{dsw12,djs12}.
We also note that $(P_2)$ together with $(D)$ implies $(LS_q)$ for some $q\in (2,\infty]$ but the converse is not true; see \cite{bcs15,hak}.

By Sturm \cite{st1,st2,st3} (see Saloff-Coste \cite{sal,sal2} and
Grigor'yan \cite{gri92} for earlier results on Riemannian
manifolds), on a metric measure space $(X,d,\mu)$ endowed with a
strongly local and regular Dirichlet form $\E$,  $(D)$ together with $(P_2)$ are equivalent to
the requirement that the heat semigroup $H_t=e^{-t\mathcal{L}}$ has a  heat kernel $h_t$  that satisfies
the Li-Yau estimate
$$ \frac{C^{-1}}{V(x,{\sqrt t})}\exp\lf\{-\frac{d^2(x,y)}{ct}\r\}\le h_t(x,y)\le
  \frac{C}{V(x,{\sqrt t})}\exp\lf\{-c\frac{d^2(x,y)}{t}\r\},\leqno(LY)$$
for all    $t>0, \mbox{ a.e.}\,x,y\in X$.
This estimate was originally obtained in \cite{ly86}
on Riemannian manifolds with non-negative Ricci curvature. Moreover,
$(UE)$ is equivalent to a parabolic Harnack inequality
for solutions to the heat equation. The parabolic
Harnack inequality obviously implies an elliptic Harnack inequality, which had been
obtained earlier under doubling and Poincar\'e by Biroli and Mosco \cite{bm93,bm}.
Furthermore, Hebisch and Saloff-Coste \cite{hs01}  showed that  an  elliptic Harnack inequality
also implies a parabolic one if one has $(D)$ and $(UE)$ (see also \cite{bcf14}).
% Notice that under $(D)$, the heat kernel is conservative, i.e. $\int_Xh_t(x,y)\,d\mu(y)=1$ (see  \cite{st2}).

A consequence of the elliptic Harnack inequality is that harmonic functions are H\"older in space, and a consequence of the parabolic Harnack inequality is that the heat kernel is H\"older in time and space. It follows from the above that this is the case if
$(D)$ and $(P_2)$ hold.

However, in general, $(D)$ and $(P_2)$ are not  sufficient   for Lipschitz
regularity of harmonic functions or heat kernels. This phenomenon  already occurs in the case of
uniformly elliptic operators of divergence form with non-smooth coefficients in  Euclidean space, see for instance \cite{cp98,shz05}.
Even in a smooth setting, additional assumptions are required in order to ensure proper pointwise estimates for gradients of harmonic functions or heat kernels.

Yau's gradient estimate for positive harmonic functions (cf. Yau \cite{ya75}, Cheng-Yau \cite{chy})
states that on non-compact Riemannian manifolds with Ricci curvature bounded below by $-K$, $K\ge 0$, it holds that
$$\sup_{x\in B(x_0,r)}|\nabla \log u(x)|\le C\left(\frac{1}{r}+\sqrt{K}\right), \leqno (Y_{\infty})$$
for every ball  $B(x_0,r)$ and every positive  harmonic function $u$  on $B(x_0,2r)$.
Li-Yau's gradient estimate for heat kernels (c.f. Li and Yau \cite{ly86})
on Riemannian manifolds with non-negative Ricci curvature states that %on $RCD^\ast(0,N)$ spaces, i.e.
$$|\nabla_x h_t(x,y)|\le \frac{C}{\sqrt{t}V(y,\sqrt t)}\exp\left\{-c\frac{ d^2(x,y)}{t}\right\}, \forall\,t>0, \,x,y\in X. \leqno (GLY_{\infty})$$
These two gradient estimates are fundamental tools in geometric analysis and related fields, and there have been many efforts  afterwards to generalise them to different settings, see  for instance \cite{CS1, dm05,dng04a,dng04b,el10,gm14,hl2,ji15,lh06,qzz13,qia95,sal3,zhang06,zz11,zz16}.

Let us review some of these generalisations.
Saloff-Coste \cite{sal3} obtained $(GLY_{\infty})$ on Lie groups with polynomial growth.
Dungey \cite{dng04a,dng04b} obtained $(GLY_{\infty})$ on Riemannian covering manifolds with polynomial growth.
On Heisenberg type groups, Driver and Melcher \cite{dm05}  and Hu and Li \cite{hl2} obtained a Bakry-\'Emery type inequality,
which implies  $(GLY_{\infty})$. Zhang \cite{zhang06} obtained Yau's gradient estimate $(K=0)$ on Riemannian manifolds of non-negative
Ricci curvature  modulo a small
perturbation.  In recent years,  in a series of works \cite{bhllmy15,gm14,ji14,ji15,zz11,zz16},
Yau's gradient estimate for harmonic functions and Li-Yau's gradient estimate for heat kernels  (and  their local versions) have
been further generalised to  metric measure spaces and graphs satisfying suitable curvature assumptions;
we refer the reader to
\cite{ags1,ags3,cc1,cc2,cc3,eks13,honda1,lv09,stm4,stm5} for
recent developments of lower Ricci curvature bounds and related calculus on metric measure spaces.
Our aim in the present paper  is to  characterise heat kernel gradient bounds  without making any curvature assumptions. One can summarise our results  by saying that we reduce $(GLY_{\infty})$
to a condition that is easily seen to be equivalent to $(Y_{\infty})$ with $K=0$.

%Although these Lipschitz regularity results work in
%non-smooth settings, they  The dependence on lower Ricci curvature bounds make these
%results a bit restrictive, and they do {\it not} work for sub-Riemannian
%geometry {or for general degenerate elliptic/parabolic equations on
%these settings.}

%Although these results work well in specific settings, so far there is not a {\it unified} theory for these gradient estimates.
%In view of the characterisations of heat kernel bounds,
%established firstly on manifolds by Saloff-Coste \cite{sal2} and  Grigor?yan  \cite{gri92},  and extended to general metric spaces by Sturm \cite{st1,st2,st3}, one may wonder if it is possible to
%have characterisations of  gradient estimates for heat kernels and harmonic functions:}
%
%{\bf Question 1:} is there a {\it necessary and sufficient condition}, which works in Riemannian and sub-Riemannian geometry,
%%\comment{in terms of what ? How about something like {\color{red}} I'll try to find a formulation}
%for some quantitative Lipschitz regularity property of harmonic functions or heat kernels?
%
%More generally:
%
%{\bf Question 2:} is there a {\it necessary and sufficient condition} for the  $L^p$ ($p>2$) quantitative regularity of gradients of harmonic functions or heat kernels?

The conjunction of \cite[Theorem 1.4]{acdh} and \cite[Corollary 2.2]{CS1} shows that
 $(D)$ and   $(GLY_{\infty})$ yield  the boundedness of the Riesz transform on $L^p$:
 $$\||\nabla \mathcal{L}^{-1/2} f|\|_p\le C \|f\|_p,\ \forall f\in L^p(X,\mu)\leqno (R_p)$$
 for all $p\in (1,+\infty)$. On the other hand,  under $(D)$ and $(UE)$, $(GLY_{\infty})$ is known  to be equivalent to the boundedness of the gradient of the heat semigroup:
$$\||\nabla H_t|\|_{\infty\to\infty}\le \frac{C}{\sqrt t} \leqno (G_{\infty})$$
(see \cite[p.919]{acdh} and \cite[Theorem 4.11]{CS}).
However, there are examples such as conical manifolds (cf. \cite{lh99})
and uniformly elliptic operators (cf.  \cite{shz05} and  \cite{cp98}) where $(R_p)$ only holds for $p$ in a finite interval $(1,p_0)$, $2<p_0<\infty$. It was discovered  in \cite{acdh} that a natural substitute for $(GLY_{\infty})$ or $(G_{\infty})$ is
the $L^{p_0}$-boundedness of the gradient of the heat semigroup together with the estimate
$$\||\nabla H_t|\|_{p_0\to p_0}\le \frac{C}{\sqrt t}, \leqno (G_{p_0})$$
$2<p_0<\infty$, which by \cite{acdh} implies $(R_p)$
for all $1<p<p_0$ under $(D)$ and $(P_2)$. Above and in what follows,
 $\|\cdot\|_{p\to p}$ denotes the (sublinear or linear) operator norm from $L^p(X,\mu)$
to $L^p(X,\mu)$ for  $p\in [1,\infty]$.
Note conversely that  $(R_p)$ easily implies $(G_p)$  for any $p\in (1,\infty)$.
Note also that $(G_p)$ is equivalent to the validity of the estimate
$\||\nabla f|\|_p^2\lesssim \| f\|_p\|\mathcal{L} f\|_p$ for all  $f\in \mathscr{D}(\mathcal{L})$
with $f, \, \mathcal{L}f\in L^p(X,\mu)$ (see \cite[Prop. 3.6]{CS1}).

Observe that  $(G_2)$ always holds. Indeed, it follows from spectral theory that for each $f\in L^2(X,\mu)$,
$$\|\mathcal{L} H_t f\|_2\le \frac{C}{t}\|f\|_2,$$
and hence
$$\||\nabla H_tf|\|_2^2=\langle H_tf, \mathcal{L} H_tf\rangle\le \frac{C}{t}\|f\|_2^2,$$
i.e. $(G_2)$ holds (here $<\cdot,\cdot>$ denotes the bracket in $L^2$).

By interpolation with $(G_2)$, $(G_\infty)$ implies $(G_p)$ for all $p\in(2,\infty)$.
Finally, if $(X,d,\mu,\E)$ satisfies $(D)$ and $(UE)$, in particular if it satisfies $(D)$ and $(P_2)$, then it follows from the above results that $(GLY_\infty)$ implies $(G_p)$ for all $p\in (2,\infty)$.

Our main results below give a characterisation of $(G_p)$ for each $2<p\le \infty$ in terms of estimates for gradients of harmonic functions.
For $p=\infty$, this can be seen as a gradient version of the equivalence between elliptic and parabolic Harnack inequalities under $(D)$ and $(P_2)$, cf. \cite{hs01,bcf14}.
Before we state these results, let us recall some terminology.

Let $\Omega\subseteq X$ be a domain. For $g\in L^2(\Omega)$, a Sobolev function $f\in
W^{1,2}(\Omega)$ is called a solution to $\mathcal{L} f=g$ in $\Omega$ if
\begin{equation}
\int_\Omega\langle \nabla f,\nabla \varphi\rangle\,d\mu=\int_\Omega
g(x)\varphi(x)\,d\mu(x), \ \ \forall \varphi\in W_0^{1,2}(\Omega).
\end{equation}
If  $\mathcal{L} u=0$ in $\Omega$, then we say that $u$ is harmonic
in $\Omega$.

\begin{defn} Let $(X,d,\mu,\E)$ be a  Dirichlet metric measure space and let $p\in (2,\infty).$ We say that  the quantitative reverse
$L^p$-H\"older inequality for gradients of harmonic functions holds
if  there exists $C>0$ such that, for every ball $B$ with radius $r$ and every function $u$ that is harmonic in $2B$,
$$\lf(\fint_B|\nabla u|^{p}\,d\mu\r)^{1/p}\le \frac {C}r \fint_{2B}|u|\,d\mu.
\leqno (RH_{p})$$
Analogously,  $(RH_\infty)$ requires that
$$|||\nabla u|||_{L^{\infty}(B)}\le \frac {C}r \fint_{2B}|u|\,d\mu.$$
\end{defn}
Note that $(RH_{p})$ implies $(RH_{q})$ for $q<p$. In \cite{ji15,jky14}, $(RH_\infty)$ was used to prove isoperimetric inequalities and gradient upper  estimates for heat kernels.
We shall see in Lemma \ref{rh-yau} below that, under $(D)$ and $(P_2)$,
$(RH_{\infty})$ is equivalent to Yau's gradient estimate $(Y_{\infty})$ with $K=0$.
See \cite{chy,ji14,ya75,zz11} for more
about  $(Y_{\infty})$ and $(RH_\infty)$. Actually, a more natural formulation for the reverse $L^p$-H\"older inequality for gradients of harmonic functions is
$$\lf(\fint_B|\nabla u|^{p}\,d\mu\r)^{1/p}\le
C\lf(\fint_{2B}|\nabla u|^{2}\,d\mu\r)^{1/2}, \leqno(\widetilde{RH}_p)$$
if $u$ is harmonic on $2B$; see \cite{ac05,shz05}. In general, $(\widetilde{RH}_p)$ is stronger than $(RH_p)$. Indeed, as soon as $(D)$ and $(UE)$ hold, the Caccioppoli inequality (Lemma \ref{l2.5} below) together with Proposition \ref{l2.3} and a simple covering argument
gives the implication $(\widetilde{RH}_p)\Longrightarrow (RH_p)$. However, $({RH}_p)$ is equivalent to  $(\widetilde{RH}_p)$, if in addition one has $(P_2)$.

We shall see in Example 4 from Section \ref{rmms} that
there exist Riemannian manifolds where $(RH_p)$ and $(G_p)$ hold for some $p>2$, but $(\widetilde{RH}_p)$ does not hold. This is why we have to characterise $(G_p)$ in terms of $(RH_p)$ instead of $(\widetilde{RH}_p)$ in Theorem \ref{main-har-heat} below.

Our first main result gives a characterisation of pointwise estimates for the gradient of the heat kernel.
%quite  a complete description of the $L^\infty$ theory of  reverse
%H\"older inequality and can be stated in the following way.
%Our first main result  reads as
%follows.
\begin{thm}\label{main-har-heat-infty}
Let $(X,d,\mu,\E)$ be a non-compact
doubling Dirichlet metric measure space endowed with a ``{\it carr\'e du champ}".  Assume that $(X,d,\mu,\E)$ satisfies $(UE)$ and $(P_{\infty,\loc})$. Then the following statements are equivalent:

(i) $(RH_{\infty})$ holds.

(ii)  There exist $C,c>0$ such that  $$|\nabla_x h_t(x,y)|\le \frac{C}{\sqrt{t}V(y,\sqrt t)}\exp\left\{-c\frac{ d^2(x,y)}{t}\right\} \leqno (GLY_{\infty})$$
for all $t>0$ and a.e. $x,\,y\in X$.

(iii) The gradient of the heat semigroup $|\nabla H_t|$ is bounded on $L^{\infty}(X)$ for each $t>0$ with
$$\||\nabla H_t|\|_{\infty\to\infty}\le \frac{C}{\sqrt t}.\leqno (G_{\infty})$$

(iv) There exist $C,c>0$ such that
$$|\nabla H_tf(x)|^2\le CH_{ct}(|\nabla f|^2)(x) \leqno (GBE)$$
for every $f\in W^{1,2}(X)$, all $t>0$  and a.e. $x\in X$.
\end{thm}

The main novelty here is the implication $(GLY_\infty)\Longrightarrow (RH_\infty)$.
Indeed, $(RH_\infty)\Longrightarrow (GLY_\infty)$ follows from ideas in the proof
of \cite[Theorem 3.2]{ji15}.
Moreover it is easy to see that $(G_\infty)$ is equivalent to
\begin{equation}\label{infty}\sup_{t>0,\ x\in M\ }\sqrt t \int_M \, |\nabla_x\, h_t(x,y)|\, d\mu(y) <+\infty;
\end{equation}
therefore  $(GLY_\infty)\Longrightarrow (G_\infty)$ follows by integration using $(D)$ (see \cite[p.919]{acdh}).
Then  the reasoning in \cite[p.919]{acdh} and \cite[Theorem 4.11]{CS} yields the converse implication.
The equivalence $(GLY_\infty)\Longleftrightarrow (GBE)$
follows by a version of \cite[Lemma 3.3]{acdh} and \cite[Theorem 3.4]{bcf14}. In the sequel, we shall call  condition $(iv)$
a generalised  Bakry-\'Emery condition $(GBE)$.

%Finally, the implication $(i)\Longrightarrow (ii)$ was discovered  in \cite[Theorem 3.2]{ji15}.

Theorem \ref{main-har-heat-infty} admits a direct corollary.
\begin{cor}\label{main-cor-infity}
Let $(X,d,\mu,\E)$ be a non-compact doubling Dirichlet
 metric measure space endowed with a ``{\it carr\'e du champ}".  Assume that $(X,d,\mu,\E)$ satisfies $(P_2)$.
 Then the conditions $(RH_{\infty})$, $ (GLY_{\infty})$,
 $(G_\infty)$ and $(GBE)$ are mutually equivalent.
\end{cor}

In sufficiently smooth settings,   the assumption $(P_{\infty,\loc})$  is automatically satisfied and we obtain the following.
\begin{cor}\label{main-cor-manifold}
Let $(M,g)$ be a non-compact  Riemannian manifold.   Assume that the Dirichlet metric measure space
associated to the Laplace-Beltrami operator satisfies $(D)$ and $(UE)$.
 Then the conditions $(RH_{\infty})$, $ (GLY_{\infty})$,
 $(G_\infty)$ and $(GBE)$ are mutually equivalent.
\end{cor}

\begin{rem}\label{remark-infty}\rm
 Note that in Theorem \ref{main-har-heat-infty} and Corollary \ref{main-cor-manifold}, we did not require $(P_2)$ or $(LY)$. However, they follow as a consequence of $(UE)$ together with
 $(GLY_\infty)$ or $(RH_\infty)$; cf. \cite{CS1,bcf14}.
\end{rem}

Note that, when $C=c=1$, $(GBE)$ is the classical Bakry-\'Emery condition
$$|\nabla H_tf(x)|^2\le H_t(|\nabla f|^2)(x),\leqno (BE)$$
which, on manifolds, is known to be equivalent to  non-negativity of Ricci curvature; see  \cite{BGL} and also \cite{bak1,be2,Wa04}.
This equivalence has been further generalised to metric measure spaces with non-negative Ricci
curvature  ($RCD^\ast(0,N)$ spaces) in \cite{ags3,ams15,eks13}.

{ On Lie groups of polynomial growth, Saloff-Coste \cite{sal3} obtained
$(GLY_\infty)$ for the heat kernels; see also \cite{Al92}.}
%On the Heisenberg groups $\mathbb{H}(2n,m)$, according to a
%recent result by Hu and Li \cite{hl2} (see Driver and Melcher \cite{dm05} for earlier results),
% $|\nabla H_tf(x)|\le CH_{t}(|\nabla f|)(x),$
%and therefore,
%$$|\nabla H_tf(x)|^2\le CH_{t}(|\nabla f|^2)(x),$$
%since by the Jensen inequality $[H_{t}(|\nabla f|)]^2\le H_{t}(|\nabla f|^2)$;
%also see \cite{el10,lh06,lh10} for related results.
More generally, on sub-Riemannian manifolds satisfying Baudoin-Garofalo's
curvature-dimension inequality
$CD(\rho_1, \rho_2, \kappa, d)$ with $\rho_1 \ge 0$, $\rho_2 > 0$, $\kappa\ge 0$ and  $  d\ge  2$,
it is known that the gradient of the heat kernel satisfies the pointwise
inequality $(GLY_\infty)$ (cf. \cite[Theorem 4.2]{bg13}).
Therefore, by Theorem \ref{main-har-heat-infty}, we see that
$(RH_\infty)$, $(GBE)$ and
$(G_\infty)$ hold on the aforementioned spaces; see Section \ref{Ex} for more
examples.

As far as $L^p$-estimates for the gradient of the heat kernel are concerned, we have the following characterisation.

\begin{thm}\label{main-har-heat}
Let $(X,d,\mu,\E)$ be a non-compact doubling Dirichlet
 metric measure space endowed with a ``{\it carr\'e du champ}".   Assume that $(X,d,\mu,\E)$ satisfies $(UE)$ and $(P_{2,\loc})$.
 Let $p\in (2,\infty)$. Then the following statements are equivalent:

(i) $(RH_{p})$ holds.

(ii)  There exists $\gz>0$ such that $$\int_X|\nabla_x h_t(x,y)|^{p}\exp\left\{\gz d^2(x,y)/t\right\}\,d\mu(x)\le \frac{C}{t^{{p}/2}V(y,\sqrt t)^{{p}-1}} \leqno (GLY_{p})$$
for all $t>0$ and a.e. $y\in X$.

(iii) The gradient of the heat semigroup, $|\nabla H_t|$, is bounded on $L^{p}(X,\mu)$ for each $t>0$ with
$$\||\nabla H_t|\|_{p\to p}\le \frac{C}{\sqrt t}.\leqno (G_{p})$$
\end{thm}
Note that in Theorem \ref{main-har-heat} it is enough to assume   $(P_{2,\loc})$   instead of the much stronger global condition $(P_2)$: a Riemannian manifold that is the union of a compact part and a finite number of Euclidean ends  is a typical example satisfying
$(UE)$, $(P_{2,\loc})$, but {\em not} $(P_2)$; see \cite{cch06,cd99}. On the other hand, since $(P_2)$ implies   $(P_{2,\loc})$ and $(UE)$, we have the following corollary.

\begin{cor}\label{main-cor-finite}
Let $(X,d,\mu,\E)$ be a non-compact doubling Dirichlet
 metric measure space endowed with a ``{\it carr\'e du champ}".
 Assume that $(X,d,\mu,\E)$ satisfies  $(P_2)$.  Let $p\in (2,\infty)$.
 Then the conditions $(RH_{p})$, $ (GLY_{p})$,
and $(G_p)$  are mutually equivalent.
\end{cor}

\begin{rem}\label{remark-main}\rm
(i) Note that for $p=2$ all the conditions $(i)$, $(ii)$, $(iii)$ in Theorem  \ref{main-har-heat} hold. This is obvious for $(i)$ and we already observed that this is also the case for $(iii)$.
 Finally, $(ii)$  follows from    $\cite{G1}$, also  see $\cite[\rm Lemma\  2.3]{cd99}$.

(ii) Also note that the limit  case  $p=\infty$ of Theorem $\ref{main-har-heat}$ is nothing but Theorem $\ref{main-har-heat-infty}$.

(iii) Theorem \ref{main-har-heat} actually holds with $(P_{2,\loc})$ replaced by the weaker condition $(P_{p,\loc})$.
However, by \cite[Theorem 6.3]{bcf14} together with
\cite[Corollary 3.8]{bf15}, one has that, $(UE)$ and $(P_{p,\loc})$ together with $(RH_p)$ or $(G_p)$ imply $(P_{2,\loc})$.

(iv) Finally, note  that under $(D)$ and $(P_2)$,  there always exists $\varepsilon>0$ such that $(RH_p)$, hence $(GLY_p)$ and $(G_p)$, hold for $2<p<2+\varepsilon$;
see $\cite[{\rm\, Section\ } 2.1]{ac05}$ and Lemma \ref{lem-open} below.
\end{rem}

To the best of our knowledge, Theorem \ref{main-har-heat-infty} and
Theorem \ref{main-har-heat} are new even on Riemannian
manifolds. Since our assumptions are quite mild, our setting
includes Riemannian metric measure spaces,
sub-Riemannian manifolds, and degenerate elliptic/parabolic
equations in these settings; see the final section.

%\comment{These comments would probably be more relevant after Theorem \ref{riesz-main}}
%

Regarding the proofs, the main difficulties and novelties appear in the proof of
``$(RH_p)\Longrightarrow (GLY_p)$" for $p\in (2,\infty)$,  and in
``$(G_p)\Longrightarrow (RH_p)$" for $p\in (2,\infty]$.

A version of the implication $(RH_\infty)\Longrightarrow (GLY_\infty)$ was
proven in \cite[Theorem 3.2]{ji15} via quantitative regularity
estimates for solutions to the Poisson equation in \cite[Theorem 3.1]{jky14}, under $(D)$ and $(P_2)$.
In the present work, we replace the assumptions $(D)$ and $(P_2)$  there by  the slightly weaker combination
$(D)$, $(UE)$ and $(P_{\infty,\loc})$.
To prove $(RH_p)\Longrightarrow (GLY_p)$ for $p\in (2,\infty)$, we follow
some ideas from \cite{ji15,jky14}.
In particular, starting from $(RH_p)$, we first establish a quantitative
regularity estimate
for solutions to the Poisson equation; see Theorem \ref{harmonic-pnorm} below.
As we already said, harmonic functions are {\it not} necessarily locally
Lipschitz  in
a non-smooth setting.
Therefore, to establish Theorem \ref{harmonic-pnorm}, we can neither assume
nor use any Lipschitz regularity
of harmonic functions. In the classical setting, the fact that quantitative regularity for harmonic
functions implies quantitative regularity for solutions to the Poisson equation
is easy to prove  and there is even an analog for
certain non-linear equations, see \cite{kms}.

To overcome the difficulties attached to the non-smooth setting, we use the pointwise approach to
 Sobolev spaces on metric measure spaces by Haj\l asz \cite{ha96};
see \cite{hkst,sh}
and Section 2.1 below for more details. Then by using $(RH_p)$ in
the full strength,
a stopping-time argument and a bootstrap argument, we obtain
 pointwise control on Haj\l asz gradients of solutions to the
Poisson equation in terms of potentials; see  \eqref{arvio} below.
We expect that such
estimates are of independent interest.

Then, by viewing the heat kernel $h_t$ as a solution to the Poisson
equation $\mathcal{L} h_t=-\frac{\partial h_t}{\partial t}$, where a suitable estimate
for $\frac{\partial h_t}{\partial t}$ can be obtained from $h_t$ by using
Cauchy transforms (cf. Sturm \cite[Theorem 2.6]{st2}), we obtain $(GLY_p)$.

To prove $(G_p)\Longrightarrow (RH_p)$ for $p\in (2,\infty]$, we
first establish a reproducing formula for harmonic functions
by using the finite propagation speed property; see Lemma \ref{mvp}
below. Then, by using this reproducing formula, we follow recent developments on
the boundedness of spectral multipliers from  \cite{bcf14,bcs15} to show that
$(G_p)\Longrightarrow (RH_p)$ for all $p\in (2,\infty].$

%However, as we   want to prove that this implication is valid for the local version of our assumption (see Section \ref{Sec-local}),
%we will  provide also a more self-contained but more complicated argument, which is based on:
%
%(i) {spectral theory;} %the $L^2$-boundedness of spectral multipliers, which always holds for non-negative self-adjoint operators;
%
%(ii) Sobolev embedding, which has a local version under a local doubling condition together with  a local upper Gaussian estimate of the heat kernel;
%
%(iii) $L^p$-boundedness of $|\nabla H_t|$, i.e., $(G_p)$, which also has a local version $(G_{p,\loc})$; see Section \ref{Sec-local}.
%
%Starting from (i), by using (ii) and (iii), and an additional bootstrap argument
% in the case where $p$ is large, we can again show that
% $(G_p)\Longrightarrow (RH_p)$; see the second proof of  Theorem \ref{heat-har} below.

\subsection{Applications to  Riesz transforms}

\hskip\parindent Let us apply the previous results to $L^p$-boundedness of the  Riesz transform $|\nabla \mathcal{L}^{-1/2}|$.
% The negative square root of the non-negative self-adjoint operator $\mathcal{L}$ can be computed as
%$$\mathcal{L}^{-1/2}=\frac{\sqrt{\pi}}{2}\int_0^\infty e^{-s\mathcal{L}}\frac{\,ds}{\sqrt s}.$$
%We refer the reader to   \cite{acdh} for more details.
%The Riesz transform is the sublinear operator  $|\nabla \mathcal{L}^{-1/2}|$.
We say that $(R_p)$ holds  if this operator is continuous from $L^{p}(X,\mu)$ to itself. One easily checks that $(R_2)$ follows from the definitions and spectral theory.

For $p\in (1,2)$, it was proved by Coulhon and Duong in \cite{cd99} that $(R_p)$ holds as soon as $(D)$ and $(UE)$ hold
(however, this condition is not necessary, see \cite{CCFR}).
In particular, $(D)$ and $(P_2)$ are sufficient   conditions for  $(R_p)$ to be valid in this range.

For $p>2$, Auscher, Coulhon, Duong and Hofmann established in \cite{acdh}  a
characterisation of the boundedness of the Riesz transform on manifolds via
boundedness of the gradient of the heat semigroup.
Although the characterisation in \cite{acdh} was stated on manifolds,
its proof indeed works on metric measure spaces, as indicated in
\cite[p.6]{bcf14}. For further information  we refer to
\cite{ac05, bf15} and references therein.

Using \cite[Theorem 1.3]{acdh}, Theorem \ref{main-har-heat} above, and the open-ended character of condition $(RH_p)$ (Lemma \ref{lem-open} below), we obtain
the following result.

\begin{thm}\label{riesz-main}
Let $(X,d,\mu,\E)$ be a non-compact doubling Dirichlet metric measure space endowed with a ``{\it carr\'e du champ}". Assume that $(X,d,\mu,\E)$ satisfies  $(UE)$.
Let $p\in (2,\infty)$.   If $(P_p)$ holds, then $(RH_{p})$, $(G_{p})$ and $(R_p)$ are equivalent.
\end{thm}

%\comment{Of course the real question here is: using the ideas of Theorem \ref{main-har-heat}, can one prove directy the equivalence of $(RH_{p})$ with $(R_{p})$ (instead of $(G_{p})$)? {\color{blue}I don't know, looks rather difficult in the abstract setting}}

%\begin{rem}\rm
%If $\mu(X)<\infty$, then in the statement of the boundedness of $|\nabla (\mathcal{L})^{-1/2}|$, we need to replace
%$L^{p_0}(X,\mu)$ by $L^{p_0}_0(X)$, which is the collection of functions $f\in L^{p_0}(X,\mu)$ with vanishing mean values, i.e., $\int_Xf\,d\mu=0$. We will not distinguish this in the paper since the proofs are the same, cf. \cite{acdh}.
%\end{rem}

%In \cite{acdh}, it was proved that, for
%$p_0\in (2,\infty]$, the validity of $(R_p)$ for all $p\in (2,p_0)$, is equivalent to the validity of $(G_p)$
%for all $p\in (2,p_0)$;
%see Section \ref{riesz} below. In \cite{ac05}, it was proved that, on a
%Riemannian manifold $M$ satisfying $(D)$ and $(P_2)$, there exists $p_M\in (2,\infty]$ depending also on
%some further geometric properties (e.g. $L^q$-Poincar\'e inequality, $q<2$),
%such that for all $p_0\in (2,p_M)$, the boundedness of
%$|\nabla \mathcal{L}^{-1/2}|$ on
%$L^{p}(X)$ for all $p\in (2,p_0)$,
%is equivalent to the validity of $(RH_p)$ for all $p\in (2,p_0)$.
%
%Theorem \ref{riesz-main}  shows that these three
%conditions are equivalent point-to-point for any $p\in (2,\infty)$,
%which  improves the aforementioned results from
%\cite{ac05,acdh} even when $(X,d,\mu)$ is a Riemannian manifold.

Let us compare
Theorems \ref{main-har-heat} and  \ref{riesz-main} with  \cite[Theorem 2.1]{ac05}. The latter result states that on
a Riemannian manifold $M$ satisfying $(D)$  and $(P_2)$, there  exists $p_M\in (2,\infty]$
  such that for all $p_0\in (2,p_M)$, $(RH_p)$
for all $p\in (2,p_0)$ is equivalent to the validity of $(R_p)$ for
all $p\in (2,p_0)$. Now $(R_p)$ easily implies $(G_p)$ and conversely, according to  \cite[Theorem 2.1]{acdh}, under the same assumptions the validity of $(G_p)$  for
all $p\in (2,p_0)$ implies the validity of $(R_p)$ for
all $p\in (2,p_0)$.  Theorems \ref{main-har-heat} and  \ref{riesz-main}  contain three improvements with respect to  \cite[Theorem 2.1]{ac05}.
First, the proof of \cite[Theorem 2.1]{ac05} makes an essential use of $1$-forms on manifolds, and we do not know how to extend the
arguments from \cite{ac05} to our general setting. Second,  Theorems \ref{main-har-heat} and  \ref{riesz-main}  state a point-to-point equivalence
among $(RH_{p})$, $(G_{p})$ and $(R_p)$,
as opposed to a mere equivalence between  $(RH_{p})$ for $p\in (2,p_0)$ and $(G_{p})$ for $p\in (2,p_0)$. Finally, we obtain that $p_M=+\infty$.

According to Gehring's Lemma (cf. \cite{ge73,iw95}),
our  reverse H\"older inequality $(RH_{p})$ is an open-ended condition; see Lemma \ref{lem-open} below.
We then have the following corollary to Theorem~\ref{riesz-main}, which generalises the main result of \cite{ac05}
and a recent result \cite[Theorem 1.2]{bf15}.
\begin{cor}\label{cor-riesz-open}
Let $(X,d,\mu,\E)$ be a non-compact doubling Dirichlet  metric measure space endowed with a ``{\it carr\'e du champ}". Assume that $(X,d,\mu,\E)$ satisfies  $(UE)$.

(i) If $(P_2)$ holds, then the set of $p$'s such that $(R_p)$ holds is an interval $(1,p_0)$, with $p_0\in (2,\infty]$.

(ii) Let $p\in (2,\infty)$. If $(P_p)$, and one of the mutually  equivalent conditions $(RH_p)$, $(G_p)$, $(R_p)$, hold, then there exists $\ez>0$ such that all the mutually  equivalent conditions $(RH_{p+\ez})$, $(G_{p+\ez})$, $(R_{p+\ez})$ hold.
%there exists $\varepsilon>0$ such that $(R_p)$ holds for all $p\in
%(1,2+\varepsilon)$.
%Moreover, if $(R_{p_0})$ holds for some $p_0\in (2,\infty)$, then there exists $\varepsilon>0$
%such that $(R_p)$ is valid
%for all $p\in (1,p_0+\varepsilon)$.
\end{cor}

\begin{rem}\rm Even though we only assume $(P_p)$ in Theorem \ref{riesz-main} and (ii) of Corollary \ref{cor-riesz-open}, recent results from \cite[Theorem 6.3]{bcf14} and \cite[Corollary 3.8]{bf15} show that $(P_p)$ together with $(RH_p)$ or $(G_p)$ implies $(P_2)$.
\end{rem}

%The following result gives a new sufficient condition for the boundedness of the Riesz transform, where in particular, we do not need $(P_2)$.
%\begin{thm}\label{srh-riesz}
%Assume that the non-compact metric measure Dirichlet space $(X,d,\mu,\E)$ satisfies  $(D)$, $(UE)$ and $(P_{2,\loc})$.
%Let $p\in (2,\infty)$. If $(\widetilde{RH}_p)$ holds, then $(R_p)$ holds.
%\end{thm}

\subsection{Sobolev inequalities and isoperimetric inequality}\label{SI}
\hskip\parindent
Let $1\le p\le q\le +\infty$. In the above setting, we say that the Sobolev inequality $(S_{p,q})$ holds
if for every ball $B$, $B=B(x,r)$ and every Lipschitz function $f$, compactly supported in
 $B$, there exists $C$ such that
$$\left(\fint_B|f|^q\,d\mu\right)^{1/q}\le Cr\left(\fint_B|\nabla f|^p\,d\mu\right)^{1/p}. \leqno(S_{q,p})$$

Applying the methods from \cite{jk,jky14}, we show in Theorem \ref{p-sobolev} below that on a
non-compact metric measure space $(X,d,\mu, \E)$ endowed with a ``{\it carr\'e du champ}" and satisfying  $(D_Q)$ and $(UE)$, if
additionally
for some $p_0\in (2,\infty)$, $(P_{p_0,\loc})$ and one of the conditions $(RH_{p_0})$, $(R_{p_0})$,  $(G_{p_0})$
holds, then the Sobolev inequality $(S_{q,p_0'})$, where $p_0'<2$ is the
H\"older conjugate
of $p_0$, $q\ge p_0'$ satisfying $1/p_0'-1/q<1/Q$, is valid. An analogue for the
isoperimetric inequality ($p_0=\infty$)
will also be established in Theorem \ref{iso}.

\subsection{Plan of the paper}
\hskip\parindent   The paper is organized as follows. In Section 2, we recall and provide some basic notions and tools,
which include Sobolev spaces, harmonic functions, Poisson equations
and some functional calculi.

In Section 3, we provide a quantitative gradient estimate for solutions
to Poisson equations, assuming
$(RH_p)$.

In Section 4, we give the proofs of Theorems \ref{main-har-heat-infty} and  \ref{main-har-heat}, and their corollaries.

In Section \ref{riesz}, we prove Theorem \ref{riesz-main}, and in
Section 6, we study Sobolev inequalities and the isoperimetric inequality.

In Section 7, we exhibit several examples that our results can
be applied to.

In Appendix \ref{appendix}, we provide additional details for the techniques that are used in the proofs.

Throughout the work, { we always assume that our space $(X,d,\mu,\E)$ is a non-compact
doubling Dirichlet metric measure space. However, we wish to point out that our results and techniques
allow a localisation for local or compact settings. In order to keep the length of this paper reasonable,
we will present the localization in a forthcoming paper. }

We denote by $C,c$ positive constants which are independent of the
main parameters, but which may vary from line to line.
We
 use $\sim$ to mean that
two quantities are comparable.

\section{Preliminaries and auxiliary tools}

\subsection{Harmonic functions and Poisson equations}\label{harpo}
\hskip\parindent In this subsection, we recall some basic properties of harmonic
functions and of solutions to the Poisson equation. { Most of these properties have been deduced via  de Giorgi-Moser-Nash theory,
requiring only  doubling property and Sobolev inequality.}

Before we start our discussion, let us recall the notion of the reverse doubling,
 which for Riemannian manifolds  originates in \cite[Theorem 1.1]{gri92}. Assume that the Dirichlet metric measure space $(X,d,\mu,\E)$ satisfies $(D)$. If in addition $X$ is connected then it is known that the so-called reverse doubling
 estimate is valid, see e.g. \cite[Proposition 5.2]{GH}.
 The reverse doubling estimate ensures that, as $(X,d)$ is non-compact, there exist
 $0<Q'<Q$ and $c>0$ such that,
for all $r\geq s>0$ and $x,y\in M$ such that $d(x,y)<r+s$,
\begin{equation*}\label{rd}\tag{$RD$}
c\left(\frac{r}{s}\right)^{Q'}\leq \frac{V(y,r)}{V(x,s)},
\end{equation*}

Notice that $(UE)$ is equivalent to the local Sobolev inequality $(LS_q)$, for any $q\in (2,\infty]$ satisfying $\frac{q-2}{q}<\frac 2Q$, see \cite[Theorem 1.2.1]{bcs15}. It follows from  \cite[Section 3.4]{bcs15} that under $(RD)$
the local Sobolev estimate  $(LS_q)$ can be strengthened to the Sobolev inequality $(S_{q,2})$.

We continue with the Harnack inequality; see for instance \cite{bm93,bm,ji2}.

\begin{prop}\label{l2.3}
Assume that the doubling Dirichlet metric measure space $(X,d,\mu,\E)$ satisfies $(UE)$. Then there exists $C$ only depending on
$C_D$ and $C_{S}$ such that if $\mathcal{L} u=0$ in $B(x_0,r)$, then
$$\|u\|_{L^\fz(B(x_0,r/2))}\le C\fint_{B(x_0,r)}|u|\,d\mu.$$
\end{prop}

\begin{prop}\label{l2.3add}
Assume that the doubling Dirichlet metric measure space $(X,d,\mu,\E)$ satisfies  $(P_{2,\loc})$.
 For each $r_0>0$, there exists $C=C(C_D,C_P(r_0))$ such that if $u$ is a positive harmonic function on $B(x_0,r)$, $r<r_0$, then
$$\sup_{y\in B(x_0,r/2)}u(y)\le C\inf_{y\in B(x_0,r/2)}u(y).$$
Further if $(P_2)$ holds, then the above constant $C$ may be chosen independent of $r_0$.
\end{prop}	

Using the Harnack inequality, we obtain the following relation between Yau's gradient estimate and our condition $(RH_\infty)$. Since Lipschitz regularity of harmonic functions is the best one can hope for in non-smooth settings (cf.
\cite{ji2,zz11}), we have to use essential supremum instead of pointwise supremum in $(Y_\infty)$.
\begin{lem}\label{rh-yau}
Assume that the doubling Dirichlet metric measure space $(X,d,\mu,\E)$ satisfies  $(P_{2})$.
Then $(RH_{\infty})$ holds if and only if $(Y_\infty)$ holds with $K=0$.
\end{lem}
\begin{proof}
{$(RH_{\infty})\Longrightarrow (Y_{\infty})$ with $K=0$}:  Suppose that $u$ is positive harmonic function on $2B$, $B=B(x_0,r)$.
By  Propositions \ref{l2.3},  \ref{l2.3add} and a simple covering argument, we see that
$$|\nabla u(x)|\le |||\nabla u|||_{L^{\infty}(B)}\le \frac {C}r \fint_{\frac 32B}|u|\,d\mu\le  \frac {C}r \sup_{y\in \frac 32B}u(y) \le
\frac {C}r \inf_{y\in \frac 32B}u(y) \le \frac {C}r u(x)$$
for a.e. $x\in B$, i.e., $(Y_{\infty})$ holds with $K=0$.

{$(Y_{\infty})\ \mbox{with}\ K=0 \Longrightarrow (RH_{\infty})$}: Suppose that  $u$ is a harmonic function in $2B$.
Let $\delta= \|u\|_{L^\infty(\frac 32B)}$, then the strong maximum principle (cf. \cite{bm})
implies that either $u+\delta\equiv 0$ in $\frac 32B$ or $u+\delta>0$ there.
In the first case, $(RH_{\infty})$ holds obviously since $|\nabla u|\equiv 0$ in $B$.
In the second case, by  Proposition~\ref{l2.3} and the same covering argument,
 we obtain
$$\delta= \|u\|_{L^\infty(\frac 32B)}\le C\fint_{2B}|u|\,d\mu,$$
and hence by $(Y_{\infty})$ with $K=0$ and Proposition \ref{l2.3add},
$$|\nabla (u(x)+\delta)|\le \frac{C}{r}[u(x)+\delta]\le \frac{C}{r}\inf_{y\in B}[u(y)+\delta]\le
\frac{C}{r}\fint_{2B}|u+\delta|\,d\mu\le \frac{C}{r}\fint_{2B}|u|\,d\mu$$
for a.e. $x\in B$. That is, $(RH_{\infty})$ holds, which completes the proof.
\end{proof}

%By using Proposition \ref{l2.3}, it is easy to obtain the
%following H\"older continuity; see \cite{bm} for instance.
%\begin{lem}\label{l2.4}
%Assume that the metric measure space $(X,d,\mu,\E)$ satisfies $(D)$  and $(P_{2})$. Then if $\mathcal{L} u=0$ in $B(x_0,r)$,
%for all $x,y\in B(x_0,r/2)$,
%$$|u(x)-u(y)|\le C\frac{d^\gz(x,y)}{r^\gz}\fint_{B(x_0,r)}|u|\,d\mu,$$
%where $C=C(C_D,C_P)$ and $0<\gz<1$ only depends on $C_D,C_P$ as well.
%\end{lem}

In what follows we will  need the following Caccioppoli inequality; see \cite{bm,ji2}.
\begin{lem}\label{l2.5}
Let  $(X,d,\mu,\E)$ be a doubling Dirichlet metric measure space. Then if $\mathcal{L} f=g$ in
$B:=B(x_0,R)$, $g\in L^2(B)$, we have that for any $0<r<R$
\begin{equation*}
\int_{B(x_0,r)}|\nabla f|^2\,d\mu\le
\frac{C}{(R-r)^2}\int_{B(x_0,R)}|f|^2\,d\mu+C(R-r)^2\int_{B(x_0,R)}|g|^2\,d\mu,
\end{equation*}
where $C$ only depends on $C_D$.
\end{lem}
%\comment{This Lemma hesitates between a local and a global version...}
%\begin{proof} Choose a Lipschitz function $\varphi$ such that $\varphi=1$ on
%$B(x_0,r)$, $\supp \varphi\subset B(x_0,R)$ and $|\nabla \varphi|\le \frac{C}{R-r}$.
%Then $u\varphi^2\in W_0^{1,2}(B(x_0,R))$. By the Leibniz rule, we have
%\begin{eqnarray*}
%\int_{X} \nabla u(x)\cdot \nabla(u\varphi^2)(x)\,d\mu(x)&&=\int_{X} \varphi(x)^2|\nabla u(x)|^2
%+2u(x)\varphi(x)\nabla u(x)\cdot \nabla\varphi(x)\,d\mu(x)\\
%&&=-\int_{X} g(x)u(x)\varphi(x)^2\,d\mu(x).
%\end{eqnarray*}
%Applying the H\"older and  Young inequalities, we conclude that
%\begin{eqnarray*}
%&&\int_{X} \varphi(x)^2|\nabla u(x)|^2\,d\mu(x)\\
%&&\hs\le \int_{X} |g(x)u(x)|\varphi(x)^2\,d\mu(x)-\int_{X}2u(x)\varphi(x)\nabla u(x)\cdot \nabla\varphi(x)\,d\mu(x)\\
%&&\hs\le \int_{B(x_0,R)} |g(x)u(x)|\,d\mu(x)
%+\frac 12\|\varphi|\nabla u|\|_{L^2(X,\mu)}^2+8\|u|\nabla\varphi|\|_{L^2(X,\mu)}^2\\
%&&\hs\le C(R-r)^{2}\|g\|^2_{L^2(B(x_0,R))}+\frac{C}{(R-r)^2}\|u\|_{L^2(B(x_0,R))}^2
%+\frac 12\|\varphi|\nabla u|\|_{L^2(X,\mu)}^2,
%\end{eqnarray*}
%which completes the proof.
%\end{proof}

The following result was proved in \cite{bm} by using Sobolev inequalities.
\begin{lem}\label{exitence-poisson}
Assume that the Dirichlet metric measure space $(X,d,\mu,\E)$ satisfies $(D_Q)$, $Q\ge 2$, and that $(UE)$ holds. Let $p\in \left(\max\left\{Q/2,2\right\},\fz\right]$. Then for each $g\in L^p(B(x_0,r))$,   there is a unique
solution $f\in W^{1,2}_0(B(x_0,r))$ to $\mathcal{L} f=g$ in $B(x_0,r).$
Moreover
$$\|f\|_{L^\fz(B(x_0,r))}\le Cr^2V(x_0,r)^{-1/p}\|g\|_{L^p(B(x_0,r))},$$
where $C=C(C_D,C_S)$.
\end{lem}
\begin{proof} See \cite[Theorem 4.1]{bm} for the existence and the given
estimate; the uniqueness follows since the difference of any two solutions
is harmonic, with boundary value zero in the Sobolev sense.
\end{proof}

%\begin{defn}\label{Q}  When $Q\ge 2$, choose
%$p_Q\in \left(1,\frac{2Q}{Q+2}\right)$ arbitrarily  close to $\frac{2Q}{Q+2}$.
%\end{defn}
%%\begin{rem}\rm
%If one has $(P_2)$, then the Sobolev inequality $(S_{q,2})$ holds with $q=\frac{2Q}{Q-2}$ when $Q>2$.
%In this case, one can take in the above definition $p_Q=\frac{2Q}{Q+2}$ for $Q>2$.
%\end{rem}

\begin{lem}\label{l2.6}
Assume that the Dirichlet metric measure space $(X,d,\mu,\E)$ satisfies $(D_{Q})$, $Q\ge 2$, and that $(UE)$ holds.
Let $p\in\left(\frac{2Q}{Q+2},\infty\right]$. For each $B=B(x_0,r)$
and $g\in L^{p}(B)$,   there exists $f\in W^{1,2}_0(B)$ that
satisfies $\mathcal{L} f=g$ in $B$. Moreover, there exists a constant
$C$  %=C(C_D,C_S)
   such that
\begin{eqnarray*}
\fint_{B}|f|\,d\mu\le Cr \lf(\fint_{B}|\nabla f|^2\,d\mu\r)^{1/2}\le Cr^2 \lf(\fint_{B}|g|^{p}\,d\mu\r)^{1/p}.
\end{eqnarray*}
%$$\fint_{B}|f|\,d\mu\le Cr \lf(\fint_{B}|\nabla f|^2\,d\mu\r)^{1/2}\le Cr^2 \lf(\fint_{B}|g|^{p_Q}\,d\mu\r)^{1/p_Q}.$$
\end{lem}
\begin{proof}
Let us first prove the existence  of $f$.  Let $p'$ be the H\"older conjugate of $p$.
 Notice that $\frac{p'-2}{p'}<\frac2Q$  and $(UE)$ implies that $(S_{p',2})$ holds.
For each $k\in\cn$, let $g_k:=\chi_{\{|g|\le k\}}\,g$.
By Lemma \ref{exitence-poisson}, there exists a solution $f_k\in W^{1,2}_0(B)$
to $\mathcal{L} f_k=g_k$
 in $B$. For all $k,j\in\cn$,  $(S_{p',2})$  yields
\begin{eqnarray*}
\int_{B}|\nabla (f_k-f_j)|^2\,d\mu
&&= \int_{B} [g_k-g_j][f_k-f_j]\,d\mu\\
&&\le \|g_k-g_j\|_{L^{p}(B)}\|f_k-f_j\|_{L^{p'}(B)}\\
&& \le C\|g_k-g_j\|_{L^{p}(B)} \frac{r}{\mu(B)^{1/2-1/{p'}}} \||\nabla (f_k-f_j)|\|_{L^2(B)},
\end{eqnarray*}
and similarly
\begin{eqnarray}\label{2.8}
\int_{B}|\nabla f_k|^2\,d\mu&&\le C\|g_k\|_{L^{p}(B)} \frac{r}{\mu(B)^{1/2-1/{p'}}} \||\nabla f_k|\|_{L^2(B)}.
\end{eqnarray}
%where we used the fact that $q\in (p_Q,\max\{Q/2,2\}]$.
Therefore, $(f_k)_k$ is a Cauchy sequence in $W^{1,2}_0(B)$, and there exists
a limit $f\in W^{1,2}_0(B).$
By this and $(S_{p',2})$, we see that for each $\varphi\in W^{1,2}_0(B)$,
$$\int_{B}\langle \nabla f,\nabla \varphi\rangle\,d\mu=\lim_{k\to\infty}\int_{B}\langle \nabla f_k,\nabla \varphi\rangle\,d\mu=\lim_{k\to\infty}\int_{B}g_k\varphi\,d\mu=\int_{B}g\varphi\,d\mu,$$
where the last equality follows from the convergence $g_k\to g$ in $L^{p}(B)$
together with  $\varphi\in W^{1,2}_0(B)\subset L^{p'}(B)$.
This implies that $f$ is a solution to $\mathcal{L} f=g$ in $B$.

Notice that by \eqref{2.8},
$$\fint_{B}|f_k|^2\,d\mu\le Cr^2\fint_{B}|\nabla f_k|^2\,d\mu\le C\|g_k\|^2_{L^{p}(B)} \frac{r^4}{\mu(B)^{2/{p}}}
\le C\|g\|^2_{L^{p}(B)} \frac{r^4}{\mu(B)^{2/{p}}}.$$ By this,
letting $k\to \infty$, we conclude that
$$\fint_{B}|f|\,d\mu\le Cr \lf(\fint_{B}|\nabla f|^2\,d\mu\r)^{1/2}\le Cr^2 \lf(\fint_{B}|g|^{p}\,d\mu\r)^{1/{p}},$$
as desired.
\end{proof}

In our discussion we will also need the following result, see  \cite[Theorem 5.13]{bm}.
%\begin{lem}[\cite{bm}]\label{l2.7}
\begin{lem}\label{l2.7}
Assume that the Dirichlet metric measure space $(X,d,\mu,\E)$ satisfies $(D_{Q})$, $Q\ge 2$, and $(P_{2,\loc})$. Suppose that $f\in W^{1,2}(B)$, $B=B(x_0,r)$, $g\in L^p(B)$ and $\mathcal{L} f=g$ in $B$, where
$p\in (\frac Q2,\infty]\cap
(2,\infty]$. Then $f$ is locally H\"older continuous on $B$. %Moreover, for all $x,y\in B(x_0,r/2)$,
%$$|f(x)-f(y)|\le C\frac{d^\gz(x,y)}{r^\gz}\left(\fint_{B}|f|\,d\mu+r^2\left(\fint_{B}|g|^p\,d\mu\right)^{1/p}\right),$$
%where $C=C(p,C_D,C_P(r_0))$ and $0<\gz=\gz(p,C_D,C_P(r_0))<1$.
\end{lem}
%
%\comment{The last statement should be made more precise. {\color{blue} in which sense? the present form is enough for our purpose}. What is the estimate that comes with the H\"older continuity? it will be in the following form:
%for $x,y\in B(x_0,r/2)$,
%$$|u(x)-u(y)|\le C\frac{d^\gz(x,y)}{r^\gz}\left(\fint_{B}|u|\,d\mu+r^2\left(\fint_{B}|g|^p\,d\mu\right)^{1/p}\right),$$
%where $C=C(p,C_D(r_0),C_P(r_0))$ and $0<\gz<1$ only depends on these constants.
%{\color{red} Why not state this ? we only need the qualitative fact that the solution is continuous, and allows us to see that every point is a Lebesgue point }}
%\begin{proof}
%By \cite[Theorem 5.13]{bm}, we see that the above inequality holds
%for $p>\max\{\frac Q2,2\}$, where the measure in \cite{bm} is a doubling measure.
%Under the assumption that $\mu$ is $Q$-Ahlfors regular here,
%following the approach in %\cite[pp.200-201]{gt01},
%we can extend it to all $p\in (\frac Q2,\fz]\cap (1,\fz]$.
%We omit the details here.
%\end{proof}
%
%
%Combining Lemma \ref{l2.4}, similarly to the proof of
%%\cite[Lemma 2.1]{ji},
% we have the following estimate. We omit the details here.
%
\subsection{Functional calculus}
\hskip\parindent
Let $\mathbb{C}_+:=\{z\in\mathbb{C}:\,\mathrm{Re}\,z>0\}.$
Let $L$ be a non-negative, self-adjoint operator on $L^2(X,\mu)$, and let us denote its spectral decomposition
by $E_{L}(\lambda)$. Then,
for every bounded measurable function
$F:\, [0,\infty)\to \mathbb{C}$, one defines the operator $F(L):\, L^2(X,\mu)\to L^2(X,\mu)$ by the formula
\begin{equation}\label{spectrum-operator}
F(L):=\int_0^\infty F(\lambda)\,dE_{L}(\lambda).
\end{equation}
In the case of $F_z(\lambda):= e^{-z\lambda}$ for $z\in \mathbb{C}_+$, one sets
$e^{-zL}:= F_z(L)$ as given by \eqref{spectrum-operator}, which
gives a definition of the heat semigroup for complex time.
By spectral theory, the family $\{e^{-zL}\}_{z\in\mathbb{C}_+}$ satisfies
$$\|e^{-zL}\|_{2\to 2}\le 1$$
for all $z\in \mathbb{C}_+ $; cf. \cite[Chapter 2]{Da95}.
%For more details see
%\cite{hlmmy}.

%Since our space $(X,d,\mu,\E)$ is a  metric measure Dirichlet space
%endowed with a ``{\it carr\'e du champ}",  the following results in  this section hold without any further requirements.

%\comment{This section should be rewritten. The above sentence is useless if one gives proper assumptions: if $(X,d,\mu,\E)$ is a  metric measure Dirichlet space
%endowed with a ``{\it carr\'e du champ}", then DG holds (Sturm) and since the equivalent between DG and FSP holds for self-adjoint semigroups, we are in business}

\begin{defn}[Davies-Gaffney estimate]
We say that the  semigroup $\{e^{-tL}\}_{t>0}$
satisfies the Davies-Gaffney estimate if for all open
sets $E$ and $F$ in $X$,
$t\in (0,\fz)$ and $f\in L^2(E)$ with $\supp\,f\subset E$, it holds that
\begin{equation}\label{DG-estimate}
\|e^{-tL}f\|_{L^2(F)}\le  \exp\bigg\{-\frac{\dist(E,F)^2}{4t}\bigg\}
\|f\|_{L^2(E)},
\end{equation}
where and in what follows, $\dist(E,F):= \inf_{x\in E,\,y\in F}d(x,y)$.
\end{defn}

\begin{defn}[Finite propagation speed property]
We say that $L$  satisfies the
finite propagation speed property if for all $0 < t < d(E,F)$ and $E,F\subset X$, $f_1\in L^2(E)$ and $f_2\in L^2(F)$,
\begin{equation}\label{fin-speed}
\int_X\langle\cos(t\sqrt{L})f_1,f_2\rangle\, d\mu=0.
\end{equation}
\end{defn}

%If the operator $\cos(t\sqrt{\mathcal{L}})$ has an integral kernel
%$K_{\cos(t\sqrt{\mathcal{L}})}$, then \eqref{fin-speed} simply means that
%\begin{equation}\label{compact-supports-operator}
%\supp K_{\cos(t\sqrt{\mathcal{L}})}\subset \cd_t:= \lf\{(x,y)\in X\times
%X:\,d(x,y)\le t\r\}.
%\end{equation}

The following result was obtained by Sikora in \cite{Si}. The statement  can also be found in \cite[Proposition 3.4]{hlmmy} and
\cite[Theorem 3.4]{CS}.

\begin{prop}\label{finite-davies-gaffney}
The operator $L$ satisfies the finite propagation speed property \eqref{fin-speed} if and only if the semigroup $\{e^{-tL}\}_{t>0}$ satisfies  the Davies-Gaffney estimate
\eqref{DG-estimate}.
\end{prop}

By the Fourier inversion formula, whenever $F$ is an even bounded
Borel-function with $\hat{F}\in L^1(\rr)$, we can write $F(\sqrt{L})$
in terms of $\cos(t\sqrt{L})$ as
\begin{equation}\label{fourier-operator}
F(\sqrt{L})=\frac{1}{2\pi}\int_{-\infty}^\infty \hat{F}(t)\cos(t\sqrt{L})\,dt.
\end{equation}

%Combining this and \eqref{compact-supports-operator}, we see that
%\begin{equation}\label{com-fourier-operator}
%K_{F(\sqrt{\mathcal{L}})}(x,y)=\frac{1}{2\pi}\int_{|t|\ge
%c_0^{-1}d(x,y)} \hat{F}(t)K_{\cos(t\sqrt{\mathcal{L}})}(x,y)\,dt.
%\end{equation}

The following result follows from
\cite[Theorem 0.1]{st2} (see also \cite{hr03}) and \cite[Theorem 3.4]{CS}.
\begin{lem}\label{davies-gaffney} Let  $(X,d,\mu,\E)$ be a  Dirichlet metric measure space
endowed with a ``{\it carr\'e du champ}". Then
the associated heat semigroup $e^{-t\mathcal{L}}$
satisfies the Davies-Gaffney estimate.
\end{lem}

In what follows, $\mathcal{L}$ is as above. Let $\mathscr{S}(\rr)$ denote the collection of all Schwartz
functions on $\rr$. We need the following $L^2$-boundedness of spectral multipliers.

\begin{lem}\label{bdd-spectral}
Let $\Phi\in \mathscr{S}(\rr)$ be an
even function with $\Phi(0)=1$.  Then there exists $C>0$ such that
$$\sup_{r>0}\|(r^2\mathcal{L})^{-1}(1-\Phi(r\sqrt{\mathcal{L}}))\|_{2\to 2}\le C,$$
and, for each $k=0,1,2,\ldots$, there exists $C$ such that
$$\sup_{r>0}\|(r^2\mathcal{L})^{k}\Phi(r\sqrt{\mathcal{L}})\|_{2\to 2}\le C.$$
\end{lem}
\begin{proof}
We only give the proof of the first inequality; the second one follows
similarly.
Since { $\Phi'(0)=0$}, spectral theory (cf. \cite[Chapter 2]{Da95}) gives
\begin{eqnarray*}
\left\|(r^2\mathcal{L})^{-1}(1-\Phi(r\sqrt{\mathcal{L}}))\right\|_{2\to 2} \le
\sup_{\lambda}\left|\frac{1-\Phi(r{{\lambda}})}{r^2{\lambda^2}}\right|
=\sup_{\lambda}\left|\frac{1-\Phi({{\lambda}})}{{\lambda^2}}\right| < \infty.
\end{eqnarray*}
%we have
%$$1-\Phi(t\sqrt{\mathcal{L}})=-\int_0^t\sqrt{\mathcal{L}}\Phi'(s\sqrt{\mathcal{L}})\,ds,$$
%and therefore,
%$$(t^2\mathcal{L})^{-1}(1-\Phi(t\sqrt{\mathcal{L}}))=-\int_0^tt^{-2}\mathcal{L}^{-1/2}\Phi'(s\sqrt{\mathcal{L}})\,ds.$$
%Since $\Phi\in \mathscr{S}(\rr)$ is an even function on $\rr$, we have $\Phi'(0)=0$ and
%$$\sup_{x\in\rr} |x^{-1}\Phi'(x)|\le C(\Phi)<\infty.$$
%
%For all $f,g\in L^2(X,\mu)$,
%\begin{eqnarray*}\left|\langle(t^2\mathcal{L})^{-1}(1-\Phi(t\sqrt{\mathcal{L}}))f,g\rangle\right|
%&&=t^{-2}\left|\int_0^\infty\int_0^t\lambda^{-1/2}\Phi'(s\lambda^{1/2})\,ds\,d\langle E_{\mathcal{L}}(\lambda)f,g\rangle\right|\\
%&&\le t^{-2}\int_0^t s \left(\sup_{s\in\rr}|s^{-1}\lambda^{-1/2}\Phi'(s\lambda^{1/2})|\right)\|f\|_2\|g\|_2\,ds\\
%&&\le C(\Phi) \|f\|_2\|g\|_2,
%\end{eqnarray*}
%which implies that
%$$\sup_{t>0}\|(t^2\mathcal{L})^{-1}(1-\Phi(t\sqrt{\mathcal{L}}))\|_{2\to 2}\le C.$$
The proof is complete.
\end{proof}

\begin{lem}\label{lem-fin-spe}
Let $\Phi\in\mathscr{S}(\rr)$ be an
even function whose Fourier transform $\hat{\Phi}$ satisfies
$\supp \hat{\Phi}\subset [-1,1]$. Then for every $\kz\in\zz_+$ and
$t>0$, the operator
$(t^2\mathcal{L})^\kz \Phi(t\sqrt{\mathcal{L}})$ satisfies
\begin{equation}\label{fin-speed-general}
\int_X\langle (t^2\mathcal{L})^\kz \Phi(t\sqrt{\mathcal{L}}) f_1,f_2\rangle\,d\mu=0
\end{equation}
for all $0 < t < d(E,F)$
and $E,F\subset X$, $f_1\in L^2(E)$ and $f_2\in L^2(F)$.
\end{lem}
\begin{proof}
Let $\Phi_{\kz}(s):=s^{2\kz}\Phi(s)$. By noticing that
$\widehat{\Phi_{\kz}}(\lambda)=(-1)^{\kz}\frac{\,d^{2\kz}}{\,d\lambda^{2\kz}}\widehat{\Phi}(\lambda),$
the conclusion follows from  Lemma \ref{davies-gaffney}, Proposition \ref{finite-davies-gaffney}
and \eqref{fourier-operator}.%; cf. \cite[Lemma 3.5]{hlmmy}.
\end{proof}

\begin{lem}\label{lem-l2-appro}
Let $\Phi\in\mathscr{S}(\rr)$ be an
even function with $\Phi(0)=1$. Then, for each $f\in L^2(X,\mu)$, it
holds that
$$\lim_{t\to 0^+}\left\|f-\Phi(t\sqrt{\mathcal{L}})f\right\|_2=0.$$
\end{lem}
\begin{proof}
The  domain $\mathscr{D}(\mathcal{L})$ is dense in $L^2(X,\mu)$ and hence  it is enough to prove
Lemma \ref{lem-l2-appro} for $f \in \mathscr{D}(\mathcal{L})$. Then
\begin{eqnarray*}
\left\|f-\Phi(t\sqrt{\mathcal{L}})f\right\|_2 \le C t^2 \|\mathcal{L}f\|_2
\left\|(t^2\mathcal{L})^{-1}(1-\Phi(t\sqrt{\mathcal{L}}))\right\|_{2\to 2}
\end{eqnarray*}
and the lemma follows from Lemma \ref{bdd-spectral}.
%For all $f,g\in L^2(X,\mu)$, and each $t>0$, it holds that
%\begin{eqnarray*}
%\left|\langle
%f-\Phi(t\sqrt{\mathcal{L}})f,g\rangle\right|&&=\left|\int_0^\infty
%\left(1-\Phi(t\lambda^{1/2})\right)\,d\langle
%E_{\mathcal{L}}(\lambda)f,g\rangle\right|\\
%&&\le \|1-\Phi\|_{L^\infty(\rr)} \|f\|_2\|g\|_2.
%\end{eqnarray*}
%Letting $t\to 0^+$, applying the dominated convergence theorem and
%the fact $\Phi(0)=1$, we conclude that
%\begin{eqnarray*}
%\lim_{t\to 0^+}\left|\langle
%f-\Phi(t\sqrt{\mathcal{L}})f,g\rangle\right|&&=\left|\int_0^\infty
%\lim_{t\to 0^+}\left(1-\Phi(t\lambda^{1/2})\right)\,d\langle
%E_{\mathcal{L}}(\lambda)f,g\rangle\right|=0,
%\end{eqnarray*}
%as desired.
\end{proof}

\section{Regularity of solutions to the Poisson equation}

\hskip\parindent
In this section, we show that suitable regularity of harmonic functions
implies a gradient estimate for solutions to the Poisson equation
$\mathcal{L} f=g$.

The following result was established in \cite[Proposition 3.1]{jky14} under the stronger
assumption of both $(D_Q)$  and $(P_2)$; we adapt the proof below to our our setting. Given $a>1$ and $r>0,$  let $[\log_ar]$ be
the largest integer smaller than $\log_ar$.

\begin{prop}\label{pointwise-bound}
Assume that the Dirichlet metric measure space $(X,d,\mu,\E)$ satisfies $(D_Q)$, $Q\ge 2$, and that $(UE)$ holds.
 Suppose that $\mathcal{L} f=g$ in $2B$,
$B=B(x_0,r)$, with $g\in L^\fz(2B)$. Then, for every $p>\frac{2Q}{Q+2}$, there exists $C>0$ such
that for almost every $x\in B$,
$$|f(x)|\le C\lf\{\fint_{2B}|f|\,d\mu+G_1(x)\r\},$$
where
\begin{equation}\label{potential-g1}
G_1(x):=\sum_{j\le [\log_2r]}
2^{2j}\lf(\fint_{B(x,2^j)}|g|^{p}\,d\mu\r)^{1/p}.
\end{equation}
\end{prop}
\begin{proof}  Let $k_0=[\log_2r]$ and $x\in B$. By Lemma \ref{l2.6}, for each $j\le k_0$, there exists $f_j\in W^{1,2}_0(B(x,2^j))$
such that $\mathcal{L}f_j=g$ in $B(x,2^j)$, and
$$\fint_{B(x,2^{j-2})}|f_{j}(y)|\,d\mu(y)\le C\fint_{B(x,2^j)}|f_{j}(y)|\,d\mu(y)
\le C2^{2j}\lf(\fint_{B(x,2^j)}|g|^{p}\,d\mu\r)^{1/p}.$$
Moreover, for each $k\le k_0-1,$ as $\mathcal{L}(f_{k+1}-f_k)=0$ in $B(x,2^k)$ (notice that $B(x,2^k)\subset 2B$),
by Proposition \ref{l2.3}, we have
$$\|f_{k+1}-f_k\|_{L^\fz(B(x,2^{k-2}))}\le C\fint_{B(x,2^{k-1})}|f_{k+1}-f_k|\,d\mu\le C\fint_{B(x,2^k)}|f_{k+1}-f_k|\,d\mu.$$
Thus, from the above two inequalities, for almost every $x\in B$,
we deduce that
\begin{eqnarray*}
|f(x)|&&=\lim_{j\to-\fz}\fint_{B(x,2^{j-2})}|f(y)|\,d\mu(y)\\
&&\le \limsup_{j\to-\fz}\lf\{\fint_{B(x,2^{j-2})}|f_{j}(y)|\,d\mu(y)
+\sum_{k=j}^{k_0-1}\fint_{B(x,2^{j-2})}|f_{k+1}(y)-f_{k}(y)|\,d\mu(y)\r\}\\
&&\hs+\limsup_{j\to-\fz}\fint_{B(x,2^{j-2})}|f_{k_0}(y)-f(y)|\,d\mu(y)\\
&&\le \limsup_{j\to-\fz}\lf\{ C2^{2j}\lf(\fint_{B(x,2^{j})}|g|^{p}\,d\mu\r)^{1/{p}}
+\sum_{k=j}^{k_0-1}\|f_{k+1}-f_{k}\|_{L^\fz{(B(x,2^{j-2}))}}+\|f_{k_0}-f\|_{L^\fz(B(x,2^{j-2}))}\r\}\\
&&\le \lim_{j\to-\fz}C\sum_{k=j}^{k_0-1}\fint_{B(x,2^{k})}|f_{k+1}(y)-f_{k}(y)|\,d\mu(y)
+C\fint_{B(x,2^{k_0})}|f_{k_0}-f|\,d\mu\\
&&\le C\sum_{k=-\fz}^{k_0-1}\fint_{B(x,2^{k})}|f_{k+1}(y)|\,d\mu(y)
+C\fint_{B(x,2^{k_0})}|f_{k_0}|\,d\mu+C\fint_{B(x,2^{k_0})}|f|\,d\mu\\
&&\le C\sum_{k=-\fz}^{k_0} 2^{2k}\lf(\fint_{B(x,2^{k})}|g|^{p}\,d\mu\r)^{1/{p}}
+C\fint_{2B}|f|\,d\mu.
\end{eqnarray*}
Above, in the third inequality, we used the fact that
$$\limsup_{j\to-\fz}2^{2j}\lf(\fint_{B(x,2^{j})}|g|^{p}\,d\mu\r)^{1/{p}}
\le \limsup_{j\to-\fz}2^{2j}\|g\|_{L^\fz(2B)}=0,$$
and in the last inequality, we used the doubling condition to conclude that
$$\fint_{B(x,2^{k_0})}|f|\,d\mu\le C\fint_{2B}|f|\,d\mu.$$
The proof is complete.
\end{proof}
%Above, in case $B(x,2^j)$ is not entirely contained in $2B,$  the integral
%of $|g|$ over $B(x,2^j)$ refers to the integral of the zero extension of $g$
%to the exterior of $2B.$
%Above and in what follows, we fix $p_Q$ as in Definition $\ref{Q}$.
\subsection{Harmonic functions satisfying condition $(RH_\infty)$}
\hskip\parindent
The next statement deals with the case when harmonic functions satisfy condition $(RH_\infty)$. The proof
of the following theorem is similar to that of \cite[Theorem 3.1]{jky14}.

\begin{thm}\label{infinite-har}
Assume that the Dirichlet metric measure space $(X,d,\mu,\E)$ satisfies $(D_Q)$, $Q\ge 2$, and that
$(UE)$ and $(P_{\infty,\loc})$ hold.
Assume that $(RH_\infty)$ holds. Then if $\mathcal{L} f=g$ in $2B$, $B=B(x_0,r)$,
$g\in L^\infty(2B)$, and $p>\frac{2Q}{Q+2}$, there exists
$C=C(C_D,C_{LS},C_P(1),p)>0$ such that, for almost every $x\in B$,
$$|\nabla f(x)|\le C\lf\{\frac{1}{r}\fint_{2B}|f|\,d\mu
+G_2(x)\r\},$$
 where
\begin{equation}G_2(x):= \sum_{j\le [\log_2
r]}2^j\lf(\fint_{B(x,2^j)}|g|^{p}\,d\mu\r)^{1/p}.\label{potential-g2}
\end{equation}
\end{thm}

In order to prove Theorem \ref{infinite-har}, we need the following Lipschitz estimate, which follows from $(D)$, $(P_{\infty,\loc})$ and $(RH_\infty)$.
Its proof, which uses a telescopic estimate, will be omitted; see for instance \cite{sh}.
\begin{lem}\label{lip-har-add}
Assume that the doubling Dirichlet metric measure space $(X,d,\mu,\E)$ satisfies  $(UE)$ and $(P_{\infty,\loc})$.
Assume that $(RH_\infty)$ holds.
 If $\mathcal{L}f=0$ in $2B$, $B=B(x_0,r)$, then for almost all $x,y\in B(x_0,r)$ with $d(x,y)<1/2$,
 it holds that
 $$|f(x)-f(y)|\le C\frac{d(x,y)}{r}\fint_{2B} |f|\,d\mu,$$
 where $C=C(C_D,C_P(1))$.
\end{lem}

\begin{proof}[Proof of Theorem \ref{infinite-har}]
Set $k_0:=[ \log_2 r]$. Let $x,y\in B$ be Lebesgue points of $f$. Note that $G_1(x) \le Cr G_2(x)$
for all $x\in B$. Hence
if $d(x,y)\ge r/16$, then by Proposition \ref{pointwise-bound}, we have
\begin{eqnarray}\label{lip-est-add}
\quad \quad |f(x)-f(y)|&&\le C\fint_{2B}|f|\,d\mu+CG_1(x)+CG_1(y)\le Cd(x,y)\lf\{\frac {1}r\fint_{2B}|f|\,d\mu+G_2(x)+G_2(y)\r\}.
\end{eqnarray}

Now assume that $d(x,y)<r/16$ and $d(x,y)<\frac 12$.  Choose $k_1\in\zz$ such that $2^{k_1-2}\le d(x,y)<2^{k_1-1}$.
As in the proof of Proposition \ref{pointwise-bound}, for each $j\in \zz$ and $j\le k_0$, pick $f_j\in W^{1,2}_0(B(x,2^j))$
with $\mathcal{L}f_j=g$ in $B(x,2^j)$.
By the choice of $k_1$, we see that for each $z\in B(y,2^{k_1-1})$,
$$d(x,z)\le d(x,y)+d(y,z)< 2^{k_1-1}+2^{k_1-1}\le 2^{k_1},$$
which further implies that $B(y,2^{k_1-1})\subset B(x,2^{k})$ for each $k\ge k_1$.
Hence, for each  $k\ge k_1$, the value $f_k(y)$ is well defined, and we have
\begin{eqnarray*}
|f(x)-f(y)|&&\le |f(x)-f_{k_0}(x)-[f(y)-f_{k_0}(y)]|\\
&&\hs+\sum_{j=k_1}^{k_0-1}|[f_j(x)-f_{j+1}(x)]-[f_j(y)-f_{j+1}(y)]|
+ |f_{k_1}(x)-f_{k_1}(y)|\\
&&=: \mathrm {I}_1+\mathrm {I}_2+\mathrm {I}_3.
\end{eqnarray*}

Let us estimate the term $\mathrm {I}_1$. According to the choice of $f_{k_0}$,
$f-f_{k_0}$ is harmonic in $B(x,2^{k_0})\subset 2B$. By using the fact that $y\in B(x,2^{k_1-1})$ together with Lemma \ref{lip-har-add}, we conclude that
\begin{eqnarray*}
\mathrm {I}_1&&\le C\frac{d(x,y)}{2^{k_0}}\fint_{B(x,2^{k_0})} |f-f_{k_0}|\,d\mu\\
&&\le Cd(x,y)\lf\{\frac {1}r\fint_{2B}|f|\,d\mu
+2^{k_0}\lf(\fint_{B(x,2^{k_0})}|g|^{p}\,d\mu\r)^{1/{p}}\r\},
\end{eqnarray*}
where we used $(D)$, estimate
$\frac 1{\mu(B(x,2^{k_0}))}\le C\frac1{\mu(2B)}$,
and Lemma \ref{l2.6} to estimate $\fint_{B(x,2^{k_0})} |f_{k_0}|\,d\mu$.

The term $\mathrm {I}_2$ can be estimated similarly as the first term.
For each $k_1\le j\le k_0-1$, $f_j-f_{j+1}$ is harmonic in $B(x,2^j)$.
As $y\in B(x,2^{k_1-1})\subset \frac 12B(x,2^j)$, by using Lemma \ref{lip-har-add}
and Lemma \ref{l2.6}, we deduce that
\begin{eqnarray*}
\mathrm {I}_2&&=\sum_{j=k_1}^{k_0-1}|[f_j(x)-f_{j+1}(x)]-[f_j(y)-f_{j+1}(y)]|
\le Cd(x,y)\lf\{\sum_{j=k_1}^{k_0-1}\frac 1{2^j}\fint_{B(x,2^j)} |f_j-f_{j+1}|\,d\mu\r\}\\
&&\le Cd(x,y)\lf\{\sum_{j=k_1}^{k_0}\frac 1{2^j}\fint_{B(x,2^{j})} |f_j|\,d\mu\r\}\le Cd(x,y)\lf\{\sum_{j=k_1}^{k_0}2^{j}\lf(\fint_{B(x,2^{j})}|g|^{p}\,d\mu\r)^{1/p}\r\}.
\end{eqnarray*}

By Proposition \ref{pointwise-bound} and Lemma \ref{l2.6}, we see that for almost every $z\in B(x,2^{k_1-1})$,
\begin{eqnarray*}
|f_{k_1}(z)|&&\le C\sum_{k=-\fz}^{k_1} 2^{2k}\lf(\fint_{B(z,2^k)}|g|^{p}\,d\mu\r)^{1/{p}}
+C\fint_{B(x,2^{k_1})}|f_{k_1}|\,d\mu\\
&&\le C\sum_{k=-\fz}^{k_1} 2^{2k}\lf(\fint_{B(z,2^k)}|g|^{p}\,d\mu\r)^{1/{p}}
+C2^{2k_1}\lf(\fint_{B(x,2^{k_1})}|g|^{p}\,d\mu\r)^{1/{p}},
\end{eqnarray*}
which together with the fact that $y\in B(x,2^{k_1+1})$  implies that
$$ \mathrm {I}_3\le C2^{k_1}\lf\{\sum_{k=-\fz}^{k_1} 2^{k}\lf(\fint_{B(x,2^k)}|g|^{p}\,d\mu\r)^{1/{p}}
+\sum_{k=-\fz}^{k_1} 2^{k}\lf(\fint_{B(y,2^k)}|g|^{p}\,d\mu\r)^{1/{p}}\r\}.$$

Combining the estimates for the terms $\mathrm {I}_1$, $\mathrm {I}_2$ and $\mathrm {I}_3$, and \eqref{lip-est-add},
we see that for almost all $x,y\in B$ with $d(x,y)<1/2$,
\begin{eqnarray*}
|f(x)-f(y)|&&\le Cd(x,y)\lf\{\frac {1}r\fint_{2B}|f|\,d\mu+G_2(x)+G_2(y)\r\}.
\end{eqnarray*}
Clearly, for $g\in L^\infty(2B)$, $G_2\in L^\infty(2B)$, and hence up to a modification on a set with measure zero,
$f$ is a Lipschitz function on $B$.

For a locally Lipschitz function $\phi$, denote by $\mathrm{Lip}\, \phi$ its pointwise Lipschitz constant as
\begin{equation}\label{Lip-constant}
\mathrm{Lip}\, \phi(x):=\limsup_{d(x,y)\to 0}\frac {|\phi(x)-\phi(y)|}{d(x,y)}.
\end{equation}
By \cite[Theorem 2.1]{kz12} (also see \cite{GSC}), we see that for almost every $x\in
B$,
\begin{eqnarray*}
 |\nabla f(x)|&&\le \mathrm{Lip} \, f(x)\le C\lf\{\frac 1r\fint_{B(x_0,2r)}|f|\,d\mu
+\sum_{j=-\fz}^{k_0}2^j\lf(\fint_{B(x,2^j)}|g|^{p}\,d\mu\r)^{1/p}\r\},
\end{eqnarray*}
proving the claim.
\end{proof}

We need the following potential estimate from Haj\l asz and Koskela
\cite[Theorem 5.3]{hak}. Again, $g$ refers both to a function defined on
$2B$ and to its zero extension to the exterior of $2B.$

\begin{thm}\label{t3.3} Let $(X,d,\mu, \E)$ be a  Dirichlet metric measure space satisfying $(D_Q)$, $Q\ge 2$. Let
$B=B(x_0,r)\subset X$, $g\in L^q(2B)$, %with $q>p>\frac{2Q}{Q+2}$,
 and $G_2$ be
defined via \eqref{potential-g2}. Then

(i) for $q\in \left(\frac{2Q}{Q+2},Q\right)$ and $q^\ast=\frac{Qq}{Q-q}$,
$$\|G_2\|_{L^{q^\ast}(B)}\le Cr\mu(B)^{-1/Q}\|g\|_{L^q(2B)};$$

%(ii) for $q=Q$
%$$\fint_B \exp\lf(\frac{ C_1\mu(B)^{1/Q}G_2}
%{r
%\,\|g\|_{L^{Q}(2B)}}\r)\,d\mu\le C_2;$$
%\comment{Is this case needed ?}

(ii) for $q>Q$
$$\|G_2\|_{L^\fz(B)}\le Cr\mu(B)^{-1/q}\|g\|_{L^q(2B)}.$$
\end{thm}

Theorem \ref {t3.3} allows us to obtain the following corollary to Theorem \ref{infinite-har}.

\begin{cor}\label{infinite-poisson}
Let $(X,d,\mu,\E)$ be  a Dirichlet metric measure space
satisfying $(D_Q)$, $Q\ge 2$, and assume that $(UE)$ and $(P_{\infty,\loc})$ hold.
 Assume that $(RH_\infty)$ holds. Then for every $f\in W^{1,2}(2B)$,  $B=B(x_0,r)$, satisfying $\mathcal{L} f=g$ with $g\in L^q(2B)$,
 $q>Q$, we have
$$\||\nabla f|\|_{L^\infty(B)}\le C\lf\{\frac{1}{r}\fint_{2B}|f|\,d\mu
+r\left(\fint_{2B}|g|^{q}\,d\mu\right)^{1/q}\r\}.$$
where $C=C(Q,C_Q,C_P)>0.$
\end{cor}
\begin{proof}
If $q=\infty$, then the conclusion follows from Theorem \ref{infinite-har}.

Suppose now that $q\in (Q,\infty)$. Let $f_0\in W^{1,2}_0(2B)$ be the solution to $\mathcal{L} f_0=g$ in $2B$.
Then $f-f_0\in W^{1,2}(2B)$ and $\mathcal{L}(f-f_0)=0$. Moreover, for each $k\in\cn$, set again  $g_k:=\chi_{\{|g|\le k\}}\,g$, and let $f_k\in W^{1,2}_0(2B)$ be the solution
to $\mathcal{L} f_k=g_k$ in $2B$. By using Lemma \ref{exitence-poisson}, Theorems \ref{infinite-har} and  \ref{t3.3}, we conclude that
\begin{eqnarray*}
\||\nabla f_k|\|_{L^\infty(B)}&&\le C\lf\{\frac{1}{r}\fint_{2B}|f_k|\,d\mu
+r\|g_k\|_{L^q(2B)}\r\}\le \frac{Cr}{\mu(B)^{1/q}}\|g_k\|_{L^q(2B)}\le \frac{Cr}{\mu(B)^{1/q}}\|g\|_{L^q(2B)}.
\end{eqnarray*}

On the other hand, notice that
\begin{eqnarray*}
\int_{2B}|\nabla (f_0-f_k)|^2\,d\mu&&=\int_{2B} (f_0-f_k)(g-g_k)\,d\mu\\
&&\le \left(\|f_0\|_{L^\infty(2B)}+\|f_k\|_{L^\infty(2B)}\right)\|g-g_k\|_{L^1(2B)}\to 0,
\end{eqnarray*}
as $k\to\infty$, which, together with the preceding inequality, implies that
\begin{eqnarray*}
\||\nabla f_0|\|_{L^\infty(B)}&&\le Cr\left(\fint_{B(x_0,2r)}|g|^{q}\,d\mu\right)^{1/q}.
\end{eqnarray*}
Combining this with $(RH_\infty)$ for $f-f_0$ yields that
$$\||\nabla f|\|_{L^\infty(B)}\le C\lf\{\frac{1}{r}\fint_{2B}|f|\,d\mu
+r\left(\fint_{B(x_0,2r)}|g|^{q}\,d\mu\right)^{1/q}\r\},$$
as desired.
\end{proof}

\subsection{Harmonic functions satisfying condition $(RH_p)$ for $p\in (2,\infty)$}
\hskip\parindent
Let us now turn to the case when only a reverse H\"older inequality
$(RH_p)$, $p\in (2,\infty)$, holds
for gradients of harmonic functions. In this case, we do not know how to
get  pointwise estimates
for the gradients of solutions to Poisson equations. As a substitute for
this, we provide a
quantitative $L^p$-estimate as follows.

\begin{thm}\label{harmonic-pnorm}
Assume that the Dirichlet metric measure space $(X,d,\mu,\E)$ satisfies $(D_Q)$, $Q\ge 2$,
and that $(P_{2,\loc})$ and $(UE)$ hold.
Assume that $(RH_{p})$ holds for some $p\in (2,\infty)$. Let  $q\in \left(\frac{pQ}{Q+p},p\right]$ with
$\frac 1q-\frac 1p<\frac 1Q$.
Then for every $f\in W^{1,2}(2B)$, $B=B(x_0,r)$, satisfying $\mathcal{L} f=g$ with $g\in L^q(2B)$,
 it holds   that
$$\left(\fint_{B(x_0,r)}|\nabla f|^{p}\,d\mu\right)^{1/p}\leq \frac {C}{r}\left(\fint_{B(x_0,2r)}|f|\,d\mu+
r^2\left(\fint_{B(x_0,2r)}|g|^{q}\,d\mu\right)^{1/q}\right),$$
where  $C=C(p,q,C_Q,C_P(1),C_{LS})$.
\end{thm}

We need the following well-known Christ's dyadic cube decomposition for metric
measure spaces $(X,d,\mu)$
from \cite{cr}; also see \cite[Theorem 1.2]{hk12}.

\begin{lem}[Christ's dyadic cubes]\label{dya-cube}
There exists a collection of open subsets $\{Q^k_\alpha\subset X:\,k\in\mathbb{Z},\,
\alpha\in I_k\}$, where $I_k$ denotes a certain (possibly finite) index set depending on
$k$, and constants $\delta\in(0, 1)$,  $a_0\in(0, 1)$ and $a_1>a_0$ with $a_1\in (0,\infty)$ such that

(i) $\mu(X\setminus \cup_{\alpha}Q_\alpha^k)=0$ for all $k\in\mathbb{Z} $;

(ii) if $i> k$, then either $Q^i_\alpha\subset Q^k_\beta$ or
$Q^i_\alpha\cap Q^k_\beta=\emptyset$;

(iii) for each $k$ and all $\alpha\neq\beta\in I_k,$ $Q_\alpha^k\cap Q_\beta^k=\emptyset$;

(iv) for each $(k,\alpha)$ and each $i < k$, there exists a unique $\beta$ such that
$Q^k_\alpha\subset Q^i_\beta$;

(v) $\diam(Q_\alpha^k)\le a_1 \delta^k$;

(vi) each $Q_\alpha^k$ contains a ball $B(z^k_\alpha, a_0\delta^k)$.
\end{lem}

\begin{rem}\rm
(i) In the above lemma, we can require $\delta$ and $a_1$ to be as small as we
wish. This can been done by removing some generations, for instance,
$2k+1$-generations,  from the set; also see \cite[Theorem 1.2]{hk12}.

(ii) Under the doubling condition $(D)$,  it is easy to
see via  conditions (iii) and (v) above that
$X=\cup_{\alpha}\bar{Q}_\alpha^k$ for each $k$.
\end{rem}

The doubling condition allows us to conclude the following bounded overlap
property for the balls $B(z^k_\alpha, a_1\delta^k)$ that contain $B(z^k_\alpha, a_0\delta^k)$ from Lemma \ref{dya-cube}.
%Let the notions be the same as in lemma \ref{dya-cube}.

\begin{prop}\label{bdd-overlap}
Let $(X,d,\mu)$  be a  Dirichlet metric measure space satisfying $(D_Q)$ for some $Q\ge 2$.
%Let the notions be the same as in lemma \ref{dya-cube}.
For each $\alpha$ and $k$, let $B^k_\alpha=B(z^k_\alpha, a_1\delta^k)$.
Then for each dilation $t>1$, there exists a constant $C(t,a_0,a_1,C_Q,Q)>0$ such
that for each $k$,
$$\sum_{\alpha}\chi_{tB_\alpha^k}(x)\le C(t,a_1,C_Q,Q).$$
\end{prop}
\begin{proof}
For each $x\in X$, let
$$C(x,k)=\sum_{\alpha}\chi_{tB_\alpha^k}(x).$$
Fix a point $x_0$ and $k\in \zz$, and consider the ball
$B(x_0,2ta_1\delta^k)$. Then there exist $C(x_0,k)$ distinct balls, say
$\{B^k_{\alpha_j}\}_{j\le C(x_0,k)}$, that are inside
$B(x_0,2ta_1\delta^k)$. Using the doubling condition and the properties (iii) and (vi) of
the dyadic cubes from Lemma \ref{dya-cube}, we see that
\begin{eqnarray*}
V(x_0,2ta_1\delta^k)&&\ge \sum_{j=1}^{C(x_0,k)}V(z^k_{\alpha_j},a_0\delta^k)
\ge \sum_{j=1}^{C(x_0,k)}\frac{1}{C_Q(4ta_1)^Q}\mu(4tB^k_{\alpha_j})\\
&&\ge \frac{C(x_0,k)a_0^Q}{C_Q(4ta_1)^Q}V(x_0,2ta_1\delta^k).
\end{eqnarray*}
Therefore, we conclude that
$$C(x_0,k)\le C_Q(4ta_1/a_0)^Q,$$
which completes the proof.
\end{proof}

We need the following geometric consequence of doubling; see \cite{hkst} for
instance.
\begin{lem}\label{geometric-doubling}
 Let $(X,d,\mu)$  be a  doubling Dirichlet metric measure space.
 Then there exists a constant $N_\mu$ such that for each $k\ge 1$, every $2^{-k}r$-separated set in any ball $B(x,r)$ in $X$ has at most
$N_\mu^k$ elements.
\end{lem}
We shall make use of the Hardy-Littlewood maximal functions.
\begin{defn}[Hardy-Littlewood maximal function]
For any locally integrable function $f$ on $X$, its Hardy-Littlewood maximal function is defined as
$$\mathcal{M}f(x):=\sup_{B:\,x\in B}\fint_B|f|\,d\mu,$$
where $B$ is any ball. For $p>1$, we define the $p$-Hardy-Littlewood maximal function as
$$\mathcal{M}_pf(x):=\sup_{B:\, x\in B}\left(\fint_B|f|^p\,d\mu\right)^{1/p}.$$
\end{defn}

Using the Poincar\'e inequality, it readily follows that $\mathcal{M}_2(|\nabla f|)$ generates a
Haj\l asz gradient in the following sense; see Appendix \ref{appendix-sobolev}. The proof uses a  telescopic argument, which is by now  classical;
see for instance \cite{hak} and the monographs \cite{bb10,hkst}.
\begin{lem}\label{hajlasz-gradient}
Assume that $(X,d,\mu,\E)$ satisfies $(D)$
and $(P_{2,\loc})$. Then for each $\beta\in (0,1)$ and $r_0>0$, there exists $C=C(C_D,\beta,C_P(r_0))>0$
such that, for all $f\in W^{1,2}(B)$,   $B=B(x_0,r)$,  it holds for almost all
$x,y\in \beta B$, that
$$|f(x)-f(y)|\le Cd(x,y)\left(\mathcal{M}_2(|\nabla f|\chi_B)(x)+\mathcal{M}_2(|\nabla f|\chi_B)(y)\right).$$
Moreover, if $f$ is continuous on $B$, then the above inequality holds for all $x,y\in \beta B$.
\end{lem}

Let us now turn to the proof of the main gradient estimate. Recall that, under $(D_Q)$ together with $(P_{2,\loc})$,
every solution $f$ to the equation $\mathcal{L} f=g$ with $g\in L^\fz$ is locally H\"older
continuous according to Lemma \ref{l2.7}: there exists a modification $\tilde f$ such that $\tilde f=f$ a.e., and
every point is a Lebesgue point of $\tilde f$.
Thus, in what follows, we may assume that every point is a Lebesgue point of our solution.

\begin{proof}[Proof of Theorem \ref{harmonic-pnorm}]
%We first prove the statement under addition assumption that  $g\in L^\infty(2B)$
For simplicity, we assume that $a_1=1$ and $\delta=\frac 14$ in  Lemma \ref{dya-cube}.

We divide the proof into  five  steps.
In first four steps we prove  that the statement is valid  under the additional assumption that $g\in L^\infty(2B)$.
Then in the last step we use truncations to remove this additional assumption and conclude the proof.

\

 {\bf Step 1.} {\it Construction of a chain of balls.}

 Let $k_0=[-\log_4r]$ be the largest integer smaller than $-\log_4r$, and
set $B^{k_0}=B(x_0,6r)$. Fix a dyadic decomposition as in Lemma \ref{dya-cube}.
 For each $k>k_0,$ let
 $$I_{B,k}:=\{\alpha:\, Q_\alpha^{k+2}\cap B(x_0,r)\neq \emptyset\},$$
 and
 $$\mathcal{F}_k=\left\{B_\alpha^k:=B(z_\alpha^{k+2},2^{-2k}):\, \alpha\in I_{B,k}\right\}.$$
 Then, by Proposition \ref{bdd-overlap}, we see that for each $k>k_0$, it
holds that
 $$\sum_{\alpha\in I_{B,k}:\,B_\alpha^k\in \mathcal{F}_k}\chi_{B_\alpha^k}(x)\le C_Q(64/a_0)^Q.$$
 From the properties of our dyadic cubes (Lemma \ref{dya-cube}),
we see that:

 (i) $B(x_0,r)\subset \cup_{\alpha\in I_{B,k}}B_\alpha^k$ for all $k>k_0$;

(ii) for each $B_\alpha^k\in \mathcal{F}_k$, there exist
balls $B_\alpha^j\in \mathcal{F}_j$, $k_0< j<k$, such that for all
$k_0<j<k$, $B_\alpha^{j+1}\subset \frac 13B^{j}_\alpha$, and
$B_\alpha^{k_0+1}\subset \frac{1}{3}B^{k_0}=2B$.

We call the collection $B_{\alpha}^{k_0+1},\dots,B_{\alpha}^{k-1}$ a chain
associated to $B_{\alpha}^k$ (and hence to $Q_{\alpha}^k$).

{\bf Proof of (ii):} If $B_\alpha^k\in \mathcal{F}_k$, then $Q_\alpha^{k+2}\cap B(x_0,r)\neq \emptyset$.
Therefore, there exists $Q_\alpha^{k+1}$ that contains  $Q_\alpha^{k+2}$ and hence,
$Q_\alpha^{k+1}\cap B(x_0,r)\neq \emptyset$ and $d(z_\alpha^{k+2},z_\alpha^{k+1})\le 2^{-2k-2}$ (by Lemma \ref{dya-cube} (v)).

For each $x\in B_\alpha^k$,
$$d(x,z_\alpha^{k+1})<2^{-2k-2}+2^{-2k}=\frac 542^{-2k}< \frac{5}{16}2^{-2k+2}.$$
 From this, we conclude that $\frac 13B_\alpha^{k-1}\supset B_\alpha^k$.

In what follows, for each $B_\alpha^k\in \mathcal{F}_k$, we fix
a chain from (ii).

\

 {\bf Step 2.} {\it Construction of a Haj\l asz  gradient via the chain.}

We first assume that $\mathcal{L} f=g$ in $6B=B^{k_0}$, $B=B(x_0,r)$,
and $g\in L^\infty(6B)$. In the last step of the proof, we will complete the proof by using $2B$
instead of $6B$.

Let $f_{k_0}\in W^{1,2}_0(B(x_0,6r))$ be the solution to
$$\mathcal{L} f_{k_0}=g$$
in  $B(x_0,6r)$; the existence of a unique solution is guaranteed by Lemma \ref{exitence-poisson}.
For each $k>k_0$ and $B_\alpha^k\in \mathcal{F}_k$, $a\in I_{B,k}$, we let
 $f_{\alpha,k}\in W^{1,2}_0(B^k_\alpha)$ be the solution to the Poisson equation
 $$\mathcal{L} f_{\alpha,k}=g$$
in  $B^k_\alpha.$
 Then by Lemma \ref{l2.6}, we see that for each $k\ge k_0$ and every $a\in I_{B,k}$,
\begin{equation}\label{van-sol}
\fint_{B^k_\alpha}|f_{\alpha,k}|\,d\mu\le C2^{-4k}\left(\fint_{B^k_\alpha}|g|^{q}\right)^{1/q}.
 \end{equation}
 In what follows, for consistency, we will write $f_{k_0}$ as $f_{\alpha,k_0}$, $\alpha\in I_{B,k_0}$,
 although there is only one element in $I_{B,k_0}$.

Define a function $w_{k_0}$ on $6B=B(x_0,6r)$ by setting
$$w_{k_0}(x)=\mathcal{M}_2(|\nabla (f-f_{\alpha,k_0})|\chi_{ 3B})(x).$$
For each $k>k_0$ and every $\alpha\in I_{B,k}$, let $\alpha'\in I_{B,k-1}$
the unique one such that  $Q^{k+2}_\alpha\subset Q^{k+1}_{\alpha'}$.
Define
 $$w_k(x):=\sum_{\alpha\in I_{B,k}: B^k_\alpha\in \mathcal{F}_k}
 \mathcal{M}_2(|\nabla (f_{\alpha,k}-f_{\alpha',k-1})|\chi_{\frac 12B^k_\alpha})(x)\chi_{B^k_\alpha}(x)$$
on $6B.$
 %Notice that the key point is that $u_\alpha^k-u_{\alpha'}^{k-1}$ are
%harmonic functions on $B^k_\alpha$.

 We also set
 $$G_3(x):=\sum_{j=2k_0-4}^{\infty}2^{-j}
 \left(\fint_{B(x,2^{-j})}|h|^{q}\,d\mu\right)^{1/q},$$
where  $h$ is the zero extension of $g$ to
$X\setminus 6B.$
%when $Q\in [1,2)$, $p_Q\in (1,2)$ can be taken arbitrarily close to 1
%when $Q=2$, and $p_Q=\frac {2Q}{Q+2}$ when $Q>2$.

\

{\it Claim:}
There exists $C=C(C_D,C_{LS},C_P(1))>0$ such that for  all $x,\,y\in B(x_0,r)$ with $d(x,y)<1/2$, it holds that
\begin{equation} \label{arvio}
|f(x)-f(y)|\le Cd(x,y)\left\{G_3(x)+G_3(y)+\sum_{k\ge k_0}w_k(x)+\sum_{k\ge k_0}w_k(y)\right\}.
\end{equation}

{\bf Proof of the claim:} Let $x,y\in B$  such that $d(x,y)<1/2$.
If $d(x,y)\ge \frac r{64}$, then
$$|f(x)-f(y)|\le |(f-f_{\alpha,k_0})(x)-(f-f_{\alpha,k_0})(y)|+|f_{\alpha,k_0}(x)-f_{\alpha,k_0}(y)|,$$
where  by Lemma \ref{hajlasz-gradient} with $(P_{2,\loc})$ for balls with radii at most one and $\beta=\frac 12$, we have
\begin{eqnarray*}
&&|(f-f_{\alpha,k_0})(x)-(f-f_{\alpha,k_0})(y)|\\
&&\quad\le Cd(x,y)\left[\mathcal{M}_2\left(|\nabla(f-f_{\alpha,k_0})|\chi_{3B}\right)(x)+\mathcal{M}_2
\left(|\nabla(f-f_{\alpha,k_0})|\chi_{ 3B}\right)(y)\right],
\end{eqnarray*}
and by  Proposition \ref{pointwise-bound} and \eqref{van-sol} we have
$$|f_{\alpha,k_0}(x)-f_{\alpha,k_0}(y)|\le Cd(x,y)[G_3(x)+G_3(y)].$$
The above two estimates complete the case $d(x,y)\ge \frac r{64}$.

Suppose now $d(x,y)<\frac r{64}$ and $d(x,y)<1/2$. Then there exists  $k>k_0$ such that $1/2^{2k+6}\le d(x,y)< 1/2^{2k+4}$.
From the properties of dyadic cubes, Lemma \ref{dya-cube}, we see that
there exists a cube $Q_\alpha^{k+2}$ such that $x\in \bar{ Q}_\alpha^{k+2}$.
Noticing that $B^k_\alpha=B(z_\alpha^{k+2},2^{-2k})$,
we see that
$$d(y,z_\alpha^{k+2})\le d(x,y)+d(x,z_\alpha^{k+2})<\frac 1{2^{2k+4}}+\frac 1{2^{2k+4}}=\frac 1{2^{2k+3}},$$
and hence, $x,y\in \frac 13B^k_\alpha$.

Let $\{B_\alpha^j\in \mathcal{F}_j\}_{k_0\le j< k}$ be the chain of
balls such that $\frac 13B_\alpha^j\supset B_\alpha^{j+1}$, whose
existence is guaranteed by {\bf Step 1} (ii). Applying a telescopic
argument, we obtain
\begin{eqnarray*}
|f(x)-f(y)|&&\le |(f-f_{\alpha,k_0})(x)-(f-f_{\alpha,k_0})(y)|+\sum_{j=k_0}^{k-1}|(f_{\alpha,j}-f_{\alpha,j+1})(x)-(f_{\alpha,j}-f_{\alpha,j+1})(y)|\\
&&\quad+|f_{\alpha,k}(x)-f_{\alpha,k}(y)|.
\end{eqnarray*}
By using Lemma \ref{hajlasz-gradient} with $(P_{2,\loc})$ for balls with radii at most one,
$\beta=\frac{2}{3}$ for $k>k_0$ and $\beta=\frac 12$ for $k_0$, we conclude that
\begin{eqnarray*}
&&   |(f-f_{\alpha,k_0})(x)-(f-f_{\alpha,k_0})(y)|+\sum_{j=k_0}^{k-1}|(f_{\alpha,j}-f_{\alpha,j+1})(x)-(f_{\alpha,j}-f_{\alpha,j+1})(y)|\\
&&\quad\le Cd(x,y)\left\{w_{k_0}(x)+w_{k_0}(y)+\sum_{k_0<j\le k}w_j(x)+\sum_{k_0<j\le k}w_j(y)\right\}\\
&&\quad \le  Cd(x,y)\left\{\sum_{k\ge k_0}w_k(x)+\sum_{k\ge k_0}w_k(y)\right\}.
\end{eqnarray*}

On the other hand, by using \eqref{van-sol}, Proposition \ref{pointwise-bound}
and that $d(x,y)\approx 2^{-2k}$, we see that
\begin{eqnarray*}
|f_{\alpha,k}(x)-f_{\alpha,k}(y)|
&&\le |f_{\alpha,k}(x)|+|f_{\alpha,k}(y)|\\
&&\le Cd(x,y)\left\{\sum_{j=2k}^\infty 2^{-j}\left(\fint_{B(x,2^{-j})}|g|^{q}\,d\mu\right)^{1/q}+\sum_{j=2k}^\infty 2^{-j}\left(\fint_{B(y,2^{-j})}|g|^{q}\,d\mu\right)^{1/q}\right\}\\
&&\le Cd(x,y)\left\{G_3(x)+G_3(y)\right\}.
\end{eqnarray*}
The above two estimates imply the claim.

\

{\bf Step 3.} {\it Claim:} For $\tilde{q}\in (q,p]$ with $1/\tilde{q}-1/p<
1/Q,$  we have the estimate
\begin{eqnarray*}
\left\|G_3+\sum_{k\ge k_0}w_k\right\|_{L^p(B(x_0,r))}&&\le
\frac {C(\tilde q)[V(x_0,6r)]^{1/p}}{r}\left(\fint_{B(x_0,6r)}|f|\,d\mu+
r^2\left(\fint_{B(x_0,6r)}|g|^{\tilde{q}}\,d\mu\right)^{1/\tilde{q}}\right).
\end{eqnarray*}

We begin by estimating the $L^p$-norm of the second term on the left-hand side.
For a ball $B^k_\alpha\in \mathcal{F}_k$, let
$B^{k-1}_{\alpha'}=B^{k-1}_{\alpha'(\alpha)}\in \mathcal{F}_{k-1}$ be the ball
from the definition of $w_k;$ then it satisfies $\frac
13B^{k-1}_{\alpha'}\supset B^k_\alpha\in \mathcal{F}_k$.
Notice that $f_{k,\alpha}-f_{\alpha,k-1}$ is harmonic on $B_{\alpha}^{k}.$
Hence $(RH_{p_0})$ \eqref{van-sol} and the boundedness of the usual
Hardy-Littlewood maximal operator on $L^{p/2}$ with $p> 2$  gives,
for all $k>k_0$ and
$\alpha\in I_{B,k}$, that
\begin{eqnarray*}
\int_{X} \left[\mathcal{M}_2(|\nabla (f_{\alpha,k}-f_{\alpha,k-1})|\chi_{\frac 12B^k_\alpha})(x)\chi_{B^k_\alpha}(x)\right]^p\,d\mu(x)&&\le \int_{{\frac 12B^k_\alpha}} |\nabla (f_{\alpha,k}-f_{\alpha,k-1})|^p\,d\mu(x)\\
&&\le C\mu(B^k_\alpha)\left(\frac{2^{2k}}{\mu(B^k_\alpha)}\int_{ B^k_\alpha}| (f_{\alpha,k}-f_{\alpha,k-1})|\,d\mu\right)^p\\
&&\le C\mu(B^{k-1}_{\alpha'})2^{-2pk}\left(\frac{1}{\mu(B^{k-1}_{\alpha'})}\int_{ B^{k-1}_{\alpha'}}| g|^{q}\,d\mu\right)^{p/q}.
\end{eqnarray*}
Combining this with the fact that the sets $\{B_\alpha^k\}_{\alpha\in I_{B,k}}$
have uniformly  bounded overlaps for each $k$,
we have that for each $k> k_0$ and for each $2<p\le p_0$,
\begin{eqnarray}\label{wk-est}
\left\|w_k\right\|_{L^p(B(x_0,6r))}^p&&\le C(\mu)
\int_{B(x_0,6r)}\left(\sum_{\alpha\in I_{B,k}: B^k_\alpha\in \mathcal{F}_k}\mathcal{M}_2(|\nabla (f_{\alpha,k}-f_{\alpha,k-1})|\chi_{\frac 12B^k_\alpha})(x)\chi_{B^k_\alpha}(x)\right)^p\,d\mu(x)\nonumber\\
&&\le C(\mu)\sum_{\alpha\in I_{B,k}: B^k_\alpha\in \mathcal{F}_k}
\int_{B^k_\alpha}\left(\mathcal{M}_2(|\nabla (f_{\alpha,k}-f_{\alpha,k-1})|\chi_{\frac 12B^k_\alpha})(x)\right)^p\,d\mu(x)\nonumber\\
&&\le  C(\mu)\sum_{\alpha\in I_{B,k}:B^{k}_\alpha\in \mathcal{F}_k}\mu(B^{k-1}_{\alpha'(\alpha)})2^{-2pk}\left(\frac{1}{\mu(B^{k-1}_{\alpha'(\alpha)})}\int_{ B^{k-1}_{\alpha'(\alpha)}}| g|^{q}\,d\mu\right)^{p/q}\nonumber\\
&&\le C(\mu)\sum_{\alpha\in I_{B,k}:B^k_\alpha\in \mathcal{F}_k}\mu(B^{k-1}_{\alpha'(\alpha)})2^{-2pk}\left(\frac{1}{\mu(B^{k-1}_{\alpha'(\alpha)})}\int_{ B^{k-1}_{\alpha'(\alpha)}}| g|^{\tilde{q}}\,d\mu\right)^{p/\tilde{q}}\nonumber\\
&&\overset{\text{doubling}}{\leq} C(\mu)\sum_{\alpha\in I_{B,k}:B^k_\alpha\in \mathcal{F}_k}V(x_0,6r)^{1-\frac p{\tilde{q}}}\frac{2^{2kQ(\frac p{\tilde{q}}-1)}}{r^{Q(\frac p{\tilde{q}}-1)}}2^{-2pk}\left(\int_{ B^{k-1}_{\alpha'(\alpha)}}| g|^{\tilde{q}}\,d\mu\right)^{p/\tilde{q}}\nonumber\\
&&\le C(\mu)V(x_0,6r)^{1-\frac p{\tilde{q}}}\frac{2^{2kQ(\frac p{\tilde{q}}-1)}}{r^{Q(\frac p{\tilde{q}}-1)}}2^{-2pk}\left(\int_{B(x_0,6r)}| g|^{\tilde{q}}\,d\mu\right)^{p/\tilde{q}}.
\end{eqnarray}
Above the last inequality relies on $\tilde{q}\le p$ and the fact that
$$\sum_{\alpha \in I_{B,k}: B^k_\alpha\in \mathcal{F}_k}\chi_{B_{\alpha'(\alpha)}^{k-1}}(x)\le C(\mu,a_0).$$
%where $B^{k-1}_{\alpha'}\in \mathcal{F}_{k-1}$ is the ball from the
%chain from {\bf Step 1} (ii)
%that satisfies $\frac 13B^{k-1}_{\alpha'}\supset B^k_\alpha\in \mathcal{F}_k$.
Indeed, since $B^k_\alpha\subset \frac 13B_{\alpha'(\alpha)}^{k-1}$, we have $Q^{k+2}_\alpha\subset \frac 13B_{\alpha'}^{k-1}$.
For each $\alpha'\in I_{B,k-1}$, let
\begin{eqnarray*}
I_{\alpha',k}:=\left\{\alpha: \alpha \in I_{B,k}, B^k_\alpha\subset \frac 13B_{\alpha'}^{k-1}\right\}.
\end{eqnarray*}
By  using  dyadic cubes again, we see that
\begin{eqnarray*}
\mu(B_{\alpha'}^{k-1})&&\ge \sum_{\alpha\in I_{\alpha',k}}V(z^{k+2}_\alpha,a_02^{-2k-4})\ge  \sum_{\alpha\in I_{\alpha',k}}\frac{C(\mu)}{(2^8a_0)^Q}V(z^{k+2}_\alpha,2^{-2k+4})\ge \sum_{\alpha\in I_{\alpha',k}}\frac{C(\mu)a_0^Q}{(2^8)^Q}\mu(B_{\alpha'}^{k-1}),
\end{eqnarray*}
which implies that $\#(I_{\alpha',k})\le \frac{2^{8Q}}{C(\mu)(a_0)^Q}$. Therefore, we conclude that
\begin{eqnarray*}
\sum_{\alpha \in I_{B,k}: B^k_\alpha\in \mathcal{F}_k}\chi_{B_{\alpha'(\alpha)}^{k-1}}(x)&&\le
\sum_{\alpha' \in I_{B,k-1}: B^{k-1}_{\alpha'}\in \mathcal{F}_{k-1}}\chi_{B_{\alpha'}^{k-1}}(x)\cdot\#(I_{\alpha',k})\le C(\mu,a_0).
\end{eqnarray*}

By \eqref{wk-est}
\begin{eqnarray*}
\left\|w_k\right\|_{L^p(B(x_0,2r))}
&&\le C(\mu)V(x_0,6r)^{\frac 1p-\frac 1{\tilde{q}}}\frac{2^{2kQ(\frac 1{\tilde{q}}-\frac 1p)}}{r^{Q(\frac 1{\tilde{q}}-\frac 1p)}}2^{-2k}\left(\int_{B(x_0,6r)}| g|^{\tilde{q}}\,d\mu\right)^{1/\tilde{q}}.
\end{eqnarray*}
Therefore, by the Minkowski inequality,
$$\left\|\sum_{k>k_0}w_k\right\|_{L^p(B(x_0,r))}\le C(\mu)V(x_0,6r)^{\frac 1p-\frac 1{\tilde{q}}}r\left(\int_{B(x_0,6r)}| g|^{\tilde{q}}\,d\mu\right)^{1/\tilde{q}}$$
provided $q<\tilde{q}\le p$ and $\frac 1{\tilde{q}}-\frac 1p<\frac 1Q$.

By applying Lemma \ref{l2.6} and $(RH_{p})$, we conclude that
\begin{eqnarray}\label{w0-est}
\|w_{k_0}\|_{L^p(B(x_0,r))}&&\le V(x_0,2r)^{1/p}\frac {C}{r}\fint_{B(x_0,2r)}|f-f_{\alpha,k_0}|\,d\mu\nonumber\\
&&\le  \frac {CV(x_0,2r)^{1/p}}{r}\left(\fint_{B(x_0,6r)}|f|\,d\mu+
r^2\left(\fint_{B(x_0,6r)}|g|^{q}\,d\mu\right)^{1/q}\right),
\end{eqnarray}
which completes the estimate for the $L^p$-integral of the second term on the
left-hand side; recall that $q<\tilde{q}.$

Regarding the first term, an estimate similar to the one in Theorem \ref{t3.3} (see also \cite[Theorem 5.3]{hak}) yields
\begin{eqnarray*}
\|G_3\|_{L^p(B(x_0,r))}&&\le C(\mu)rV(x_0,6r)^{\frac 1p-\frac 1{\tilde{q}}}\left(\int_{B(x_0,6r)}| g|^{\tilde{q}}\,d\mu\right)^{1/\tilde{q}}.
\end{eqnarray*}

The claim then follows by combining the last inequality with \eqref{wk-est} and \eqref{w0-est}.

\

{\bf Step 4.} {\it Completion of the $L^\infty$ case.}

For each  $y_0\in B(x_0,r/2)$,  let $0\leq \psi_r\leq 1$ be a { one-parameter family of Lipschitz cut-off functions such that}
\begin{equation}\label{cutoff}
\aligned
&\psi_r(x)=1\text{ whenever } x\in B(y_0, \min\{r/4,1/8\}),\\
&\psi_r(x)=0\text{ whenever }x\in X\setminus B(y_0,\min\{r/2,1/4\}),\,\,\,\text{and }|\nabla\psi_r(x)|\leq\frac{C}{\min\{r,1\}}.
\endaligned\end{equation}
Then by {\bf Step 2}, we see that,
for all $x,\,y\in X$,
\begin{eqnarray*}
&&|(f\psi_r)(x)-(f\psi_r)(y)|\\
&&\quad\le Cd(x,y)\left\{\left(\frac{|f(x)|}{\min\{r,1\}}+G_3(x)+\sum_{k\ge k_0}w_k(x)\right)\chi_{2B}(x)+\left(\frac{|f(y)|}{\min\{r,1\}}+G_3(y)+\sum_{k\ge k_0}w_k(y)\right)\chi_{2B}(y)\right\}.
\end{eqnarray*}
Recall our assumption that $g\in
L^\infty(6B)$. Therefore, by applying Lemma \ref{exitence-poisson}
to $f_{\alpha,k_0}$ and Proposition \ref{l2.3} to $f-f_{\alpha,k_0}$, we see
that  $f\in L^\infty(2B)$. This, together with
{\bf Step 3}, implies that $f\psi_r\in W^{1,p}_0(B(y_0,\min\{r/2,1/4\}))$; see Appendix \ref{appendix-sobolev}.

By \eqref{arvio} and the pointwise estimate of the gradient of a Sobolev
function (see Appendix \ref{appendix-sobolev}) for
$f\psi_r$,  we conclude that
$$|\nabla f(x)|=|\nabla (f\psi_r)(x)|\le CG_3(x)+C\sum_{k\ge k_0}w_k(x)$$
for a.e. $x\in B(y_0,\min\{r/4,1/8\})$. By the arbitrariness of $y_0$, we see this estimate holds for a.e.  $x\in B(x_0,r/2)$.
This together with the estimate from
{\bf Step 3} yields
 \begin{eqnarray}\label{est-12}
\||\nabla f|\|_{L^p(B(x_0,r/2))}&& \le
\frac {CV(x_0,6r)^{1/p}}{r}\left(\fint_{B(x_0,6r)}|f|\,d\mu+
r^2\left(\fint_{B(x_0,6r)}|g|^{q}\,d\mu\right)^{1/q}\right).
\end{eqnarray}

%Since a $Q$-doubling measure, $Q\le 2$, is always a $\tilde Q$-doubling measure, $\tilde Q>2$,
%for simplicity of the terminologies, we here assume that $Q>2$.
%Let $j_0$ be the smallest natural number such that $Q\le jk_0$.
%For each $j\in\cn$ satisfying $j< j_0$, let
%$p_j:=\frac{2Q}{Q-2j}$, and $p_{j_0}:=\infty$. Notice that, for $j<j_0$, $p_j=(p_{j-1})^\ast=\frac{Qp_{j-1}}{Q-p_{j-1}}$.
%
%Recall that $p\le p_0<\infty$.
%
%{\bf Case 1.} $p\in (2,p_1]$.
%
%In this case, by applying Lemma \ref{l2.6} to $u^{k_0}_\alpha$ ($q>p_Q$), Lemma \ref{l2.3} to $u-u^{k_0}_\alpha$ and
%the Sobolev inequality \eqref{x2.5}, we see that
%\begin{eqnarray*}
%\left(\fint_{B}|u|^{p_1}\,d\mu\right)^{1/p_1}&&\le \left(\fint_{B}|u-u_\alpha^{k_0}|^{p_1}\,d\mu\right)^{1/p_1}+ \left(\fint_{B}|u_\alpha^{k_0}|^{p_1}\,d\mu\right)^{1/p_1}\\
%&&\le C \fint_{6B}|u|\,d\mu+CR \lf(\fint_{6B}|\nabla u_\alpha^{k_0}|^2\,d\mu\r)^{1/2}\\
%&&\le C \fint_{6B}|u|\,d\mu+CR^2 \lf(\fint_{B}|g|^{q}\,d\mu\r)^{1/q}.
%\end{eqnarray*}

Let us now replace $B(x_0,6r)$ on the R.H.S. by $2B$, $ B=B(x_0,r)$.
By using Lemma \ref{geometric-doubling}, we see that
$B(x_0,r)$ contains at most $N_\mu^5$ separate balls with radii  $r/32$.
Fix such a maximal collection, which we for simplicity assume to have exactly $N_\mu^5$ elements.
Denote these balls by $\{B(x_i,r/32)\}_{1\le i\le N_\mu^5}$.
Then
$$B(x_0,r)\subset \bigcup_{1\le i\le N_\mu^5}B(x_i,r/16).$$
By applying the estimate \eqref{est-12} to each $B(x_i, \frac{12r}{16})$ yields
 \begin{eqnarray*}
\||\nabla f|\|_{L^p(B(x_i,r/16))}&&\le
\frac {CV(x_i,\frac{12r}{16})^{1/p}}{r}\left(\fint_{B(x_i,\frac{12r}{16})}|f|\,d\mu+
r^2\left(\fint_{B(x_i,\frac{12r}{16})}|g|^{q}\,d\mu\right)^{1/q}\right)\\
&&\le \frac {CV(x_0,2r)^{1/p}}{r}\left(\fint_{B(x_0,2r)}|f|\,d\mu+
r^2\left(\fint_{B(x_0,2r)}|g|^{q}\,d\mu\right)^{1/q}\right).
\end{eqnarray*}
Therefore, we conclude that
\begin{eqnarray*}
\||\nabla f|\|_{L^p(B(x_0,r))}&&\le \sum_{i=1}^{N_\mu^5} \||\nabla f|\|_{L^p(B(x_i,r/16))}\\
&&\le \frac {CV(x_0,2r)^{1/p}}{r}\left(\fint_{B(x_0,2r)}|f|\,d\mu+
r^2\left(\fint_{B(x_0,2r)}|g|^{q}\,d\mu\right)^{1/q}\right),
\end{eqnarray*}
which completes the proof in the case $g \in L^\infty(2B)$.

\

{\bf Step 5.}  {\it Truncation argument.}

Once again, for each $k\in\cn$, let $g_k:=\chi_{\{|g|\le k\}}\,g$, and let $f_k\in W^{1,2}_0(2B)$ be the solution
to $\mathcal{L} f_k=g_k$ in $2B$.

By Lemma \ref{l2.6}, we see that there exists a solution
$f_0\in W^{1,2}_0(2B)$ to $\mathcal{L} f_0=g$ in $2B$,
with
\begin{equation}\label{perus}
\fint_{2B}|f_0|\,d\mu\le Cr \lf(\fint_{2B}|\nabla f_0|^2\,d\mu\r)^{1/2}\le Cr^2 \lf(\fint_{2B}|g|^{q}\,d\mu\r)^{1/q}.
\end{equation}
For each $k\in\cn$, by using Lemma \ref{l2.6} again, we obtain
$$\fint_{2B}|f_0-f_k|\,d\mu\le Cr \lf(\fint_{2B}|\nabla (f_0-f_k)|^2\,d\mu\r)^{1/2}\le Cr^2 \lf(\fint_{2B}|g-g_k|^{q}\,d\mu\r)^{1/q},$$
since $f_0-f_k\in W^{1,2}_0(2B)$. Consequently $f_k\to f_0$ in $W^{1,2}_0(2B)$.

By Lemma \ref{l2.6} and Theorem \ref{harmonic-pnorm}, we have for each $k\in\cn$ that
\begin{eqnarray*}
\left(\fint_{B(x_0,r)}|\nabla f_k|^{p}\,d\mu\right)^{1/p}&&\leq \frac {C}{r}\left(\fint_{B(x_0,2r)}|f_k|\,d\mu+
r^2\left(\fint_{B(x_0,2r)}|g_k|^{q}\,d\mu\right)^{1/q}\right)\\
&&\le Cr\left(\fint_{B(x_0,2r)}|g_k|^{q}\,d\mu\right)^{1/q}\le Cr\left(\fint_{B(x_0,2r)}|g|^{q}\,d\mu\right)^{1/q}.
\end{eqnarray*}
Letting $k\to \infty$, we conclude that
\begin{eqnarray*}
\left(\fint_{B(x_0,r)}|\nabla f_0|^{p}\,d\mu\right)^{1/p}&&\le Cr\left(\fint_{B(x_0,2r)}|g|^{q}\,d\mu\right)^{1/q}.
\end{eqnarray*}
By applying this together with $(RH_p)$ for the harmonic function $f-f_0$  on
$B(x_0,r)$ and \eqref{perus} we obtain
\begin{eqnarray*}
\left(\fint_{B(x_0,r)}|\nabla f|^{p}\,d\mu\right)^{1/p}&&\le \left(\fint_{B(x_0,r)}|\nabla (f-f_0)|^{p}\,d\mu\right)^{1/p}+\left(\fint_{B(x_0,r)}|\nabla f_0|^{p}\,d\mu\right)^{1/p}\\
&&\le \frac{C}{r}\fint_{B(x_0,2r)}|f-f_0|\,d\mu+Cr\left(\fint_{B(x_0,2r)}|g|^{q}\,d\mu\right)^{1/q}\\
&&\le \frac{C}{r}\fint_{B(x_0,2r)}|f|\,d\mu+Cr\left(\fint_{B(x_0,2r)}|g|^{q}\,d\mu\right)^{1/q},
\end{eqnarray*}
as desired.

\end{proof}

Corollary \ref{infinite-poisson} and Theorems \ref{harmonic-pnorm} yield the following quantitative H\"older regularity of solutions to the
Poisson equation.
\begin{cor}\label{cor-holder}
Let $(X,d,\mu,\E)$ be a Dirichlet metric measure space
satisfying $(D_Q)$ with $Q\ge 2$, and suppose that $(UE)$ holds. Assume that $(RH_{p})$ and $(P_{2,\loc})$ hold for some
$p\in (Q,\infty]$. Let $q>\max\{Q/2,1\}$ and $\alpha:=\alpha(p,q)=\min\{1-Q/p,2-Q/q\}$. If
$\mathcal{L} f=g$ in  $B(x_0,r)$ with $g\in L^q(B)$, then $f$ belongs to
$C^\alpha_\loc(B)$.
\end{cor}

\section{Elliptic equations vs parabolic equations}

\subsection{From elliptic equations  to parabolic equations}

\hskip\parindent
In this section, we give quantitative gradient estimates for
the heat kernel by using the regularity of solutions to the
Poisson equation.

%The standing assumptions in this section are that $(X,d)$
%is complete and non-compact.

To begin with, let us recall that, under  $(D)$ and  $(UE)$,
we have the estimate
\begin{equation}\label{tly}
\lf|\frac{\partial h_t}{\partial t}(x,y)\r|\le
\frac C{t\,V(y,\sqrt t)}\exp\lf\{-\frac{d^2(x,y)}{ct}\r\},
\end{equation}
for the time derivative of the heat kernel for all $t>0$;
see \cite{bcs15,sal,st2,st3}.

A version of the following result, requiring the slightly stronger condition $(P_2)$, has been established in \cite{ji15}.
The proof below follows the ideas of the proof of \cite[Theorem 3.2]{ji15}.
%The following proof is the same as that of \cite[Theorem 3.2]{ji15}, but since the assumptions here is different from
%\cite[Theorem 3.2]{ji15},  we give a fully adapted proof for the sake of completeness.
\begin{prop}\label{infinite-heat}
Assume that the  doubling Dirichlet metric measure space $(X,d,\mu,\E)$ satisfies  $(UE)$ and $(P_{\infty,\loc})$.
Then $(RH_\infty)$ implies
$(GLY_{\infty})$.
\end{prop}
\begin{proof}
By using Theorem \ref{infinite-har} and following the  proof of \cite[Theorem 3.2]{ji15}, we conclude the claim.
\end{proof}

Our next result follows via the argument in \cite[p 941]{acdh}.
\begin{prop}\label{bdd-infinite-heat}
Assume that the doubling Dirichlet metric measure space $(X,d,\mu,\E)$ satisfies  $(UE)$.
Then $(GLY_\infty)$ implies $(G_p)$ for all $p\in [1,\infty]$.
\end{prop}
\begin{proof}
By  decomposing $X$ into the union of $B(x,\sqrt t)$ and the sets $B(x,2^k\sqrt t)\setminus B(x,2^{k-1}\sqrt t)$ for $k\ge 1$, one sees via  $(D)$ that
\begin{eqnarray}\label{quasi-conservative}
\int_X \frac1{V(x,\sqrt t)}\exp\lf\{-\frac{d^2(x,y)}{ct}\r\} \,d\mu(y) \le C(C_D).
\end{eqnarray}
The conclusion then follows from this and $(GLY_\infty)$.
\end{proof}

We will also need the following observation.

\begin{prop}\label{hk-gly}
Assume that the  doubling Dirichlet metric measure space $(X,d,\mu,\E)$ satisfies  $(UE)$.
Suppose that $(P_{2,\loc})$ and $(RH_{p})$ hold for some $p\in (2,\infty)$. Then $(GLY_{p})$ holds.
\end{prop}
\begin{proof}
Decompose the space $X$ into $B=B(y,2\sqrt t)$ and the sets $B(y,2^{k+1}\sqrt  t)\setminus B(y,2^k\sqrt t),$  $k\ge 1$.
Denote $B(y,2^{k+1}\sqrt  t)\setminus B(y,2^k\sqrt t)$ by $U_k(B)$.
By Theorem \ref{harmonic-pnorm}, $(UE)$ and \eqref{tly}, we see that
\begin{eqnarray*}
&& \||\nabla_x h_t(\cdot,y)|\|_{L^{p}(B)}\\
&&\quad\le \frac {CV(y,4\sqrt t)^{1/{p}}}{\sqrt t}\left(\fint_{B(y,4\sqrt t)}|h_t(x,y)|\,d\mu(x)+
t\left(\fint_{B(y,4\sqrt t)}\left|\frac{\partial }{\partial t}h_t(x,y)\right|^{{p}}\,d\mu(x)\right)^{1/{p}}\right)\\
&&\quad\le\frac {C}{\sqrt t}V(y,\sqrt t)^{1/{p}-1}.
\end{eqnarray*}

%Let $k\ge 1$ and consider $U_k(B)$.
Let $\{B_{k,j}=B(x_{k,j},\sqrt t/2)\}_j$ be a maximal  set of pairwise disjoint
balls with radius $2^{-1}\sqrt t$ in $B(y,2^{k+1}\sqrt  t)$.
Then it is easy to see that
 $$B(y,2^{k+1}\sqrt t)\subset \cup_{j}B(x_{k,j},\sqrt t)$$
 and
 $$\sum_{j}\chi_{4B_{k,j}}(x)\le C(C_D).$$
 Therefore, by applying Theorem \ref{infinite-har}, $(D)$,
$(UE)$, and \eqref{tly}, we conclude that
 \begin{eqnarray*}
 &&\int_{U_k(B)}|\nabla_x h_t(x,y)|^{p}\,d\mu(x)\\
 &&\quad \le \sum_{j:\,2B_{k,j}\cap U_k(B)\neq\emptyset}  \int_{2B_{k,j}}|\nabla_x h_t(x,y)|^{p}\,d\mu(x) \\
 &&\quad\le \sum_{j:\,2B_{k,j}\cap U_k(B)\neq\emptyset}\frac {C\mu(4B_{k,j})}{t^{p/2}}
 \left(\fint_{4B_{k,j}}|h_t(x,y)|\,d\mu(x)+
t\left(\fint_{4B_{k,j}}\left|\frac{\partial }{\partial t}h_t(x,y)\right|^{{p}}\,d\mu(x)\right)^{1/{p}}\right)^{p}\\
&&\quad\le \sum_{j:\,2B_{k,j}\cap U_k(B)\neq\emptyset}\frac {C\mu(4B_{k,j})}{t^{p/2}}V(y,\sqrt t)^{-p}\exp\left\{\frac{-c2^{2k}t}{t}\right\}\\
&&\quad\le C V(y,2^{k+2}\sqrt t) \frac {\exp\left\{-c2^{2k}\right\}}{t^{p/2} V(y,\sqrt t)^{{p}}}\le C V(y,\sqrt t) 2^{kQ}\frac {\exp\left\{-c2^{2k}\right\}}{t^{p/2} V(y,\sqrt t)^{{p}}}\\
&&\quad\le C\frac {\exp\left\{-c2^{2k}\right\}}{t^{p/2} V(y,\sqrt t)^{{p}-1}} .
 \end{eqnarray*}
 This together with the estimate on $\|\nabla_x h_t(\cdot,y)\|_{L^{p}(B)}$
from the beginning of the proof
%$(LY)$ and \eqref{tly}, we conclude that
% \begin{eqnarray*}
% &&\int_{U_k(B)}|\nabla_x h_t(x,y)|^{p_0}\,d\mu(x)\\
% &&\quad \le \sum_{j:\,2B_{k,j}\cap U_k(B)\neq\emptyset}  \int_{2B_{k,j}}|\nabla_x h_t(x,y)|^{p_0}\,d\mu(x) \\
% &&\quad\le \sum_{j:\,2B_{k,j}\cap U_k(B)\neq\emptyset}\frac {C\mu(4B_{k,j})}{\sqrt t}
% \left(\fint_{4B_{k,j}}|h_t(x,y)|\,d\mu(x)+
%t\left(\fint_{4B_{k,j}}|\frac{\partial }{\partial t}h_t(x,y)|^{{p_0}}\,d\mu(x)\right)^{1/{p_0}}\right)^{p_0}\\
%&&\quad\le \sum_{j:\,2B_{k,j}\cap U_k(B)\neq\emptyset}\frac {C\mu(4B_{k,j})}{\sqrt t}\mu(B(y,\sqrt t))^{-p}\exp\left\{\frac{-c2^{2k}t}{t}\right\}\\
%&&\quad\le C \mu(B(y,2^{k+2}\sqrt t)) \frac {\exp\left\{-c2^{2k}\right\}}{\sqrt t \mu(B(y,\sqrt t))^{{p_0}}}\\
%&&\quad\le C \mu(B(y,\sqrt t)) 2^{kQ}\frac {\exp\left\{-c2^{2k}\right\}}{\sqrt t \mu(B(y,\sqrt t))^{{p_0}}}\\
%&&\quad\le C\frac {\exp\left\{-c2^{2k}\right\}}{\sqrt t \mu(B(y,\sqrt t))^{{p_0}-1}} .
% \end{eqnarray*}
% From this,
allow us to deduce that there exists $\gz>0$ such that
 \begin{eqnarray*}
 &&\int_X|\nabla_x h_t(x,y)|^{p}\exp\left\{\gz d^2(x,y)/t\right\}\,d\mu(x)\le \frac{C}{t^{{p}/2}V(y,\sqrt t)^{{p}-1}},
 \end{eqnarray*}
 which completes the proof.
\end{proof}

The following conclusion follows via the argument in \cite[p. 944]{acdh} applied to our setting.

\begin{prop}\label{hk-pnorm}
Assume that the doubling Dirichlet metric measure space $(X,d,\mu,\E)$ satisfies $(UE)$.
Then $(GLY_{p})$  for some $p\in (2,\infty)$  implies
$(G_{p})$.
\end{prop}
%\begin{proof}
%Applying $(GLY_{p})$, the H\"older inequality, $(UE)$ and \eqref{quasi-conservative}, we deduce that for each $f\in L^{p}(X,\mu)$,
% \begin{eqnarray*}
%\||\nabla H_tf|\|_{p}^{p}&&\le \int_X\left(\int_X|\nabla_xh_t(x,y)||f(y)|\,d\mu(y)\right)^{p}\,d\mu(x)\\
%&&\le \int_X\left(\int_X|\nabla_xh_t(x,y)|^{p}\exp\left\{\gz d^2(x,y)/t\right\}V(y,\sqrt t)^{{p}-1}|f(y)|^{p}\,d\mu(y)\right)\\
%&&\quad \quad\times \left(\int_X\frac{1}{V(y,\sqrt t)}\exp\left\{-\frac{{p}'\gz d^2(x,y)}{pt}\right\}\,d\mu(y)\right)^{1/{p}'}\,d\mu(x)\\
%&&\overset{\eqref{quasi-conservative}}{\leq} \int_X\int_X|\nabla_xh_t(x,y)|^{p}\exp\left\{\gz d^2(x,y)/t\right\}V(y,\sqrt t)^{{p}-1}|f(y)|^{p}\,d\mu(y)\,dmu(x)\\
%&&\le \frac{C}{t^{{p}/2}}\|f\|_{p}^{p}.
% \end{eqnarray*}
% This yields $(G_{p})$.
%\end{proof}

\begin{rem}\rm If $(P_p)$ holds, then  one can also use the open-ended property of the reverse
H\"older inequality $(RH_p)$ (Lemma \ref{lem-open} below),
Theorem \ref{harmonic-pnorm}
 and the Hardy-Littlewood maximal operator to prove the fact that
$(RH_{p})$ ($p\in (2,\infty)$) yields $(G_{p})$. We will
not go through this argument  and leave the details to interested readers.
\end{rem}

\subsection{From parabolic equations to elliptic equations}
\hskip\parindent In this section, we show that  $(G_p)$
implies $(RH_p)$. We begin with an abstract reproducing formula for
harmonic functions.

\begin{lem}[Reproducing formula]\label{mvp}
Let $(X,d,\mu,\E)$ be a
doubling  Dirichlet metric measure space.
 Assume that  $(UE)$ holds,  Let
$\Phi\in\mathscr{S}(\rr)$ be an even function whose Fourier
transform $\hat{\Phi}$ satisfies $\supp \hat{\Phi}\subset [-1,1]$
and $\Phi(0)=1$ . Then if $u\in W^{1,2}(3B)$ is harmonic on $3B$, $B=B(x_0,r)$,
for each $0<t\le r$, $u=\Phi(t\sqrt{\mathcal{L}})u$ as functions in $W^{1,2}(B)$.
\end{lem}
\begin{proof}
Since $\Phi'(0)=0$, the function
 $\tilde\Phi(s):=s^{-1}\Phi'(s)\in \mathscr{S}(\rr)$ extends to an analytic function which satisfies  a
Paley-Wiener estimate with the same exponent as $\Phi$; see \cite{ru87} or  Appendix \ref{appendix-paley}.  By applying Lemma \ref{lem-fin-spe} to
the functions $t^{2\kz}\Phi(t)$, $\kz\in\mathbb{Z}_+$,
and $\tilde
\Phi$, we conclude that the operators
$(t^2\mathcal{L})^\kz\Phi(t\sqrt{\mathcal{L}})$ and
$(t^2\mathcal{L})^{-1/2}\Phi'(t\sqrt{\mathcal{L}})$ satisfy
\begin{equation}\label{finite-4-1}
\int_X\langle (t^2\mathcal{L})^\kz\Phi(t\sqrt{\mathcal{L}})f_1,f_2\rangle\,d\mu=0,
\end{equation}
and
\begin{equation}\label{finite-4-2}
 \int_X\langle (t^2\mathcal{L})^{-1/2}\Phi'(t\sqrt{\mathcal{L}})f_1,f_2\rangle\,d\mu=0,
\end{equation}
for all $0<t<d(E,F)$ with $E,F\subset X,$ $f_1\in L^2(E),$ and
$f_2\in L^2(F).$

Let $\psi$ be a Lipschitz cut-off function such that $\psi=1$ on
$\frac 83B$, $\psi=0$ outside $3B$. Let $\varepsilon\in (0,r/4)$.
For each $g\in L^2(\frac32 B)$ with support in $\overline{\frac 32B}$, by \eqref{finite-4-1} and Lemma \ref{bdd-spectral} we have
$$\Phi(\ez\sqrt{\mathcal{L}})g\in \mathscr{D}(\mathcal{L})$$
with support in $\overline{\frac 74 B}$. Since $\Phi(0) = 1$, we
have
$$1-\Phi(t\sqrt{\mathcal{L}})=-\int_0^t\sqrt{\mathcal{L}}\Phi'(s\sqrt{\mathcal{L}})\,ds,$$
which together with \eqref{finite-4-2} implies that
\begin{equation}\label{finite-4-3}
 \int_X\langle (t^2\mathcal{L})^{-1}(1-\Phi(t\sqrt{\mathcal{L}}))f_1,f_2\rangle\,d\mu=
 \int_0^t\int_X\langle (t^2\mathcal{L})^{-1}\sqrt{\mathcal{L}}\Phi'(s\sqrt{\mathcal{L}})f_1,f_2\rangle\,d\mu\,ds=0,
\end{equation}
for all $0<t<d(E,F)$ with $E,F\subset X,$ $f_1\in L^2(E),$ and
$f_2\in L^2(F).$
This together with Lemma \ref{bdd-spectral} implies that for each $t\le r$
\begin{equation}\label{finite-4-4}
(t^2\mathcal{L})^{-1}(1-\Phi(t\sqrt{\mathcal{L}}))\Phi(\ez\sqrt{\mathcal{L}})g
\in \mathscr{D}(\mathcal{L}),
\end{equation}
with support in $\overline{\frac {11}4 B}$. By this, the self-adjointness  of $\mathcal{L}$
and the fact that $u$ is harmonic on $3B$, we obtain that
\begin{eqnarray*}
\int_X\langle
(1-\Phi(tr\sqrt{\mathcal{L}}))u,\Phi(\ez\sqrt{\mathcal{L}})g\rangle\,d\mu&&=\int_X\langle u,(1-\Phi(tr\sqrt{\mathcal{L}}))\Phi(\ez\sqrt{\mathcal{L}})g\rangle\,d\mu\\
&&=\int_X\langle u\psi,(1-\Phi(tr\sqrt{\mathcal{L}}))\Phi(\ez\sqrt{\mathcal{L}})g\rangle\,d\mu\\
&&=r^2\int_{3B}\langle \nabla u,\nabla(r^2\mathcal{L})^{-1}(1-\Phi(tr\sqrt{\mathcal{L}}))\Phi(\ez\sqrt{\mathcal{L}})g\rangle\,d\mu\\
&&=0.
\end{eqnarray*}
Since $g$ is arbitrary, and by Lemma \ref{lem-l2-appro}
$\Phi(\ez\sqrt{\mathcal{L}})g\to g$ in $L^2(X,\mu)$ as $\varepsilon\to 0$, we
find that $(1-\Phi(tr\sqrt{\mathcal{L}}))u=0$ in $L^2(B).$ Hence
$u(x)= \Phi(tr\sqrt{\mathcal{L}})u(x)$
for a.e. $x\in \frac 32B$. Therefore,
 $u=\Phi(tr\sqrt{\mathcal{L}})u$ in $W^{1,2}(B)$ for each $t\le 1$. The proof is complete.
\end{proof}

%By the continuity of harmonic functions, Lemma \ref{l2.4}, we can
%conclude that
%$$u\equiv \Phi(tr\sqrt{\mathcal{L}})u$$
%on $B$ for each $t\le 1$, as desired.

{ \begin{rem}\label{sobolev-spectral-l2}\rm
{Notice that, for each $f\in L^2(X,\mu)$,
$\Phi(r\sqrt{\mathcal{L}})f\in W^{1,2}(X)$ and
$$\||\nabla \Phi(r\sqrt{\mathcal{L}})f|\|_2
= \|\sqrt{\mathcal{L}} \Phi(r\sqrt{\mathcal{L}})f\|_2
\le \frac Cr \|f\|_2, $$
see Lemma \ref{bdd-spectral}.}
%Indeed, by noticing that
%$$(1+R^2\mathcal{L})^{-1}=\int_0^\infty e^{-t(1+R^2\mathcal{L})}\,dt,$$
%and $\||\nabla H_t|\|_{2\to 2}\le \frac C{\sqrt t},$
%we can conclude via Lemma \ref{bdd-spectral} that
%\begin{eqnarray*}
%\||\nabla \Phi(R\sqrt{\mathcal{L}})f|\|_2&&= \||\nabla (1+R^2\mathcal{L})^{-1}  (1+R^2\mathcal{L})\Phi(R\sqrt{\mathcal{L}})f|\|_2\\
%&&\le \int_0^\infty \||\nabla e^{-t(1+R^2\mathcal{L})}  (1+R^2\mathcal{L})\Phi(R\sqrt{\mathcal{L}})f|\|_2\,dt\\
%&&\le \int_0^\infty \frac{Ce^{-t}}{\sqrt{tR^2}}\|f\|_2\,dt\\
%&&\le \frac CR \|f\|_2.
%\end{eqnarray*}
\end{rem}
}

\begin{cor}\label{cor-mvp}
Assume that the doubling Dirichlet metric measure space $(X,d,\mu,\E)$  satisfies  $(UE)$.  Let
$\Phi\in\mathscr{S}(\rr)$ be an even function whose Fourier
transform $\hat{\Phi}$ satisfies $\supp \hat{\Phi}\subset [-1,1]$ and $\Phi(0)=1$. Then
if $u\in W^{1,2}(X)$ is harmonic on $3B$, $B=B(x_0,r)$,
for each $0<t\le 1$,
$u$ equals $\Phi(tr\sqrt{\mathcal{L}})u$ as functions in $W^{1,2}(B)$.
\end{cor}
\begin{proof}
Notice that by Lemma \ref{mvp}, for each $0<t\le 1$,
$u(x)=\Phi(tr\sqrt{\mathcal{L}})(u\chi_{3B})(x)$,  a.e. $x\in B$.
On the other hand, by \eqref{finite-4-1}, we see that
$$\Phi(tr\sqrt{\mathcal{L}})(u\chi_{X\setminus 3B})(x)=0$$
on $B$, which allows us to conclude that  for each $0<t\le 1$,
$u=\Phi(tr\sqrt{\mathcal{L}})u$ in $W^{1,2}(B)$.
\end{proof}

The main result of this section reads as follows.

\begin{thm}\label{heat-har}
Assume that the  doubling Dirichlet metric measure space $(X,d,\mu,\E)$  satisfies $(UE)$. If $(G_{p_0})$ holds for some
 $p_0\in (2,\infty]$,  then $(RH_{p_0})$ holds.
\end{thm}
\begin{proof}
Let $\Phi\in\mathscr{S}(\rr)$ be an even function whose Fourier
transform $\hat{\Phi}$ satisfies $\supp \hat{\Phi}\subset [-1/2,1/2]$ and  $\Phi(0)=1$.
Then it follows that $\Phi^2\in\mathscr{S}(\rr)$ and $\supp \hat{\Phi^2}\subset [-1,1]$.
In the proof, for simplicity we denote $V(x,r)$ by $V_r(x)$.

{\bf Step 1.} {\it Boundedness of the spectral multipliers.}

{\bf Claim 1.} We first claim that, for each $p\in [1,2]$, there exists $C>0$ such that
$$\sup_{r>0}\|V_r^{1/p-1/2}\Phi(r\sqrt{\mathcal{L}})\|_{p\to 2}\le C.$$
By \cite[Proposition 4.1.1]{bcs15} and the fact that $\sup_{t>0}|\Phi(t)(1+t^2)^N|<\infty$, one has
\begin{eqnarray*}
\|V_r^{1/p-1/2}\Phi(r\sqrt{\mathcal{L}})\|_{p\to 2}&&=\|V_r^{1/p-1/2}\Phi(r\sqrt{\mathcal{L}})(1+r^2\mathcal{L})^NV_r^{1/2-1/p}
V_r^{1/p-1/2}(1+r^2\mathcal{L})^{-N}\|_{p\to 2}\\
&&\le C\|\Phi(r\sqrt{\mathcal{L}})(1+r^2\mathcal{L})^N\|_{2\to 2}\|V_r^{1/p-1/2}(1+r^2\mathcal{L})^{-N}\|_{p\to 2}\\
&&\le C\|V_r^{1/p-1/2}(1+r^2\mathcal{L})^{-N}\|_{p\to 2},
\end{eqnarray*}
where we choose $N>Q$ with $Q$ the number from $(D_Q)$. Notice that for any $f\in L^p(X,\mu)$ one has
\begin{eqnarray*}
\|V_r^{1/p-1/2}(1+r^2\mathcal{L})^{-N}f\|_2
&&\le C\int_0^\infty \left(\int_X\left|e^{-s}s^{N-1}V_r(x)^{1/p-1/2}e^{-sr^2\mathcal{L}}f(x)\right|^2\,d\mu(x)\right)^{1/2}\,ds\\
&&\le C\int_0^\infty e^{-s}s^{N-1}\| V_r^{1/p}e^{-sr^2\mathcal{L}}f\|^{1-{p}/{2}}_{\infty} \left(\int_X\left|e^{-sr^2\mathcal{L}}f(x)\right|^p\,d\mu(x)\right)^{1/2}\,ds\\
&&\le C\int_0^\infty e^{-s}s^{N-1}\left\| \frac{V_r}{V_{\sqrt sr}}\right\|^{1/p-{1}/{2}}_{\infty} \|f\|^{1-p/2}_p \|f\|_p^{p/2}\,ds\\
&&\le C\int_0^\infty e^{-s}s^{N-1}\frac{1}{(s\wedge1)^{Q/2(1/p-1/2)}} \|f\|_p\,ds\\
&&\le C\|f\|_p.
\end{eqnarray*}
Above in the third inequality, we used the fact that
\begin{eqnarray*}
|V_r^{1/p}(x)e^{-sr^2\mathcal{L}}f(x)|&& \le \frac{CV_r(x)^{1/p}}{V_{\sqrt{sr^2}}(x)^{1/p}}
\int_{X}\frac{V_{\sqrt{sr^2}}(x)^{1/p}}{V_{\sqrt{sr^2}}(x)}e^{-\frac{d(x,y)^2}{c\sqrt{sr^2}}}|f(y)|\,d\mu(y)\\
&&\le \frac{CV_r(x)^{1/p}}{V_{\sqrt{sr^2}}(x)^{1/p}}\|f\|_p
\left(\int_{X}\frac{1}{V_{\sqrt{sr^2}}(x)}e^{-\frac{d(x,y)^2}{c\sqrt{sr^2}}}\,d\mu(y)\right)^{(p-1)/p}\\
&&\le \frac{CV_r(x)^{1/p}}{V_{\sqrt{sr^2}}(x)^{1/p}}\|f\|_p.
\end{eqnarray*}
The claim is proved.

{\bf Claim 2.} For each $p\in (2,\infty]$, if $(G_p)$ holds, then there exists $C>0$ such that
$$\sup_{r>0}\|rV_r^{1-1/p}|\nabla\Phi(r\sqrt{\mathcal{L}})^2|\|_{1\to p}\le C.$$

By Claim 1 and \cite[Proposition 4.1.1]{bcs15} again, we have
\begin{eqnarray*}
\|rV_r^{1-1/p}|\nabla\Phi(r\sqrt{\mathcal{L}})^2|\|_{1\to p}&&= \|rV_r^{1-1/p}|\nabla\Phi(r\sqrt{\mathcal{L}})|
V_r^{-1/2}V_r^{1/2}\Phi(r\sqrt{\mathcal{L}})\|_{1\to p}\\
&&\le C\|rV_r^{1-1/p}|\nabla\Phi(r\sqrt{\mathcal{L}})|V_r^{-1/2}\|_{2\to p}\|V_r^{1/2}\Phi(r\sqrt{\mathcal{L}})\|_{1\to 2}\\
&&\le C\|rV_r^{1-1/p}|\nabla\Phi(r\sqrt{\mathcal{L}})|V_r^{-1/2}\|_{2\to p}\\
&&\le  Cr\||\nabla\Phi(r\sqrt{\mathcal{L}})|V_r^{1/2-1/p}\|_{2\to p}\\
&&\le  Cr\||\nabla(1+r^2\mathcal{L})^{-1}|\|_{p\to p} \|(1+r^2\mathcal{L})\Phi(r\sqrt{\mathcal{L}})V_r^{1/2-1/p}|\|_{2\to p}.
\end{eqnarray*}
Claim 1 together with a duality argument easily implies
$$\sup_{r>0}\|(1+r^2\mathcal{L})\Phi(r\sqrt{\mathcal{L}})V_r^{1/2-1/p}|\|_{2\to p}<\infty,$$
while $(G_p)$ implies that
$$\||\nabla(1+r^2\mathcal{L})^{-1}|\|_{p\to p}\le C\int_0^\infty \left\|\left|\nabla e^{-t(1+r^2)\mathcal{L}}\right|\right\|_{p\to p}\,dt \le \frac{C}{r}.$$
Combining these two estimate proves the second claim.

{\bf Step 2.} {\it Completion of the proof.}

 Suppose first that $u\in W^{1,2}(3B)$, $B=B(x_0,r)$, satisfies $\mathcal{L} u=0$ in $3B$.
By Claim 2 and the validity of $(G_{p_0})$, we then have
$$\left\|rV_r^{1-1/p_0}\left|\nabla\Phi(r\sqrt{\mathcal{L}})^2(u\chi_{3B})(\cdot)\right|\right\|_{p_0}\le C\|u\|_{L^1(3B)}.$$
The doubling condition together with Lemma \ref{mvp} implies that
\begin{eqnarray*}
\||\nabla u|\|_{L^p(B)}&&\le \frac{1}{rV_r(x_0)^{1-1/p_0}} \|rV_r^{1-1/p_0}|\nabla\Phi(r\sqrt{\mathcal{L}})^2(u\chi_{3B})(\cdot)|\|_{p_0}\le C\frac{1}{rV_r(x_0)^{1-1/p_0}}\|u\|_{L^1(3B)},
\end{eqnarray*}
i.e.,
\begin{eqnarray*}
\left(\fint_{B}|\nabla u|^{p_0}\,d\mu\right)^{1/p_0}\le \frac Cr \fint_{3B}|u|\,d\mu.
\end{eqnarray*}
Finally following the same argument as in {\bf Step 4} of proof of Theorem
\ref{harmonic-pnorm}, we see that $(RH_{p_0})$ holds,
which completes the proof.
\end{proof}

\begin{rem}\rm
 Using Claim 1 from Step 1 and \cite[Proposition 4.1.1]{bcs15} one can see that for each $r>0$
\begin{eqnarray*}
\|V_r\Phi(r\sqrt{\mathcal{L}})^2\|_{1\to \infty}&&= \|V_r\Phi(r\sqrt{\mathcal{L}})V_r^{-1/2}V_r^{1/2}\Phi(r\sqrt{\mathcal{L}}) \|_{1\to \infty}\\
&&\le \|V_r\Phi(r\sqrt{\mathcal{L}})V_r^{-1/2}\|_{2\to\infty}\|V_r^{1/2}\Phi(r\sqrt{\mathcal{L}})\|_{1\to 2}\\
&&\le  \|\Phi(r\sqrt{\mathcal{L}})V_r^{1/2}\|_{2\to\infty}\|V_r^{1/2}\Phi(r\sqrt{\mathcal{L}})\|_{1\to 2}\\
&&\le C.
\end{eqnarray*}
This together with Lemma \ref{mvp} then gives a simple proof of Proposition \ref{l2.3}.
\end{rem}

We can now finish the proofs of  Theorem \ref{main-har-heat-infty} and Theorem \ref{main-har-heat}, and their corollaries.

\begin{proof}[Proof of Theorem \ref{main-har-heat-infty}]
$(RH_\infty) \Longrightarrow (GLY_\infty) $ is contained in Proposition \ref{infinite-heat},
$(GLY_\infty) \Longrightarrow (G_\infty) $ is straightforward and is contained in Proposition \ref{bdd-infinite-heat} (see  \cite[p.919]{acdh}),
and $(G_\infty)\Longrightarrow (RH_\infty) $ is contained in Theorem \ref{heat-har}.

$(GBE)\Longrightarrow (GLY_\infty)$ follows from \cite[Lemma 3.3]{acdh} whose proof only requires $(D)$ and $(UE)$.
Notice that $(GLY_\infty)$ together with $(UE)$ implies $(LY)$, and therefore $(P_2)$; see \cite[Theorem 3.4]{bcf14}.
Using $(D)$ and $(P_2)$, $(GLY_\infty)\Longrightarrow (GBE)$ then also follows from the same proof of \cite[Lemma 3.3]{acdh}.
\end{proof}

\begin{proof}[Proof of Corollary \ref{main-cor-infity}]
Note that  $(P_2)$ implies $(P_\infty)$ and $(LY)$ (cf. \cite{sal,st3}), in particular $(P_{\infty,\loc})$ and $(UE)$.
\end{proof}

\begin{proof}[Proof of Corollary \ref{main-cor-manifold}]
If $(X,d,\mu)$ is a Riemannian manifold, then for any locally smooth function $v$ with bounded gradient $\nabla v$ on a ball $B$, $B=B(x_0,r)$,
it holds that
$$\fint_{B}|v-v_B|\,d\mu\le \fint_{B}\fint_B |v(x)-v(y)|\,d\mu(x)\,d\mu(y)\le
Cr\||\nabla v|\|_{L^\infty(B)}.$$
Since harmonic functions are locally smooth on a Riemannian manifold, this together with the assumption $(RH_\infty)$ implies that the conclusion of
Lemma \ref{lip-har-add} holds under the current assumptions. Therefore,  $(D)$ and $(UE)$ are enough to guarantee $(RH_\infty)\Longrightarrow (GLY_\infty)$
if $(X,d,\mu)$ is a Riemannian manifold, by the proof of Theorem \ref{main-har-heat-infty}.

The implications $(GLY_\infty)\Longrightarrow (G_\infty)$ and $(G_\infty)\Longrightarrow (RH_\infty)$ are contained in  Proposition \ref{bdd-infinite-heat} and Theorem \ref{heat-har}, respectively, requiring only $(D)$ and $(UE)$.

$(GBE)\Longrightarrow (GLY_\infty)$ is straightforward; see \cite[Lemma 3.3]{acdh}.
On the other hand, since under $(D)$ and $(GLY_\infty)$,
$(P_2)$ holds by \cite[Corollary 2.2]{CS1} (see also \cite[Theorem 3.4]{bcf14}),
 one can apply  \cite[Lemma 3.3]{acdh} to see that  $ (GLY_\infty)\Longrightarrow (GBE)$.
\end{proof}

\begin{proof}[Proof of Theorem \ref{main-har-heat}]
$(RH_{p}) \Longrightarrow (GLY_{p}) $ is contained in Proposition \ref{hk-gly},
$(GLY_{p}) \Longrightarrow (G_{p}) $ is explained in Proposition \ref{hk-pnorm},
and $(G_{p}) \Longrightarrow (RH_{p}) $ is contained in Theorem \ref{heat-har}.
\end{proof}

\begin{proof}[Proof of Corollary \ref{main-cor-finite}]
The conclusion holds,  since $(P_2)$ implies  $(P_{2,\loc})$ and $(UE)$ (cf. \cite{bcs15,sal,st2}).
\end{proof}

\section{Riesz transforms}\label{riesz}
\hskip\parindent  In this section we apply our results to the
Riesz transform. The following result was essentially proved by
Auscher, Coulhon, Duong and Hofmann \cite{acdh}; see \cite{bcf14}.
As we already said, $(D)$ together with $(P_2)$
guarantees $(R_p)$ for all $p\in (1,2]$,
see \cite{cd99}.

\begin{thm}\label{riesz-heat}
Assume that the doubling Dirichlet metric measure space $(X,d,\mu, \E)$ satisfies  $(P_2)$. Let $p_0\in (2,\infty)$. Then the following statements are equivalent:

(i) $(R_p)$ holds for all $p\in (2,p_0)$.

(ii) $(G_p)$ holds for all $p\in (2,p_0)$.

\end{thm}

First we record the open-ended character  of  condition $(RH_p)$.
\begin{lem}\label{lem-open}
Let  $(X,d,\mu, \E)$  be a doubling Dirichlet metric measure space.

(i)  If $(P_2)$ holds, then there exists $\varepsilon>0$, such that  $(RH_p)$ holds for each $p\in (2,2+\varepsilon)$.

(ii) If there exists $p_0\in (2,\infty)$ such that $(P_{p_0})$ and $(RH_{p_0})$ holds, then there exists $\varepsilon_1>0$
such that $(RH_p)$ holds for each $p\in (2,p_0+\varepsilon_1)$.
\end{lem}
\begin{proof}
(i) By the self-improving property of $(P_2)$ from  \cite{kz08} (see Appendix \ref{appendix-poincare}), we have that
there exists $0<\tilde \ez<1$ such that for each ball $B=B(x,r)$ and every $v\in W^{1,2}(B)$
$$\fint_{B}|v-v_{B}|\,d\mu\le
Cr\lf(\fint_{B}|\nabla v|^{2-\tilde \ez}\,d\mu\r)^{1/(2-\tilde \ez)},$$
where $C$ is independent of $B$ and $v$.
Therefore, %there exists $0<\tilde \ez<1$ such that
by Lemma \ref{l2.5}, for each
$u\in W^{1,2}(2B)$ satisfying $\mathcal{L} u=0$ in $2B$, $B=B(x_0,r)$, it holds that
\begin{eqnarray*}
\lf(\fint_B|\nabla u|^{2}\,d\mu\r)^{1/2}=\lf(\fint_B|\nabla (u-u_{2B})|^{2}\,d\mu\r)^{1/2}\le \frac {C}r \fint_{2B}|u-u_{2B}|\,d\mu\le
C\lf(\fint_{2B}|\nabla u|^{2-\tilde \ez}\,d\mu\r)^{1/(2-\tilde \ez)}.
\end{eqnarray*}
By applying the Gehring Lemma (cf. \cite{ge73,iw95}), we see that there exists $\varepsilon>0$ such that,
for each $p\in (2,2+\varepsilon)$,
\begin{eqnarray*}
\lf(\fint_{B(x_0,r/2)}|\nabla u|^{p}\,d\mu\r)^{1/p}&&\le
C\lf(\fint_{B(x_0,r)}|\nabla u|^{2-\tilde \ez}\,d\mu\r)^{1/(2-\tilde \ez)}\le C\lf(\fint_{B(x_0,r)}|\nabla u|^{2}\,d\mu\r)^{1/2}\\
&&\le \frac {C}r \fint_{2B}|u|\,d\mu.
\end{eqnarray*}
Applying the geometric doubling lemma, Lemma \ref{geometric-doubling}, as in
Step 4 of the proof of Theorem \ref{harmonic-pnorm},
we conclude  that $(RH_{p})$ holds for each $p\in (2,2+\varepsilon)$.

(ii) The second statement follows by noticing that  $(P_{p_0})$ implies $(P_{p_0-\hat\ez})$ for some $\hat\ez>0$ (cf. \cite{kz08} or Appendix \ref{appendix-poincare}).
This and $(RH_{p_0})$ imply
\begin{eqnarray*}
\lf(\fint_B|\nabla u|^{p_0}\,d\mu\r)^{1/p_0}=\lf(\fint_B|\nabla (u-u_{2B})|^{p_0}\,d\mu\r)^{1/p_0}\le \frac {C}r \fint_{2B}|u-u_{2B}|\,d\mu\le
\lf(\fint_{2B}|\nabla u|^{p_0-\hat\ez}\,d\mu\r)^{1/(p_0-\hat \ez)},
\end{eqnarray*}
if  $u\in W^{1,2}(2B)$ satisfies $\mathcal{L} u=0$ in $2B$, $B=B(x_0,r)$.

Using the Gehring Lemma once more gives the existence of $\varepsilon_1>0$
such that $(RH_p)$ holds for each $p\in (2,p_0+\varepsilon_1)$.
\end{proof}

%From Theorem \ref{main-har-heat} and Theorem \ref{riesz-heat}, we obtain the following conclusion, which implies that boundedness of the
%Riesz transform is an open-ended condition.
%
%\begin{thm}\label{Riesz}
%Assume that the metric measure space $(X,d,\mu, \E)$  satisfies $(D)$ and $(P_2)$.
%Assume that $(RH_{p_0})$  holds for some $p_0\in (2,\infty)$.
%Then there exists $\varepsilon>0$ such that  $(R_p)$ holds for each $p\in (1,p_0+\varepsilon)$.
%\end{thm}
%
%\comment{Such a $p_0$ always exists, see p.9}
%\begin{proof}
%
%Combining Theorem \ref{main-har-heat} with Theorem \ref{riesz-heat},
%we see that $(R_p)$ holds for each $p\in (1,p_0+\varepsilon)$ as claimed.
%\end{proof}

We can now prove Theorem \ref{riesz-main} by using Theorem \ref{main-har-heat} and the lemma above.
\begin{proof}[Proof of Theorem $\ref{riesz-main}$]
Notice that under the assumption of $(D)$, $(UE)$ and $(P_p)$, $(RH_p)$ or $(G_p)$ implies $(P_2)$; see  \cite[Corollary 2.8]{bf15} and
 \cite[Theorem 6.3]{bcf14}. The equivalence of $(RH_{p})$ and $(G_{p})$ follows from Corollary \ref{main-cor-finite},
and  we only need to prove that $(G_{p})\Longleftrightarrow (R_{p})$.

{\bf Step 1.} $(R_{p})\Longrightarrow (G_{p})$.

This is well known (cf. \cite{acdh}), but we recall the argument for the sake of completeness. Assume $(R_{p})$. By  analyticity of the heat semigroup
on $L^{p}(X,\mu)$ (cf. \cite{ste70})
$$\|\mathcal{L}^{1/2}e^{-t\mathcal{L}}\|_{p\to p}\le \frac{C}{\sqrt t}.$$
Therefore, we conclude via $(R_p)$ that
$$\||\nabla H_t|\|_{p\to p}=\||\nabla\mathcal{L}^{-1/2} \mathcal{L}^{1/2}H_t|\|_{p\to p}=
\||\nabla\mathcal{L}^{-1/2} \mathcal{L}^{1/2}e^{-t\mathcal{L}}|\|_{p\to p}\le \frac{C}{\sqrt t},$$
i.e., $(G_{p})$ holds.

{\bf Step 2.} $(G_{p})\Longrightarrow (R_{p})$.

Suppose that $(G_{p})$ holds.
According to Corollary \ref{main-cor-finite}, we know that $(RH_{p})$ holds.
By Lemma \ref{lem-open},  there exists $\varepsilon_1>0$
such that $(RH_q)$ holds for each $q\in (2,p+\varepsilon_1)$. This, together with
 Theorem \ref{main-har-heat} and  Theorem \ref{riesz-heat} above,
yields that $(R_q)$ holds  for each $q\in (2,p+\varepsilon_1)$,
and in particular, $(R_{p})$ holds,  as desired.
\end{proof}

%Based on the self-improvement of Poincar\'e inequalities due to
%Keith and Zhong \cite{kz08}, we obtain the following result
%that extends the main result
%of  \cite{ac05} from Riemannian manifolds to the metric measure space setting.
%
%\begin{cor}\label{riesz-c}
% Let $(X,d,\mu,\E)$ be a  metric measure space that satisfies $(D)$ and $(P_2)$.
%Then there exists $\varepsilon>0$, such that  $(R_p)$ holds for each $p\in (1,2+\varepsilon)$.
%\end{cor}
%\begin{proof} By the result of Keith and Zhong \cite{kz08}, we know that there exists $\tilde \ez>0$ such that
%the weak $L^{2-\tilde \ez}$-Poincar\'{e} inequality holds. Therefore, for each
%$u\in W^{1,2}(2B)$ satisfies $\mathcal{L} u=0$ in $2B$, $B=B(x_0,R)$, we have
%\begin{eqnarray*}
%\lf(\fint_B|\nabla (u-u_{2B})|^{2}\,d\mu\r)^{1/2}\le \frac {C}R \fint_{2B}|u-u_{2B}|\,d\mu\le
%C\lf(\fint_{2B}|\nabla u|^{2-\tilde \ez}\,d\mu\r)^{1/(2-\tilde \ez)}.
%\end{eqnarray*}
%By applying the Gehring Lemma (cf. \cite{ge73,iw95}), we see that there exists $\varepsilon>0$ such that
%\begin{eqnarray*}
%\lf(\fint_{B(x_0,R/2)}|\nabla u|^{2+\ez}\,d\mu\r)^{1/(2+\ez)}&&\le
%C\lf(\fint_{B(x_0,R)}|\nabla u|^{2-\tilde \ez}\,d\mu\r)^{1/(2-\tilde \ez)}\le C\lf(\fint_{B(x_0,R)}|\nabla u|^{2}\,d\mu\r)^{1/2}\\
%&&\le \frac {C}R \fint_{2B}|u|\,d\mu.
%\end{eqnarray*}
%Applying the geometric doubling lemma, Lemma \ref{geometric-doubling}, as in
%Step 4 of the proof of Theorem \ref{harmonic-pnorm},
%we can conclude  that $(RH_{2+\ez})$ holds.  Then Theorem \ref{Riesz}
%yields the desired conclusion and completes the proof.
%\end{proof}

Corollary \ref{cor-riesz-open} now easily follows from Lemma \ref{lem-open} and Theorem \ref{riesz-main}.
\begin{proof}[Proof of Corollary \ref{cor-riesz-open}]
This corollary follows by combining Theorem \ref{riesz-main} and Lemma \ref{lem-open}.
\end{proof}

%\begin{proof}[Proof of Theorem \ref{srh-riesz}]
%Assume that $(\widetilde{RH}_p)$ holds for some $p>2$.  By the Caccioppoli inequality Lemma \ref{l2.5}, we see that
% $(RH_p)$ holds, which together with Theorem \ref{main-har-heat} gives $(G_p)$.
%
%By \cite[Theorem 1.1]{bf15}, $(\widetilde {RH}_p)$ together with $(G_p)$ implies $(R_q)$ for all $q\in (2,p)$.
%That is,  $(\widetilde {RH}_p)$ implies $(R_q)$ for all $q\in (2,p)$.
%
%Since the reverse H\"older inequality $(\widetilde {RH}_p)$ is open-ended, there exists $\ez>0$ such that $(\widetilde {RH}_{p+\ez})$
%holds, which in particular implies $(R_p)$. The proof is complete.
%\end{proof}

\begin{rem}\rm
One can also find a characterization of boundedness of local Riesz transforms
via boundedness of the gradient heat semigroup for small time, $(G_p^{\loc})$, in \cite{acdh}.
We expect that the ideas of this paper can be employed to show that  $L^p$-boundedness of the local Riesz transform
is point-to-point  equivalent to $(G_p^{\loc})$ for each $p\in (2,\infty)$.
\end{rem}

\section{Sobolev inequalities and isoperimetric inequality}
\hskip\parindent In this section, following the central idea of \cite{jk,jky14} and using Theorem \ref{infinite-har}, we show that
$(RH_p)$ for $p>2$ yields a Sobolev inequality or an isoperimetric
inequality. Combining this and Theorem \ref{riesz-main},
we find a new necessary condition for quantitative regularity of harmonic
functions and heat kernels,  and for boundedness of Riesz transforms.

\subsection{Sobolev inequalities} \hskip\parindent
Recall the definition of the Sobolev inequality $(S_{q,p})$ given in Section \ref{SI}. In our setting, under $(D)$ and $(UE)$, $(S_{q,2})$ holds for some $q>2$ (see  Section \ref{harpo})
 and hence by H\"older  $(S_{q,p})$ holds for every $p\ge 2$. Here we are interested in the non-trivial range $p\in [1,2)$.
\begin{thm}\label{p-sobolev}
 Let $(X,d,\mu,\E)$ be a  Dirichlet metric measure space. Assume that  $(X,d,\mu,\E)$ satisfies $(D_Q)$,  $Q> 2$, and  that  $(UE)$ and $(P_{2,\loc})$ hold. Let  $p_0\in (2,\infty)$. Suppose that one of the mutually  { equivalent} conditions $(RH_{p_0})$,
$(GLY_{p_0})$, $(G_{p_0})$,   holds.
Then the Sobolev inequality $(S_{q,p})$ holds for all $p\in [\frac{p_0}{p_0-1},2]$ and $q\in [1,\frac{pQ}{Q-p})$.
%each Lipschitz function $v$ compactly supported in some ball $B$, $B=B(x_0,R)$.
\end{thm}
\begin{proof}
Let $p'_0=\frac{p_0}{p_0-1}$ and  $q\in [1,\frac{p'_0Q}{Q-{p'_0}})$. Then the conjugate exponent $q'$ of $q$ satisfies  $q'>\frac{Qp_0}{Q+p_0}$.
For any $B=B(x_0,r)$ and $g\in L^{\infty}(B)$, let $f\in W^{1,2}_0(B)$ be the solution to $\mathcal{L} f=g$ in $B(x_0,2r)$, see Lemma \ref{exitence-poisson}.
For a compactly supported Lipschitz function $h$ on $B$, we have
\begin{eqnarray*}
\left|\int_{B}h(x)g(x)\,d\mu(x)\right|
&&=\left|\int_Bh\mathcal{L}f\,d\mu\right|\\&&=\left|\int_B\langle\nabla h,\nabla f\rangle\,d\mu\right|\\
&&\le C\||\nabla h|\|_{L^{p'_0}(B)}\||\nabla f|\|_{L^{p_0}(B)}\\
&&=  C\||\nabla h|\|_{L^{p'_0}(B)}\left[V(x_0,r)\right]^{1/p_0}\left(\fint_B |\nabla f|^{p_0}\right)^{1/p_0}.
\end{eqnarray*}
Thus, by Theorem \ref{infinite-har},
$$
\left|\int_{B}h(x)g(x)\,d\mu(x)\right|
\le  C\||\nabla h|\|_{L^{p'_0}(B)}\frac{\left[V(x_0,r)\right]^{1/p_0}}{r}\left(\fint_{B(x_0,2r)}|f|\,d\mu+
r^2\left(\fint_{B(x_0,2r)}|g|^{q'}\,d\mu\right)^{1/q'}\right).$$
Since $p_0>2$, $\frac1{q'}<\frac{1}{p_0}+\frac{1}{Q}<\frac{1}{2}+\frac{1}{Q}$, and therefore we may apply Lemma \ref{l2.6},
which yields
$$\fint_{B(x_0,2r)}|f|\,d\mu\le C
r^2\left(\fint_{B(x_0,2r)}|g|^{q'}\,d\mu\right)^{1/q'},$$
and hence
$$\left|\int_{B}h(x)g(x)\,d\mu(x)\right|\le   C\||\nabla h|\|_{L^{p'_0}(B)}r\left[V(x_0,r)\right]^{1/p_0}
\left(\fint_{B(x_0,2r)}|g|^{q'}\,d\mu\right)^{1/q'}.
$$

Taking the supremum over all $g$ with $\|g\|_{L^{q'}(B)}\le 1$ yields
$$\left(\fint_B|h|^{q}\,d\mu\right)^{1/q}\le Cr\left(\fint_B|\nabla h|^{p'_0}\,d\mu\right)^{1/p'_0},$$
i.e. $(S_{q,p'_0})$. Finally
$(S_{q,p})$ follows by the H\"older inequality for every $p\in [p'_0,2]$ and $q\in [1,\frac{Qp}{Q-p})$, as desired.
\end{proof}

\subsection{Isoperimetric inequality}
\hskip\parindent In this section, we give an application of
Theorem \ref{main-har-heat-infty} to isoperimetric inequalities. The
following definition of
perimeter can be found in \cite{am02,mir} (see Appendix \ref{appendix-sobolev}).

For an open set $\Omega\subset X$,
denote by $\Lip(\Omega)$ ($\Lip_{\mathrm {loc}}(\Omega)$)
the space of all (locally) Lipschitz functions on $\Omega$, and by
$\Lip_0(\Omega)$ the space of all Lipschitz functions with compact support
in $\Omega$. Denote by $\mathscr{B}(X)$ the collection of all Borel sets in $X$.

\begin{defn}\label{d4.1}
Let $E\in \mathscr{B}(X)$ and $\Omega\subset X$ open. The perimeter of $E$
in $\Omega$, denoted by $P(E,\Omega)$, is defined by
\begin{equation}
 P(E,\Omega)=\inf\lf\{\liminf_{h\to\fz}\int_{\Omega} |\nabla v_h|\,d\mu:
\, \{v_h\}_h\subset \mathrm{Lip_\loc}(\Omega),
v_h\to \chi_E \ {\mathrm {in}} \ L^1_{\loc}(\Omega)\r\}.
\end{equation}
$E$ is a set of finite perimeter in $X$ if $P(E,X)<\fz$.
\end{defn}

The following proof is adapted from \cite{jky14}. We include it for completeness.
\begin{thm}\label{iso}
 Let $(X,d,\mu,\E)$ be a  Dirichlet metric measure space. Assume that  $(X,d,\mu,\E)$  satisfies $(D_Q)$,  $Q\ge 2$, and that  $(P_{\infty,\loc})$ and $(UE)$ hold. Suppose that one of the  mutually { equivalent}
conditions $(RH_{\infty})$, $(GLY_{\infty})$, $(G_\infty)$,
$(GBE)$, holds.
Then, for every bounded Borel set $E$ and every $x\in E,$
$$\mu(E)^{1-\frac 1 Q}\le C\frac{r}{\left[V(x,{r})\right]^{1/Q}}P(E,X).$$
where we choose  $r>\mathrm {diam} (E)$ such that $E\subset B(x,r)$.
\end{thm}
\begin{proof}
Let $E$ be a bounded Borel set in $X$. We can find a ball
$B=B(x,r)$ with center in $E$ and radius $r> \mathrm{diam} (E)$ such that
$E\subset\subset B$.

Consider the Poisson equation $\mathcal{L}f=\chi_E$ in $2B$. Then
there exists a solution $u\in W_0^{1,2}(2B)$ to the equation by Lemma \ref{l2.6}.  By
using $(RH_{\infty})$ and Theorem \ref{infinite-har}, we obtain that
for each $p>\frac{2Q}{Q+2}$, there exists
$C=C(C_D,C_{LS},C_P(1),p)>0$ such that, for almost every $y\in B$,
$$|\nabla f(y)|\le C\lf\{\frac{1}{r}\fint_{2B}|f|\,d\mu
+\sum_{j\le [\log_2
r]}2^j\lf(\fint_{B(y,2^j)}|\chi_{E}|^{p}\,d\mu\r)^{1/p}\r\}.$$

By Lemma \ref{l2.6} we have
\begin{eqnarray}\label{7.2}
\frac{1}{r}\fint_{2B}|f|\,d\mu\le Cr \lf(\fint_{B}|\chi_{E}|^{Q}\,d\mu\r)^{1/Q}\le \frac{Cr\mu(E)^{1/Q}}{\mu(B)^{1/Q}}.
\end{eqnarray}
Fix $p\in (\frac{2Q}{Q+2},Q)$. A direct calculation (cf. \cite[Proposition 4.1]{jky14}) shows that for any $y\in B$
\begin{eqnarray}\label{7.3}
\sum_{j\le [\log_2
r]}2^j\lf(\fint_{B(y,2^j)}|\chi_{E}|^{p}\,d\mu\r)^{1/p}\le C\frac{r}{\mu(B)^{1/Q}}\mu(E)^{1/Q}.
\end{eqnarray}

By the definition of perimeter, we may choose a sequence of
Lipschitz functions  $\{v_h\}_h\subset \mathrm{Lip_0}(B)$, $v_h\to
\chi_E \ {\mathrm {in}} \ L^1(B)$ such that
$$\lim_{h\to\fz}\int_{B} |\nabla v_h|\,d\mu=P(E,X).$$
As $f$ is a solution to the Poisson equation $\mathcal{L} u=\chi_E$ in $2B$, we then have for each $h\in \cn$,
$$\int_{2B}\nabla u\cdot \nabla v_h\,d\mu=\int_{2B}\chi_E v_h\,d\mu=\int_{E}v_h\,d\mu.$$
Since $\supp v_h\subset B$, by using the estimates \eqref{7.2} and \eqref{7.3}, and
passing $h$ to infinity, we obtain
\begin{eqnarray*}
\mu(E)&&=\lim_{h\to \fz}\|v_h\|_{L^1(B)}=\lim_{h\to \fz}\int_{2B}\nabla u\cdot \nabla v_h\,d\mu\le \lim_{h\to \fz}\|\nabla v_h\|_{L^1(B)}\||\nabla u|\|_{L^\fz(B)}\\
&&\le CP(E,X)\frac{r}{\mu(B)^{1/Q}}\mu(E)^{1/Q},
\end{eqnarray*}
which gives the conclusion and completes the proof.
\end{proof}

%\begin{cor}
% Let $(M,g)$ be a non-compac Riemannian manifold such that the associated metric measure space  satisfies $(D_Q)$,  $Q\ge 2$,  and $(UE)$.
% Suppose that one of the following { equivalent}
%conditions, $(RH_{\infty})$, $(GLY_{\infty})$, $(G_\infty)$,
%$(GBE)$, holds.
%Then, for every bounded Borel set $E$ and every $x\in E,$
%$$\mu(E)^{1-\frac 1 Q}\le C\frac{r}{\left[V(x,{r})\right]^{1/Q}}P(E,X).$$
%where we choose  $r>\mathrm {diam} (E)$ such that $E\subset B(x,r)$.
%\end{cor}

\begin{rem}\rm We remark that Theorem \ref{p-sobolev} and
Theorem \ref{iso} admit localisation. Since the arguments are
the same as for the global versions, we leave them
to interested readers.
\end{rem}

\section{Examples}\label{Ex}
\hskip\parindent  In this section, we  apply our results to several concrete examples of  interest. Notice
that since our assumptions are quite mild ($(D)$,  $(UE)$ and $(P_{2,\loc})$), our results
have broad applications. Below we will mainly concentrate on three different settings, and we refer the readers to
\cite{ags3,ams15,acdh,bg11,bm93,eks13,FKS82,kz12,VSC1992} for more examples.

\subsection{Riemannian metric measure spaces}\label{rmms}
\hskip\parindent Let us begin with some examples arising from
Riemannian geometry.

\medskip

{\bf Example 1.} Riemannian metric measure spaces with Ricci curvature bounded from below, i.e., $RCD^\ast(K,N)$ spaces, $K\in\rr$ and $N\in [1,\infty)$; see \cite{ams15,eks13}. Examples satisfying $RCD^\ast(K,N)$  include complete Riemannian manifolds with dimension not bigger than $N$
 and Ricci curvature not less than $K$, and complete Alexandrov spaces with dimension not bigger than $N$
 and curvature not less than $K$. An important fact is that  the $RCD^\ast(K,N)$ condition is stable under Gromov-Hausdorff convergence,
 which means that a Gromov-Hausdorff limit, of a sequence of manifolds satisfying $RCD^\ast(K,N)$,  satisfies also $RCD^\ast(K,N)$.

The $RCD^\ast(K,N)$ condition can be defined as follows; see \cite{ams15,eks13}. Let $(X, d,\mu,\E)$ be a Dirichlet metric measure space satisfying $\supp \mu= X$ and $V(x,r)\le Ce^{cr^2}$ for some $C,c>0$, $x\in X$ and each $r>0$. We call $(X, d,\mu,\E)$ a $RCD^\ast(K,N)$ space, where $K\in\rr$ and $N\in [1,\infty)$,
 if for all $f\in \mathscr{D}$ and each $t>0$, it holds that
\begin{equation}\label{bekn}
|\nabla H_tf(x)|^2+\frac{4Kt^2}{N(e^{2Kt}-1)}|\mathcal{L} H_tf(x)|^2\le e^{-2Kt}H_t(|\nabla f|^2)(x).
\end{equation}
Equivalently, $(X, d,\mu)$ is a $RCD^\ast(K,N)$ space if the Cheeger energy is a quadratic form
and $CD^\ast(K,N)$ condition holds; see \cite{ams15,eks13}.

Under the $RCD^\ast(K,N)$ condition, the (local) doubling condition was established in \cite{lv06,stm5}, and the (local) Poincar\'e inequality was established in \cite{raj2}. The doubling condition and Poincar\'e inequality have the same behaviour as in the case of classical smooth manifolds.

Gradient estimates for harmonic functions and heat kernels on $RCD^\ast(K,N)$ spaces were established in \cite{gm14,ji14,ji15,zz16}.
Our results recover these gradient estimates in a more obvious and simple way.
By the validity of the (local) doubling condition and (local) Poincar\'e inequality,
the definition \eqref{bekn} implies directly $(RH_{\infty})$, $(G_{\infty})$ and $(R_p)$ for all $p\in (1,\infty)$ if $K\ge 0$, and their local versions  if $K<0$.

\

{\bf Example 2.} On an $n$-dimensional conical manifold with compact
basis $N$ without boundary, $C(N):=\rr^+\times N$, let $\lambda_1$ be
the smallest nonzero eigenvalue of the Laplacian on the basis (see \cite{CW97,qzz13} for studies on the
first eigenvalue).
By a result of Li \cite{lh99}, the Riesz transform is bounded on
$L^p(C(N))$ for all $p\in (1,p_0)$ and not bounded for $p\ge p_0$,
where
$$p_0:=n\left(\frac n2-\sqrt{\left(\frac{n-2}{2}\right)^2+\lambda_1}\right)^{-1}$$
if $\lambda_1<n-1$ and $p_0=\infty$ otherwise; see also \cite{acdh}.

Therefore by Theorem \ref{riesz-main}, we see that $(RH_p)$ and
$(G_p)$ hold for all $p<p_0$.  Moreover, if $\lambda_1<n-1$, then
$(RH_p)$ and $(G_p)$ do not hold on $C(N)$ for any $p\ge p_0$.

\

{\bf Example 3.} By a result of Zhang \cite{zhang06}, it is known
that Yau's gradient estimate for harmonic functions is globally
stable under certain perturbations of the metric in the following
sense.

Let $M$ be an $n$-dimensional Riemannian manifold, $n>2$, suppose
that the volume of each ball $B(x,r)$ is comparable with $r^n$ for
any $x\in M$ and $r>0$, and assume that the $L^2$-Poincar\'e inequality holds.
If
$$Ric(x)\ge -\frac{\varepsilon}{1+d(x,x_0)^{2+\delta}}$$
for a fixed $x_0\in M$, $\delta>0$ and a sufficiently small
$\varepsilon>0$, then Yau's gradient estimate $(Y_{\infty})$ holds with $K=0$.
This holds, in particular, if $M$ is a
small compact perturbation of a manifold of dimension at least $3$
that has nonnegative Ricci curvature and maximum volume growth,
i.e., $V(x,r)\sim r^n$.

By Lemma \ref{rh-yau}, $(Y_{\infty})$ with $K=0$ is
equivalent to our $(RH_\infty)$. Therefore, by Theorem
\ref{main-har-heat-infty}, we see that $(RH_{\infty})$,
$(G_{\infty})$, $(GLY_\infty)$ and $(GBE)$ hold on these spaces.

\

{\bf Example 4.} Let $M$ be a  Riemannian manifold that is the union of a compact part, $M_0$, and a
finite number of Euclidean ends, $\mathbb{R}^n\setminus B(0,1)$, $n\ge 3$, each of which carries
the standard metric. The volume of balls in $M$ grows as
$V(x,r)\sim r^n$, in particular, volume is a doubling measure. Moreover, $(UE)$ holds as a consequence of the Sobolev inequality $(LS_q)$, $q>2$.
Notice also that, while $(P_{2,\loc})$ holds on $M$, $(P_p)$ does not hold for any $p\le n$; see \cite{cch06,cd99}.
%By \cite{cch06}, the Riesz transform is bounded on $L^p(M)$ for $p\in (2,n)$, and it is unbounded if $p\ge n$.
%Since $(R_p)$ implies $(G_p)$, Theorem \ref{main-har-heat} implies that $(RH_p)$  holds for all $p<n$.
By \cite{cch06}, the Riesz transform is bounded on $L^p(M)$ if and only if  $p\in (1,n)$.
Since $(R_p)$ implies $(G_p)$, Theorem \ref{main-har-heat} implies that $(RH_p)$  also holds if and only if $1<p<n$

Actually, it is rather easy to see that $(RH_p)$ holds on $M$ for $p<n$. Suppose that $u$ is a harmonic function on $2B$. There is nothing to prove if $r$ is small, since in this case, it holds
$$\||\nabla u|\|_{L^\infty(B)}\le \frac Cr \fint_{B}|u|\,d\mu.$$ If $r>>1$, then by applying the pointwise Yau's gradient estimate $(Y_{\infty})$ to
$u+\|u\|_{L^\infty(\frac 32B)}$, we conclude that
$$|\nabla u(x)|\le \frac{C}{1+\dist(x,M_0)}\left(u(x)+\|u\|_{L^\infty(\frac 32B)}\right)$$
 for each $x\in B$, which implies, if $p<n$,
\begin{eqnarray*}
\left(\fint_{B}|\nabla u|^p\,d\mu\right)^{1/p}\le C\|u\|_{L^\infty(\frac 32B)}\left(\fint_{B}\frac{1}{(1+\dist(x,M_0))^p}\,d\mu\right)^{1/p}\le \frac Cr\|u\|_{L^\infty(\frac 32B)}\le \frac Cr\fint_{2B}|u|\,d\mu.
\end{eqnarray*}

Notice that, however, $(\widetilde{RH}_p)$ does not hold on $M$ for any $p>2$. Indeed, if $(\widetilde{RH}_p)$ holds,  then we have
$$\lf(\int_B|\nabla u|^{p}\,d\mu\r)^{1/p}\le
C\mu(B)^{1/p-1/2}\lf(\int_{2B}|\nabla u|^{2}\,d\mu\r)^{1/2}, $$
if $u$ is harmonic on $2B$. By  \cite[Theorem 2.1]{lt92}, there exists  a bounded, non-constant harmonic function $u$
with finite Dirichlet energy. Applying the above estimate to $u$ and letting the radius of $B$ tend to infinity, we see that $\||\nabla u|\|_{p}=0$,
which cannot be true.

\

{\bf Example 5.} Consider a complete, non-compact, connected Riemannian
manifold $M$. Suppose that a finitely generated discrete group $G$
acts properly and freely on $M$ by isometries, such that the orbit space
$M_1 = M/G$ is a compact manifold. In other words, $M$ is a Galois covering
manifold of the compact Riemannian manifold $M_1$, with deck transformation
group (isomorphic to) $G$. The most simple example is $M = \mathbb{R}^D$ endowed with a Riemannian metric
which is periodic under the standard action of $G = \mathbb{Z}^D$ by translations.

Assuming that $G$ has polynomial volume growth of some order
$D \ge 1$, Dungey \cite[Theorem 1.1]{dng04b} (see also \cite{dng04a}) showed that
$(GLY_\infty)$ holds on $M$.  Our Theorem \ref{main-har-heat-infty} then implies
that $(RH_{\infty})$, $(G_{\infty})$, $(GLY_\infty)$ and $(GBE)$ hold on these spaces.
Indeed, by using the group structure of $M$, it is relative easier to show that $(RH_{\infty})$ holds on $M$;
 see Appendix \ref{appendix-covering}.

\subsection{Carnot-Carath\'eodory spaces}
\hskip\parindent A large class of examples that our results can be applied to
come from Carnot-Carath\'eodory spaces;
we refer the readers to \cite{bg13,bg11,FS86,grom96,hak,js87,nsw85} for
background and recent developments.

Let $M$ be a smooth, connected manifold and $\mu$ a Borel measure.
Let $\{X_i\}_{i=1,\cdots,m}$ be Lipschitz
vector fields on $M$, with real coefficients. The ``carr\'e du champ" operator $\Gamma$ is given as
$$\Gamma(f):=\sum_{i=1}^m |X_if|^2$$
for each $f\in C^\infty(M)$, where the corresponding Dirichlet form $\int_M \sum_i X_if X_ig\,d\mu$ generalises  a second-order diffusion
operator $L$.

A tangent vector $v\in T_xM$ is called subunit for $L$ at $x$ if $v = \sum_{i=1}^m a_i X_i(x)$, with
$\sum_{i=1}^m a_i^2\le 1$;  see  \cite{FP86}.   A Lipschitz
curve $\gamma: [0, T]\mapsto M$ is called subunit for $L$ if $\gamma'(t)$ is subunit for $L$ at $\gamma(t)$ for a.e. $t\in [0,T]$.
The subunit length of $\gamma$, $\ell(\gamma)$, is given as $T$. We assume that
for any $x,y\in M$, there always exists a subunit curve $\gamma$ joining $x$ to $y$.
The Carnot-Carath\'eodory distance then is defined as
$$d_{cc}(p,q):=\inf\{\ell(\gamma):\,\gamma\ \mbox{is a subunit curve joining}\ p\ \mbox{to} \ q\}.$$
Notice that for any $x,y\in M$, the Carnot-Carath\'eodory distance
$d_{cc}(x,y)$ is the same as $d(x,y)$
induced from the Dirichlet forms; see \cite{bg11,cks87}.

Once again, our results can be applied to this setting as soon as a
(local) doubling condition and an (local)
$L^2$-Poincar\'e inequality are available. Notice that all Carnot groups
equipped with the Lebesgue measure and the natural vector fields
satisfy an $L^2$-Poincar\'e inequality; see \cite{hak} for instance.

For general vector fields satisfying the H\"ormander condition (cf.
\cite{FP86,hak,js87,nsw85}), it is known that the doubling condition
and $L^2$-Poincar\'e inequality hold locally with constants depending
on the balls under consideration, which is not sufficient in order
to apply our
results. However,  the potential estimates for the Poisson equation
from Section 3, Theorem \ref{infinite-poisson} and Theorem
\ref{infinite-har}, still work in these settings.

As we recalled in the introduction, by Theorem \ref{main-har-heat-infty},
$(RH_\infty)$, $(GLY_\infty)$, $(G_\infty)$  and $(GBE)$ hold on
any Lie groups of polynomial growth (cf. \cite{Al92,sal3}),
and more generally, on sub-Riemannian manifolds satisfying
Baudoin-Garofalo's curvature-dimension inequality $CD(\rho_1,
\rho_2, \kappa, d)$ (cf. \cite{bg11}) with $\rho_1 \ge 0$, $\rho_2 >
0$, $\kappa\ge 0$ and  $  d\ge  2$.

Examples satisfying $CD(\rho_1, \rho_2, \kappa, d)$
include all Sasakian manifolds whose horizontal Webster-Tanaka-Ricci
curvature is bounded from below, all Carnot groups with step two,
and wide subclasses of principal bundles over Riemannian manifolds
whose Ricci curvature is bounded from below; see \cite[Section
2]{bg11}.

\subsection{Degenerate (sub-)elliptic/parabolic equations}
\hskip\parindent  Our results are also applicable to degenerate
(sub-)elliptic/parabolic equations on Euclidean spaces. It is of course also
possible to extend these degenerate equations to
general metric measure spaces. For instance, one may consider a
Dirichlet form given by
$$\dint_{X}\langle\nabla f(x)\cdot \nabla g(x)\rangle w(x)\,dx,$$
where  $\la \nabla f, \nabla g\ra$ is the natural energy density of energy on
an infinitesimally Hilbertian space $(X, d,\mu)$ (cf. \cite{ags3}),
or $\nabla$ is the Cheeger differential operator (cf. \cite{ch}), and
$w$ is a suitable weight.

We focus on  degenerate elliptic/parabolic equations on
Euclidean spaces, and we refer the reader to
\cite[p.133]{bm93} and \cite{FGW91} for more examples of degenerate
(sub-)elliptic equations.

Let  $w$ be a Muckenhoupt $A_2$-weight or a qc-weight, where by qc-weight we mean that
$w=|J_f|^{1-\frac{2}{n}}$, where $|J_f|$ denotes the Jacobian of a
quasiconformal mapping $f$ on $\rn$; see \cite{bm93,FKS82}. Let
$A:=(A_{ij}(x))_{i,j=1}^n$ be a symmetric matrix of functions on
$\rn$ satisfying the \emph{degenerate ellipticity condition},
namely, there exist constants $0<\lz\le \Lambda<\fz$ such that, for
all $\xi\in\rn$,
\begin{eqnarray}\label{degenerate EC1}
\lz w(x)|\xi|^2\le \langle A \xi,\,\xi \rangle \le \Lambda w(x)|\xi|^2.
\end{eqnarray}
For all $f,\,g\in C^\infty_c(\rn)$, consider the Dirichlet form
given by
\begin{eqnarray}\label{sesquilinear form}
\dint_{\rn}A(x)\nabla f(x)\cdot \nabla g(x)\,dx.
\end{eqnarray}

Then the intrinsic distance $d$ and the usual
Euclidean metric $d_E$ are comparable, that is  $d\sim d_E$. On the space $(\rn,d_E,w(x)\,dx)$, the doubling condition is a
well-known property of a Muckenhoupt weight or  follows from properties of
quasiconformal mappings, and an $L^2$-Poincar\'e inequality was
established in \cite{FKS82}. From this, one can deduce that a doubling
condition and a weak $L^2$-Poincar\'e inequality,
i.e.
\begin{equation} \fint_{B}|f-f_B|\,d\mu\le
Cr\lf(\fint_{cB}|\nabla f|^2\,d\mu\r)^{1/2},
\end{equation}
for all  $B=B(x,r)$, for some constant $c\ge 1$,  hold on $(\rn,d,w(x)\,dx)$. By using
the results from \cite[Section 9]{hak}, together with the fact that
$(\rn,d,w(x)\,dx)$ is geodesic, we see that $(\rn,d,w(x)\,dx)$
supports a scale-invariant $L^2$-Poincar\'e inequality. %The Sobolev
%space $W^{1,2}(w,\,\rn,)$ actually coincides with $H^1(w,\,\rn)$, which is
%defined as the closure of $C_c(\rn)$ with respect to the \emph{norm}
%$$\|f\|_{H^1(w,\,\rn)}:=\|f\|_{L^2(w,\,\rn)}+\|\nabla f\|_{L^2(w,\,\rn)}.$$

Therefore, our results are applicable to $(\rn,d,w(x)\,dx)$ as well.
We would like to point out that Caffarelli and Peral \cite{cp98} established a $W^{1,p}$-estimate for elliptic equations in divergence form
by using the technique of approximation to a reference equation. Shen \cite{shz05}
employed the techniques from \cite{cp98} to prove the equivalence of $(R_p)$ and
$(RH_p)$, for {\it uniformly elliptic} operators of divergence form on $\rn$.
Recently, for degenerate elliptic operators with $A$ being  complex-valued and satisfying suitable
weighted condition,  Cruz-Uribe et al. \cite{CMR15} obtained the boundedness of the Riesz transform
in an open interval containing 2.

For degenerate equations satisfying condition \eqref{degenerate EC1} for some
$A_2$-weight or qc-weight, although the heat kernel and harmonic
functions are known to be H\"older continuous (cf.
\cite{bm93,st1,st2,st3}), harmonic functions and the heat kernel
are not Lipschitz  in general; see the examples from the
introductions of \cite{ji14,krs} for instance.

Moreover, given an explicit $p>2$, we do not even know if the
gradients of harmonic functions or heat kernels are locally $L^p$-integrable.
Indeed, in view of Corollary \ref{cor-riesz-open} and
Theorem \ref{riesz-main}, we see that there exists $\varepsilon>0$
(implicit), such that $(RH_p)$ and $(G_p)$ hold for $p\in
(2,2+\varepsilon)$. However, for an explicitly given $p>2$, the
assumption $w\in A_2$ alone is not sufficient for quantitative
$L^p$-regularity of harmonic functions or heat kernels, in view of Theorem
\ref{p-sobolev}. Since if $(RH_p)$ or $(G_p)$ holds for some $p>2$,
then one has a Sobolev inequality $(S_{p',q})$ for some $q> p'$ on
$(\rn,d,w(x)\,dx)$, and it is well-known that $w\in A_2$ is not
sufficient to guarantee such a Sobolev inequality for (small) $p'.$
It would be of great interest to know how to quantify
the regularity of harmonic functions and heat kernels in this case.

\medskip

{Finally we apply our results to the simplest possible form of  degenerate elliptic operators in  dimension one.
The correspond to the Dirichlet form
\begin{eqnarray*}
Q_\alpha (f,g) = \dint_{{\rr}}|x|^\alpha f'(x)\cdot  g'(x)\,dx
\end{eqnarray*}
for some $\alpha>0$ on $L^2(\rr, {|x|^\alpha} dx)$. { The corresponding intrinsic distance coincides with the Euclidean distance}. Note that the weight $\omega_\alpha(x)=|x|^\alpha$ belongs
to Muckenhoupt  class $A_p$ only  if $\alpha +1< p $.  Observe that in the range $0 \le \alpha <1$ any harmonic function for the operator discussed here is of the form
{$a \,\mathrm{sign}(x)  |x|^{1-\alpha} +b$} for some constants $a,b \in \rr$. Hence  a simple calculation shows that $(RH_p)$
holds if and only if $\alpha(1-p)>-1$, i.e. $p<(1+\alpha)/\alpha$. It follows from examples and the results obtained in \cite{RSi6a} that the $L^2$-Poincar\'e inequality holds if and only if $\alpha <1$ or equivalently if $\omega_\alpha \in A_2$. Now it follows from Theorem~\ref{riesz-main}
that   for $0<\alpha <1$, $(R_p)$ holds also if and only if $p<(1+\alpha)/\alpha$. This range of validity of  $(R_p)$ was first obtained in \cite[Theorem 5.3]{hasi09}
(see also \cite[Section 6.3]{hasi09}).  Theorem \ref{riesz-main} yields  this result avoiding relatively tedious
calculations.
We point out that  the heat kernel and harmonic functions are usually discontinuous (at the point $x=0$)
for $\alpha\ge1$, see \cite {ersz}. We refer the reader to \cite{hasi09} for more about the Riesz transform.

%Next on the space $L^2(\rr, |x|^{d-1}dx)$ define quadratic form $\widetilde{Q_d}$ by the formula
%$$
%\widetilde{Q_d}( f ,g) = \int_{ \rr} f'(r) g'(r)|r|^{d-1} dr.
%$$
%%and that $ we denote the operator corresponding this form.
%Now set $1/d=1-\alpha/2$ and for function $f \colon {\rr}
% \to  \rr$ we put
% \begin{equation*}
%\tilde{f}(x)= \left\{ \begin{array}{ll}
%f(x^d)  & \mbox{if} \quad x  \ge 0\\
%f(-|x|^d) &    \mbox{if} \quad x <0 .
% \end{array}
%    \right.
%\end{equation*}
% Note that
%$d\|f\|_{L^p({\rr},dx)}^p=\|\tilde f\|_{L^p({\rr},|x|^{d-1}dx)}^p$ and
%that
%$$
%\widetilde{Q_d}(\tilde f ,\tilde g) = \int_{ \rr}\tilde f'(r)\tilde g'(x)|x|^{d-1} dx=\frac{1}{d}\int_{\tilde \rr} |x|^{\alpha}   f'(x)g'(x) dx=Q_\alpha(f,g),
%$$
%where the quadratic form $Q_\alpha$ is defined on $L^p({\rr},dx)$. This shows that our discussion of the degenerate operators corresponding to $Q_\alpha$ can be also applied to the one dimensional operators described by
%\eqref{sesquilinear form} and corresponding to $\widetilde{Q_d}$, see also Section 6.3 of \cite{hasi09}.
}

\appendix

\section{Appendix}\label{appendix}
%\addcontentsline{toc}{section}{Appendix} \hskip\parindent
\subsection{Sobolev spaces on domains}\label{appendix-domain}
\hskip\parindent Let $U\subset X$ be an open set.
The local Sobolev space $W_{\loc}^{1,2}(U)$
is defined to be the collection of all  functions $f\in L^2_\loc(U)$, such that for any compact set $K\subset U$
there exists $F_K\in\D$ satisfying $f=F_K$ a.e. on $K$; see \cite[Definition 2.3]{GSC}.
Notice that bounded closed sets are compact (cf. \cite[Theorem 2.11]{GSC}).
So we can find a sequence of compact sets $\{U_j\}_{j=1}^\infty$ such that
$U_j$ is contained in the interior of $U_{j+1}$
and $\cup_{j=1}^\infty U_j=U$. Write $F_j$ for $F_{U_j}$.

By the locality of the Dirichlet form $\E$, for a given $f\in W_{\loc}^{1,2}(U)$,
one has that for any $j\in\cn$ and measurable set $E\subset U_j$
$$\int_{U_j}\chi_E\,d\Gamma(F_j,F_j)=\int_{U_{j}}\chi_{E}\,d\Gamma(F_{j+1},F_{j+1}).$$
This implies that we may consistently define $|\nabla f |$ on $U$ by setting
$$|\nabla f|^2=\frac{\,d\Gamma(F_j,F_j)|_U}{\,d\mu}$$
on $U_j$.

Now for each $p\ge 2$, the Sobolev space  $W^{1,p}(U)$, defined as the collection of all  functions $f\in W_{\loc}^{1,2}(U)$ satisfying $f,\,|\nabla f|\in L^p(U)$, is well defined.

\subsection{Equivalence of differently defined Sobolev spaces}\label{appendix-sobolev}
\hskip\parindent
There are several different types of Sobolev spaces on metric measure spaces: Haj\l asz Sobolev spaces \cite{ha96},
Newtonian Sobolev spaces \cite{sh}, Cheeger's Sobolev spaces \cite{ch} etc. We refer the readers to
the monographs by Heinonen, Koskela, Shanmugalingam and Tyson \cite{hkst} and
A. Bj\"orn and J. Bj\"orn \cite{bb10} for these studies.

We first recall the definition of Haj\l asz Sobolev spaces.
\begin{defn}
Let $1\le p\le \infty$. Given a measurable function $u$  on $X$, a non-negative measurable
function $g$ on $X$ is called a \textit{Haj\l asz gradient} of $f$
if there is a set $E\subset X$ with $\mu(E)=0$
such that for all $x,\ y\in X\setminus E$,  %it holds true that
\begin{equation}\label{HG}
|f(x)-f(y)|\leq d(x,y)[g(x)+g(y)].
\end{equation}

The Haj\l asz-Sobolev space $M^{1,p}(X)$ is defined to be the set of
all functions $f\in L^p(X,\mu)$ that have a  Haj\l asz gradient
$g\in L^p(X,\mu)$. The norm on this space is given by % and normed by
$$\|f\|_{M^{1,p}(X)}:=\|f\|_{p}+\inf_g\|g\|_{p},$$
where the infimum is taken over all Haj\l asz gradients of $f$.
\end{defn}

It is known that $M^{1,p}(X)$ embedded continuously into
the Newtonian Sobolev space $N^{1,p}(X)$, which was introduced by Shanmugalingam;
 see \cite[Theorem 4.8]{sh}. Notice that the embedding $M^{1,p}(X)\hookrightarrow N^{1,p}(X)$ actually holds
 on any metric measure space $(X,d)$ equipped with a Borel regular measure $\mu$;
 see \cite[Theorem 1.3]{jsyy15}.

Under the requirements of doubling and Poincar\'e inequality $(P_{2})$, it is known that
$$W^{1,2}(X)=N^{1,2}(X)=W^{1,2}(X),$$
and that Lipschitz functions are dense in these spaces;
see \cite{sh,kz12}.
Moreover, for any function $f\in W^{1,2}(X)$,
the square root of its density energy, $|\nabla f|$, equals its approximate pointwise Lipschitz constant,  $\mathrm{apLip} f$; see \cite[Theorem 2.2]{kz12}.  Here,
\begin{eqnarray*}
 \mathrm{apLip} f(x):=\inf_{A}\limsup_{y\in A:\,d(x,y)\to 0}\frac {|f(x)-f(y)|}{d(x,y)},
\end{eqnarray*}
where the infimum is taken over all Borel sets $A\subset X$ with a point of density
at $x$.

We note that, without the validity of Poincar\'e inequality, the above conclusions are not true in general.
In particular, it may happen that, $M^{1,p}(X)\subsetneq N^{1,p}(X)$, see \cite{hkst,sh} for instance.

Nevertheless, for locally Lipschitz functions $\phi$, assuming only doubling but not Poincar\'e, one still has that $|\nabla \phi|=\mathrm{Lip}\,\phi$, a.e.; see \cite[Remark 2.20]{GSC} and
\cite[Theorem 2.1]{kz12}. Therefore, our perimeter $P(E;\Omega)$ (see Definition \ref{d4.1}) coincides with
that from \cite{am02,mir}.

\subsection{Self-improving property of Poincar\'e inequalities}\label{appendix-poincare}
\hskip\parindent  The self-improving property of Poincar\'e inequality was obtained by Keith and Zhong \cite{kz08}
on a complete doubling metric space.
In our setting, together with the fact $(X,d)$ is geodesic, their result gives: if for some $p\in (1,\infty)$ it holds for each ball $B=B(x_0,r)$ and every Lipschitz function $f$ that
$$
\fint_{B}|f-f_B|\,d\mu\le
Cr\lf(\fint_{B}|\mathrm{Lip} f|^p\,d\mu\r)^{1/p},$$
then there exists $q\in (1,p)$ such that the above inequality holds with $p$ replaced by $q$.

From the previous subsection, we see that for each locally Lipschitz function $\phi$ it holds $|\nabla \phi|=\mathrm{Lip}\,\phi$.
This together with a density argument implies that, if $(P_p)$ holds for some $p\ge 2$,
then there exists $1<q<p$ such that for any ball $B=B(x_0,r)$ and any $g\in W^{1,p}(B)$ it holds that
$$
\fint_{B}|g-g_B|\,d\mu\le
Cr\lf(\fint_{B}|\nabla g|^q\,d\mu\r)^{1/q},$$
where $C$ is independent of $B$ and $g$.

Notice that however $(P_\infty)$ does not have the self-improving property,  see \cite{dsw12,djs12}.

%In the above definitions of Poincar\'e type inequalities, we assume that $p\ge 2$ to avoid the technicalities of
%defining $W^{1,p}(U)$  for $1\le p<2$. Note however that we can extend this definition to all $p$ by assuming that for $p<2$  test functions $f$ range in $W^{1,2}$. This approach allows us to define $(P_p)$ and
%$(P_{p,\loc})$ for all $p\in [1,\infty]$, which will be useful in Section \ref{riesz}.

\subsection{Paley-Wiener Estimate}\label{appendix-paley}
\hskip\parindent Suppose that $F\in \mathscr S (\rr)$ satisfies $\supp \hat{F} \subset [-1,1]$ and $F(0)=0$, where
$\hat{F}$ denotes the Fourier transform of $F$. By the Paley-Wiener theorem (see \cite{ru87}), we can extend $F$ to analytic function,
$F(z)$, on $\mathbb{C}$, and so also $G(z)=F(z)/z$. It holds obviously that for $|z| \ge 1$,
$|G(z)| \le |F(z)|$. Note that Paley-Wiener¡¯s estimate
$( |F(z)| \leq C_N (1 + |z|)^{-N} e^{B|\text{Im}(z)|}$ for $z\in[-B,B]$
 is a condition only for large $|z|$, so
$|G(z)|$ satisfies the same estimates as $|F(z)|$, which implies that  $\hat{G} \subset [-1,1]$.

From the above discussion, we see that if $F$ is a Schwartz function, then $G$, $\hat{F}$, $\hat{G}$ are all Schwartz functions.
It is easy to note that $d_t \hat{G}(t)=\hat{F}(t)$, so if $\hat{F} \subset [-1,1]$ then $\hat{G}(t)$
is constant on both half lines $(-\infty,-1]$ and $[1,\infty)$. But $\hat{G}(t)$ is a Schwartz function,
so it has to converge to zero at the ends, which means that $\supp\hat{G} \subset [-1,1]$.

\subsection{Gradient estimates on covering manifolds}\label{appendix-covering}
\hskip\parindent Let us provide a proof of $(RH_\infty)$ on covering manifolds in {\bf Example 5} from Section \ref{Ex}.
\begin{thm}
Let $M$ be a complete, non-compact, connected Riemannian
manifold. Suppose that a finitely generated discrete group $G$
acts properly and freely on $M$ by isometries, such that the orbit space
$M_1 = M/G$ is a compact manifold. Assume that $G$ has polynomial volume growth of some order
$D \ge 1$, Then $(RH_\infty)$ holds on $M$.
\end{thm}

Let us observe that, due to the group action and the polynomial volume growth of $G$,
 $(D)$ and $(P_2)$ hold on $M$; see \cite{sal}.  Moreover, by Yau's gradient estimate (cf. \cite{chy,ya75}), $(RH_{\infty,\loc})$ holds, i.e., for each $r_0>1$, there exists $C(r_0)>0$ such that
if  $u$ is harmonic in  $2B$, $B=B(x_0,r)$, $r<r_0$, it holds that
$$\||\nabla u|\|_{L^\infty(B)}\le \frac{C(r_0)}{r}\fint_{2B}|u|\,d\mu. \leqno(RH_{\infty,\loc})$$

\begin{proof} Since $(RH_{\infty,\loc})$ holds, we only need to prove $(RH_\infty)$ for balls of large radii.

Since $(D)$ and $(P_2)$ hold, by applying Proposition \ref{l2.3add}  there exist $C>0$ and $\gz\in (0,1)$, such that for each ball $B=B(x_0,r)$ and if $u$ is harmonic on $2B$,
it holds for all
$x,y\in B(x_0,r)$ that
\begin{equation}\label{holder}
|u(x)-u(y)|\le C\frac{d^\gz(x,y)}{r^\gz}\fint_{2B}|u|\,d\mu;
\end{equation}
see for instance \cite{bm}.

We may assume that $r>1$ is large enough so that $B=B(x_0,r)$ contains a copy of
the  fundamental domain $X$. Then $g\cdot X$ are pairwise disjoint for different $g\in G$,
and $M\setminus (G\cdot X)$ is of measure zero. For simplicity of notions we assume that $u$ is harmonic on $8B$.

{\bf Claim 1:} For each $x\in X$, $g\in G$ such that $g\cdot x\in B(x_0,r)$, it holds that
\begin{equation}\label{lipschitz}
|u(x)-u(g\cdot x)|\le C\frac{\rho(g)}{r}\fint_{4B}|u|\,d\mu.
\end{equation}
{\em Proof of Claim 1.}
If $d(x,g\cdot x)\ge 2^{-16}r$, then \eqref{lipschitz} is obvious by \eqref{holder}. Consider now
the $d(x,g\cdot x)< 2^{-16}r$.
Let $k\in \mathbb{N}$ such that
$2^{k}< r\rho(g)^{-1}\le 2^{k+1}$ (remember $r>>1$ and $\rho(g)\ge 1$).
For each $j\in \{1,\cdots,k\}$, notice that $g^{2^j}\cdot x\in 2B$,  since $d(g^{2^j}\cdot x,x)<r$.
By using \eqref{holder} twice, we see that for $1\le j\le k$ it holds that
\begin{eqnarray*}
\left|[u(x)-u(g^{2^j}\cdot x)]-[u(g^{2^j}\cdot x)-u(g^{2^{j+1}}\cdot x)]\right|&&\le C\frac{d^\gz(x,g^{2^j}\cdot x)}{r^\gz}\fint_{3B}|u(x)-u(g^{2^j}\cdot x)|\,d\mu(x)\\
&&\le C\frac{2^{2j\gz}\rho(g)^{2\gz}}{r^{2\gz}}\fint_{4B}|u|\,d\mu.
\end{eqnarray*}
%Notice it holds the identity
%$$I-A=\sum_{j=0}^{k}2^{-j-1}[1-A^{2^j}]^2+2^{-k-1}(1-A^{2^{k+1}}),$$
Using the identity
$$u(x)-u(g\cdot x)= \sum_{j=0}^{k}2^{-j-1}[u(x)-2u(g^{2^j}\cdot x)+u(g^{2^{j+1}}\cdot x)]+2^{-k-1}[u(x)-u(g^{2^{k+1}}\cdot x)]$$
together with the above estimate and $2^{k}< r\rho(g)^{-1}\le 2^{k+1}$, we see that for $\gz<\beta<2\gz$ and $\beta\le 1$, it holds that
\begin{eqnarray*}
\left|u(x)-u(g\cdot x)]\right|&&\le \sum_{j=0}^{k}2^{-j-1}|u(x)-2u(g^{2^j}\cdot x)+u(g^{2^{j+1}}\cdot x)|+2^{-k-1}|u(x)-u(g^{2^{k+1}}\cdot x)|\\
&&\le \sum_{j=0}^{k}C2^{-j-1}\frac{2^{2j\gz}\rho(g)^{2\gz}}{r^{2\gz}}\fint_{4B}|u|\,d\mu
+C2^{-k-1}\frac{2^{k\gz}\rho(g)^{\gz}}{r^{\gz}}
\fint_{4B}|u|\,d\mu\\
&&\le  C\fint_{4B}|u|\,d\mu\left(\sum_{j=0}^{k}2^{-j-1+2j\gz-2k\gz}+2^{-k-1}\right)\\
&&\le C\fint_{4B}|u|\,d\mu\left(\sum_{j=0}^{k}2^{j(2\gz-1)-k(2\gz-\beta)} 2^{-k\beta}+2^{-k-1}\right)\\
&&\le C2^{-k\beta}\fint_{4B}|u|\,d\mu\le C\frac{\rho(g)^{\beta}}{r^{\beta}}
\fint_{4B}|u|\,d\mu.
\end{eqnarray*}
Repeating this argument sufficiently many times, we  conclude that \eqref{lipschitz} holds.

{\bf Claim 2:} There exists  a finite set $J\subset G$, with $e\in J$, such that
$$\sup_{x,y\in \cup_{g\in J}g\cdot X}|u(x)-u(y)|\le C\sup_{g\in J,x\in X}|u(x)-u(g\cdot x)|.$$
{\em Proof of Claim 2.}
Take $y_0\in X$ and fix $0<r_1$ such that $X\subset B(y_0,r_1)$. Then by Proposition  \ref{l2.3},
there exists $r_2>r_1$ such that
\begin{equation}\label{a3}
\sup_{x,y\in B(y_0,r_1)}|u(x)-u(y)|\le \frac 12\sup_{x,y\in B(y_0,r_2)}|u(x)-u(y)|.
\end{equation}
Let $J\subset G$ be the collection of $g\in G$ such that $g\cdot X\cap  B(y_0,r_2)\neq \emptyset$.
Then $B(y_0,r_2)\subset \cup_{g\in J}g\cdot X$ and $J$ only has finitely many elements.
For  $x,y\in \cup_{g\in J}g\cdot X$, take $\tilde x, \tilde y\in X$ and $g,h\in G$ such that $x=g\cdot\tilde x$ and $y=h\cdot \tilde y$.
Then by \eqref{a3} we obtain
\begin{eqnarray*}
|u(x)-u(y)|&&\le |u(g\cdot \tilde x)-u(\tilde x)|+|u(h\cdot \tilde y)-u(\tilde y)|+|u(\tilde x)-u(\tilde y)|\nonumber\\
&&\le 2\sup_{g\in J,x\in X}|u(x)-u(g\cdot x)|+\sup_{\tilde x,\tilde y\in X}|u(\tilde x)-u(\tilde y)|\\
&&\le 2\sup_{g\in J,x\in X}|u(x)-u(g\cdot x)|+\frac 12\sup_{\tilde x,\tilde y\in  B(y_0,r_2)}|u(\tilde x)-u(\tilde y)|,
\end{eqnarray*}
which, together with the fact $B(y_0,r_2)\subset \cup_{g\in J}g\cdot X$, implies that
$$\sup_{x,y\in \sup_{g\in J}g\cdot X}|u(x)-u(y)|\le C\sup_{g\in J,x\in X}|u(x)-u(g\cdot x)|.$$

We can now complete the proof.

Recall that  $y_0\in X\subset B$, $X\subset B(x_0,r_1)\subset B(x_0,r_2)\subset \cup_{g\in J}g\cdot X$.
Fix $r_3>r_2$ such that $\cup_{g\in J}g\cdot X\subset B(x_0,r_3)$. Notice $J\subset G$ is a fixed finite set.

Since $(RH_{\infty,\loc})$ holds, we may assume that $r>r_1+r_3$ is large enough, so that for each $h\in \{h\in G:\, h\cdot X \cap B\neq\emptyset\}$, $hg\cdot X\subset 2B$ for each $g\in J$.
By  $(RH_{\infty,\loc})$, together with the previous two claims, we obtain for each $h\in \{h\in G:\, h\cdot X \cap B\neq\emptyset\}$,
\begin{eqnarray*}
\||\nabla u|\|_{L^\infty(h\cdot X)}&&\le \||\nabla (u-u(h\cdot y_0))|\|_{L^\infty(B(h\cdot y_0,r_1))}
\le C\fint_{B(h\cdot y_0,r_2)}|u-u(h\cdot y_0)|\,d\mu\\
&&\le C \fint_{ \cup_{g\in J}hg\cdot X}|u-u(h\cdot y_0)|\,d\mu \le C\sup_{x,y\in \cup_{g\in J}gh\cdot X}|u(x)-u(y)|\\
&&\le C\sup_{g\in J,x\in  X}|u(h\cdot x)-u(hg\cdot x)|\le \frac{C}{r}\fint_{8B}|u|\,d\mu.
\end{eqnarray*}
This implies that
\begin{eqnarray*}
\||\nabla u|\|_{L^\infty(B)}&&\le \sup_{h\in G:h\cdot X \cap B\neq\emptyset}\||\nabla u|\|_{L^\infty(h\cdot X)}\le \frac{C}{r}\fint_{8B}|u|\,d\mu.
\end{eqnarray*}
A covering argument similar to that of {\bf Step 4} in the proof of Theorem \ref{harmonic-pnorm} then gives $(RH_\infty)$.
\end{proof}

\section*{Acknowledgments}
\addcontentsline{toc}{section}{Acknowledgments} \hskip\parindent
T. Coulhon and A. Sikora were partially supported by Australian Research Council Discovery grant DP130101302. This research was  undertaken while T. Coulhon was employed by the Australian National University.
R. Jiang was partially supported by
NNSF of China (11301029 \& 11671039), P. Koskela was partially supported
by the Academy of Finland via the Centre of Excellence in Analysis and Dynamics Research (project No. 307333).

\noindent Thierry Coulhon \\
PSL Research University,
75005 Paris, France.

\

\noindent Renjin Jiang \\
 Center for Applied Mathematics, Tianjin University, Tianjin 300072, China

\

\noindent Pekka Koskela\\
Department of Mathematics and Statistics, University of Jyv\"{a}skyl\"{a}, P.O. Box 35 (MaD), FI-40014, Finland.

\

\noindent Adam Sikora\\
Department of Mathematics, Macquarie University, NSW 2109,
Australia.

\

\noindent{\it E-mail addresses}:  \texttt{thierry.coulhon@univ-psl.fr}

\hspace{2.3cm} \texttt{rejiang@tju.edu.cn}

\hspace{2.3cm} \texttt{pekka.j.koskela@jyu.fi}

\hspace{2.3cm} \texttt{adam.sikora@mq.edu.au}

\end{document}